\documentclass{aims_arxiv}

\usepackage{multirow}
\usepackage{booktabs}
\usepackage{tabularx}
\usepackage{ltablex}
\usepackage{mathptmx}
\usepackage{amssymb,amsmath,amsfonts,latexsym} 
\usepackage{longtable} 
\usepackage{textcomp}
\usepackage{lscape}
\usepackage{mathtools}
\mathtoolsset{showonlyrefs=true}
\usepackage[perpage]{footmisc}
\usepackage[normalem]{ulem}
\usepackage{soul}
\def\typeofarticle{}

\def\ppages{xxx--xxx}
\def\DOI{10.3934/math.2024xxx}

\usepackage{calrsfs}
\DeclareMathAlphabet{\pazocal}{OMS}{zplm}{m}{n}
\let\mathcal\pazocal
\usepackage{dsfont}
\numberwithin{equation}{section}

\makeatletter
\renewcommand{\@biblabel}[1]{#1\hfill \hspace{-0.2cm}}
\makeatother
\allowdisplaybreaks

\newtheorem{theorem}{Theorem}{\bfseries}{\itshape}
\newtheorem{corollary}[theorem]{Corollary}{\bfseries}{\itshape}
\newtheorem{lemma}[theorem]{Lemma}{\bfseries}{\itshape}
\newtheorem{proposition}[theorem]{Proposition}{\bfseries}{\itshape}
\newtheorem{definition}[theorem]{Definition}{\bfseries}{\itshape}
\newtheorem{remark}[theorem]{Remark}{\bfseries}{\upshape}
\newtheorem{assumption}[theorem]{Assumption}{\bfseries}{\itshape}
\counterwithin{theorem}{section}

\newcommand{\eps}[1]{{#1}_{\varepsilon}}
\usepackage{tikz}
\numberwithin{table}{section}
\numberwithin{figure}{section}
\def \R{\mathbb{R}}               % \R reelle Zahlen
\def \N{\mathbb{N}}               % \N nat"urliche Z.
\def \G{\mathbb{G}}             % \Q rationale Z.
\def \1{{\bf 1}}                % \1 1-Vektor
\def \0{{\bf 0}}
\def\eps{\varepsilon}

\def \State{X}

\def\eps{\varepsilon}
\def \admiss{\mathcal{A}}

\def\eps{\varepsilon}

\def \Noise{\mathcal{B}}
\def \noise{b}

\renewcommand{\paragraph}[1]{{\smallskip\noindent\textbf{#1.}}}
\def \texteuro{EUR}
\def \ufix{\afix}
\def \texteuro{EUR}

\def \one{\mathds{1}}
\def \dimred{\ell}
\def \phx{\text{PHX} }
\def \phxk{\text{PHX}}
\def \phxs{\text{PHX}s }
\def \phxsk{\text{PHX}s}

\def \medium{M}
\def \fluid{F}
\def \outlet{O}
\def \inlet{I}
\def \bottom{B}
\def \interface{J}
\def \Qav{{\overline{Q}}{}}
\def \Qm{\Qav^\medium}
\def \Qf{\Qav^\fluid}

\def \Qout{\Qav^{\outlet}}
\def \Qin{Q^{\inlet}}
\def \QinC{\Qin_C}

\def \QavAppr{{Q}{}}
\def \QmAppr{\QavAppr^\medium}
\def \QfAppr{\QavAppr^\fluid}

\def \QoutAppr{\QavAppr^{\outlet}}

\def \Qg{Q^{G}}
\def \Qmm{Q^\medium}
\def \Qff{Q^\fluid}
\def \Dm{\Domainspace^\medium}
\def \Df{\Domainspace^\fluid}

\def \Din{\Domainspace^{I}}
\def \Dout{\Domainspace^{\outlet}}
\def \Dbottom{\Domainspace^\bottom}
\def \Dtop{\Domainspace^{T}}
\def \Dleft{\Domainspace^{L}}	
\def \Dright{\Domainspace^{R}}
\def \DInterface{\Domainspace^{\interface}}

\def \Cav{\omatrix}

\def \OutputM{\Cav^{\medium}}
\def \OutputF{\Cav^{\fluid}}
\def \OutputOut{\Cav^{\outlet}}

\def \Pamb{P_{\text{amb}}}
\newcommand{\Pambn}[1]{P_{\text{amb},#1}}
\def \Pin{P_{\,\text{in}}}
\def \Pout{P_{\text{out}}}
\def \HPspread{\Delta_{\,\text{HP}}}

\def \rhom{\rho^\medium}
\def \rhof{\rho^\fluid}
\def \kappam{\kappa^\medium}
\def \kappaf{\kappa^\fluid}
\def \cp{c_p}
\def \cpm{\cp^\medium}
\def \cpf{\cp^\fluid}
\def \cpw{\cp^W}
\def \am{d^\medium}
\def \af{d^\fluid}
\newcommand{\vconst}{\overline{v}_0}

\newcommand{\heattransfer}{\lambda^{\!G}}
\newcommand{\Celsius}{{\text{\textdegree C}}}

\newcommand{\mat}[1]{{#1}}

\newcommand{\normalvec}{\mathfrak{n}}   % \mathbf{n}

\newcommand{\dom}{\dagger}

\def \texteuro{EUR}
\def \State{X}
\def \Statespace{\mathcal{\State}}

\def \Domainspace{\mathcal{D}}  %{\Domainspace}
\def \omatrix{C}
\def \charg{+1}
\def \discharg{-1}

\def \fuel{+2}

\def \Charg{+1}
\def \Discharg{-1}
\def \Wait{0}
\def \Fuel{+2}
\def \Spill{-2}

\def \charge{+1}
\def \discharge{-1}
\def \waite{0}
\def \fuelf{+2}

\def \Resi{\widetilde{R}}
\def \ResD{R}
\def \resd{r}
\def \resi{\widetilde{r}}
\def \nO{n_{O}}

\def \Fu{\widetilde{F}}
\def \FuD{F}
\def \fud{f}

\def \RState{Y}
\def \Rstate{y}
\def \rstate{y^1,\ldots,y^\ell}

\def \Ptrans{\overline{P}}   %\mathbf{P}
\def \runCDis{\Psi}
\def \runCCont{\psi}
\def \termC{\phi_N}
\def \termD{\Phi_N}
\renewcommand{\top}{\prime}
\def \actionspace{\mathcal{K}}
\def \feasible{\mathcal{K}}
\def \amarkov{\widetilde \alpha}
\def \amarkovdiscrete{\widehat \alpha}
\def \admiss{\mathcal{A}}
\def \uprocess{u}
\def \aprocess{\alpha}
\def \ud{\alpha}
\def \afix{a}
\def \bbox{\mathcal{C}}
\def \price{\xi}
\def \drift{\eta}
\def \truncop{\mathcal{R}}

\newcommand{\mymarginpar}[1]{ \marginpar{{\tiny #1}}}
\renewcommand{\mymarginpar}[1]{}
\begin{document}
	
	\title{Cost-optimal Management of a Residential Heating System With a Geothermal Energy Storage Under Uncertainty}
	
	\author{%
		Paul Honore Takam \affil{1}
		Ralf Wunderlich \affil{1}
		%  and
		%  First name Last name\affil{1,}\corrauth
	}
	
	% \shortauthors is used in copyright information in the end of the paper
	\shortauthors{the Author(s)}
	
	\address{%
		\addr{\affilnum{1}}{Brandenburg University of Technology Cottbus-Senftenberg, Institute of Mathematics, P.O. Box 101344, 03013 Cottbus, Germany; }}
	%  \addr{\affilnum{2}}{Affiliation}}

% corresponding author
\corraddr{ralf.wunderlich@b-tu.de
	\\% Fax: +1-111-111-1111.
}

\begin{abstract}
	In this paper, we consider a residential heating system with renewable and non-renewable heat generation and different consumption units and investigate a stochastic optimal control problem for its cost-optimal management. As a special feature, the heating system is equipped with a geothermal storage that enables the intertemporal transfer of thermal energy by storing surplus heat for later use. In addition to the numerous technical challenges, economic issues such as cost-optimal control also play a central role in the design and operation of such systems. The latter leads to challenging mathematical optimization problems, as the response of the storage to charging and discharging decisions depends on the spatial temperature distribution in the storage. We take into account uncertainties regarding randomly fluctuating heat generation from renewable energies and the environmental conditions that determine heat demand and supply. The dynamics of the multidimensional controlled state processes is governed by a partial, a random ordinary and two stochastic differential equations. We first apply a spatial discretization to the partial differential equation and use model reduction techniques to reduce the dimension of the associated system of ordinary differential equations. Finally, a time-discretization leads to a Markov decision process for which we apply a state discretization to determine approximations of the cost-optimal control and the associated value function.
\end{abstract}

\keywords{
	{Stochastic optimal control, Markov decision process, Geothermal energy storage,  Model order reduction,  Residential heating system,  Numerical simulation}
	\\[1ex]
}

\maketitle

%\setcounter{tocdepth}{3}
%\tableofcontents
%\newpage
\section{Introduction}

Climate change and energy dependency require urgent measures for the improvement of energy efficiency in all areas. District heating and cooling systems play an important role for  increasing energy efficiency in buildings and for including renewable energy sources. In addition to numerous technical issues, economic questions such as the cost-optimial control and management of such heating systems also play a central role. The latter  leads to challenging mathematical optimization problems which require advanced and sophisticated solution techniques. One of these problems  arising in the optimal management of a residential heating system with access to an additional geothermal energy storage (GES),  as depicted in Fig.~\ref{S-model}, is addressed in this paper. 

Thermal storage facilities and in particular geothermal energy storage help to mitigate and to manage temporal fluctuations of heat supply and demand for heating and cooling systems of single buildings as well as for district heating systems. They make it possible to store heat in the form of thermal energy and to use it again hours, days, weeks or months later. This is attractive  for space heating, domestic or process hot water production, or generating electricity, since it improves the efficiency and saves costs.   Geothermal storage are becoming increasingly important and are very attractive for heating systems in residential buildings, as they are very economical to build and maintain. Furthermore, such storage can be integrated both in new buildings and in renovations. 
GES can also be used  to cushion peaks in the electricity grid by converting electrical  energy into thermal  energy (power to heat). Pooling several GESs within the framework of a virtual power plant gives the necessary capacity which allows to participate in the balancing energy market. 

\begin{figure}[h!]
	\centering 
	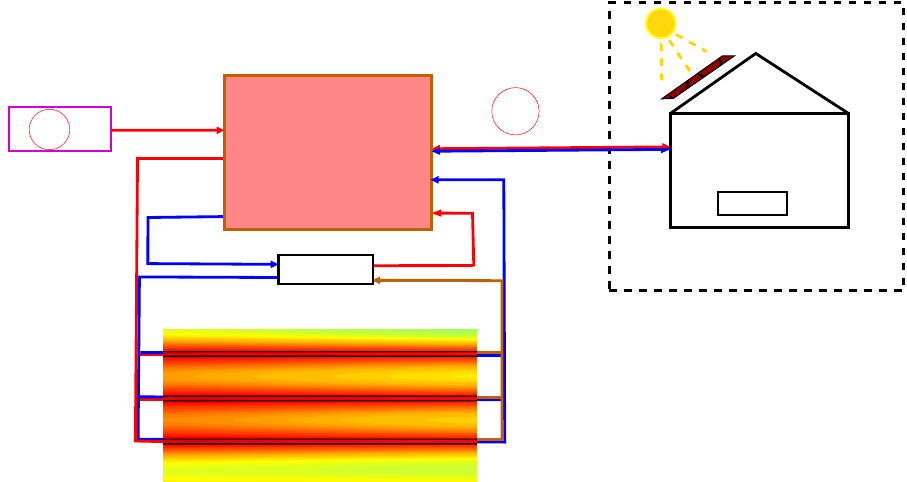
	\caption{Simplified model of a residential heating system. Sect. \ref{sec:Heating-S} introduces the notation and gives explanation.}
	\label{S-model}
\end{figure}

% see Fig.~\ref{etank_longtermsimu}
In the GES investigated in this work, a certain volume under or adjacent to a building is filled with soil and insulated from the surrounding ground.  The thermal energy is stored by increasing the temperature of the soil inside the storage tank. It is charged and discharged through pipe heat exchangers (\phx) filled with a liquid such as water.
The fluid carrying the thermal energy is moved using pumps.  The \phxs are connected to an internal energy storage (IES) which is \st{ as} a water tank. Contrary to the GES, the IES  has a smaller capacity and is designed to buffer imbalances of heat supply and demand in the heating system on short time scales.   In addition, the GES stores heat at a lower temperature level, so heat pumps must be used for heat transfer from the GES to the IES.

A special feature of the GES considered in this paper is that,  it is not insulated at the bottom, allowing thermal energy to flow into deeper layers. %, as can be seen in Fig.~\ref{etank_longtermsimu}. 
This is a natural extension of the storage capacity, as this heat can be recovered to some extent if the GES is sufficiently discharged (cooled) and a heat flow is induced back into the storage. Of course, there are inevitable diffusion losses to the environment, but due to the open architecture, the GES can take advantage of higher temperatures in deeper layers of the ground and acts as a production unit similar to  a downhole heat exchanger. 

\smallskip	
In this work, we consider a mathematical model of a residential heating system consisting of
\begin{itemize}
	\item a local renewable  heat production unit such as a solar thermal collector,
	\item a non-renewable heat  production unit such as a fuel-fired boiler,
	\item several heat  consumption units in the building, 
	\item an internal energy storage, that   serves as a short-term buffer storage and is typically realized as a water tank,
	\item and, as a special feature, a geothermal storage with large capacity, which can store energy for longer periods.
\end{itemize}
In this model, we do not describe all the technical details of heat transfer to and from the consumption units of the building and the contribution of the solar collector.  Instead, we only consider the aggregated residual demand describing the imbalance  of the intermittent demand for thermal energy in the building and the supply of thermal energy from local production of solar collector.   This residual demand generally does not only  fluctuate over time, but its future values cannot be predicted with certainty, as they depend on the weather-dependent heat production of the solar collector and the demand behavior of consumers. Therefore, we model it using a stochastic process. Further, depending on the size of the heating system and the selected tariff for the fuel purchase, this also applies to the future fuel price, which may depend on the situation on the energy market. Therefore, like  residual demand, we also model the fuel price as a stochastic process. 

From an economic point of view, the aim is to operate the heating system in such a way that the aggregated costs of purchasing fuel and the electricity costs for operating the system, in particular for the transfer of heat between the IES and GES using heat pumps, are minimized over a certain period of time. These minimum costs must be known in order to make decisions about investing in certain heating systems or  the size of  an additional external geothermal storage. The optimal operation of the heating system, including the interaction of the two storage systems IES and GES, is a decision-making problem under uncertainty, since the control decisions must be made in the face of uncertainty about future energy prices and future residual demand. We will formulate it mathematically as a stochastic optimal control problem. 

In contrast to the typical situation in continuous-time stochastic control theory, where the dynamics of the controlled state is determined by a system of stochastic and ordinary differential equations (SDEs and ODEs), one of the state components is described by a heat equation with convection, that is a parabolic partial differential equation (PDE). This is an non-standard feature of the control problem that requires special attention. 
The reason for the inclusion of the heat equation is that the GES response to charging and discharging operations depends on the spatial temperature distribution in the storage medium, in particular in the vicinity of the \phxk. The dynamics of this distribution is described by the heat equation. 
Please note that when charging the GES, the liquid in the \phx reaches the GES inlet at a high temperature from the IES. On its way through the GES, the heat of the  fluid is transferred to the colder storage medium and  it returns to the IES at a lower temperature.
A heat pump is used for transferring heat from the GES to the IES. It consists of two circuits. In the first circuit, heat is extracted from the GES by sending the  fluid through the \phx at a low temperature so that it can absorb heat from the surrounding medium and return to the heat pump at a higher temperature. 
There, heat is extracted from the  fluid in the first circuit  and transferred to the  fluid  in the second circuit. In addition, the temperature is increased by using additional electrical energy so that it can be sent to the IES at a level suitable for the heating system.

Due to this mode of operation of the charging and discharging processes, heat transfer from the IES to the GES is delayed when the \phx environment is saturated and at a high temperature. In this case, the fluid in the \phx exits the GES at an almost identical temperature to its initial entry, with only a negligible amount of heat transferred to the storage medium. In order to save costs for operating the pumps, it is advantageous in such a case to stop the charging process first and instead wait until the heat in the immediate vicinity of the \phx has spread to colder regions within the GES. Vice versa, the heat transfer from the GES to the IES becomes inefficient if there is a non-homogeneous temperature distribution with low temperatures in the \phx environment. Then, the fluid in the \phx can only absorb a small amount of heat from the storage medium.

%Given the considerable influence of the aforementioned effects on the operating costs of the heating system, it is necessary to model the dynamics of the spatial temperature distribution in the GES and take them into account accordingly. This leads to a heat equation with convection term, which describes the heat transfer in the moving fluid, and enables the calculation of the response of the GES to charging and discharging processes on time scales from a few minutes to a few days.

\paragraph{Literature review on  geothermal energy storage} 
We refer to Dincer and Rosen \cite{dincer2021thermal} and Regnier et al.~\cite{Regnier_et_al_2022} for an overview on thermal energy storage.
The work of Guelpa and  Verda \cite{guelpa2019thermal} and  Kitapbayev et al.~\cite{KITAPBAYEV2015823} showed that 
thermal energy storage can significantly increase both the flexibility and the performance of district energy systems and enhancing the integration of intermittent renewable energies into thermal networks. The article  Major et al. \cite{major2018numerical} considered heat storage capabilities of deep sedimentary reservoirs. Here, the governing heat and flow equations are solved using finite element methods.  Further contributions on the numerical simulation of such storage are provided in \cite{bazri2022thermal,soltani2019comprehensive}.

The GES examined in this article is a relatively new and specialized  technology that has been developed and used only in the last 15 years. To our knowledge, there are only a few references such as \cite{bahr2022efficient,bahr2017fast,TakamWunderlichPamen2023,TakamWunderlich_redu2024,takam2025energies,takam_PhD_2023} which deal with the mathematical modeling and numerical simulation of such storage facilities.
However, the heat transfer and exchange processes between the heat exchanger and the surrounding soil also  play a crucial role  not only in this work, but also for  so-called ground source heat	pumps. 	
They extract heat from the ground and feed it into heating systems. However, a storage function is of little or no importance here.	 There is a large amount of specialist literature on modeling and simulation for these systems, which are divided into horizontal and vertical geothermal heat exchangers. An overview can be found in Zayed et al.~\cite{zayed2023recent}. \

\paragraph{Literature review on optimal  energy storage management}
Energy storage  can be used to create profit by trading in the energy market and taking advantage of the fluctuating energy price by applying an active storage management, see   B\"auerle and Riess \cite{bauerle2016gas},  Chen and Forsyth \cite{chen2008semi,chen2010implications}, Ware  \cite{ware2013accurate},   Shardin and Wunderlich \cite{shardin2017partially}. The basic principle is to  buy and store energy  when prices are low, and to release and sell energy when market prices are high, and to keep the intermediate storage and operating costs, and unavoidable dissipative losses under control.

\paragraph{Literature review on  optimal management of heating systems} 
The residential heating system studied in this work can be considered as a thermal microgrid. This is an autonomous energy system with local thermal energy generation and storage units used to meet the heating demand. There is an extensive literature on the topic of optimal management of electrical microgrids. Some of these articles also include thermal units in the microgrid and investigate  combined heat and electricity systems.	
In Testi et al. \cite{testi2020stochastic}, an optimal integration of electrically driven heat pumps within a hybrid distributed energy system is investigated. There, the authors proposed a  multi-objective stochastic optimization methodology to evaluate the integrated optimal sizing and operation of the energy systems under uncertainties in climate, space occupancy, energy loads, and fuel costs. In Kuang et al. \cite{kuang2019stochastic}, a stochastic dynamic solution for off-design operation optimization of combined heating and power systems with energy storage is considered.  In Gu et al. \cite{gu2014modeling}, a review on optimal energy management of combined cooling, heating and power microgrid is considered. Further contributions on combined heating, cooling, and power system are given in  \cite{ehsan2019scenario,zhong2021distributed} and references therein.

\paragraph{Literature review on stochastic optimal control}
As already mentioned, the cost-optimal management of a heating system for residential buildings under uncertainty is treated mathematically as a stochastic optimal control problem. A large body of literature on this theory
investigates  \textit{dynamic programming}  solution techniques. In the continuous-time setting, where diffusion or jump-diffusion processes form the controlled state process, this leads to the Hamilton-Jacobi-Bellman equation as a necessary optimality condition, see Fleming and Soner \cite{fleming2006control}, Pham \cite{pham2009c}, and Oksendal and Sulem \cite{oksendal2019stochastic}. These nonlinear PDEs can usually only be solved with \textit{numerical methods}, as in Shardin and Wunderlich \cite{shardin2017partially}, Chen and Forsyth \cite{chen2010implications}.	

For discrete-time models, the theory of Markov decision processes (MDPs) offers a solutions based on backward recursion. We refer to Bäuerle and Rieder \cite{bauerle2011markov}, and Puterman \cite{puterman2014markov}, Hern{\'{a}}ndez-Lerma and Lasserre \cite{HernandezLerma1996} and Powell \cite{powell2007approximate}. Such MDPs also result from  time discretisation of continuous-time control problems. 

\paragraph{Our contribution}  This article presents a comprehensive mathematical approach for modelling the operation of a heating system equipped with a GES in residential buildings, which makes it possible to formulate and solve optimization problems that arise in the cost-optimal management of such systems.	
It explicitly takes into account the stochasticity of the imbalance of heat supply and demand due to intermittent heat production of the solar collector and the  fluctuating  demand behavior of consumers, as well as fluctuating market prices for fuel. These variables are modeled by suitable stochastic processes. Further, the model  captures the dynamics of the spatial temperature distribution in the GES which is described by a  heat equation with convection term. This enables a precise description of the interaction between the two storage units, IES and GES, and the response of the GES to charging and discharging processes, which is necessary to control the heating system.  

The cost-optimal management of the heating system based on the continuous-time setting leads to a non-standard stochastic optimal control problem in which the dynamics of the continuous-time state process is described by a system of two SDEs, one ODE and one PDE. Therefore, we first approximate the dynamics of the spatial GES temperature distribution determined by the PDE by a low-dimensional system of ODEs combining semi-discretization of  the PDE and  model order reduction techniques. In a second step, we perform a time discretization that leads to state dynamics described by a system of random recursions.  While the controls between two discrete points are assumed to be constant, the dynamics of the state variables in each of the periods are still analyzed in continuous time to avoid unnecessary time discretization errors.

This approach enables the cost-optimal management problem to be treated as an MDP.
Since the solution of MDP with dynamic programming methods is confronted with the curse of dimensionality due to the high dimension of the state space, we approximate the MDP by another MDP for a discretized state space.  This finally enables an efficient computation of the approximation of the value function and the optimal decision rules for which we provide numerical results.

\paragraph{Paper organization}   Sect.~\ref{sec:Heating-S} provides a description  of the residential heating system. In Sect.~\ref{Control-System}, the state and control variables are introduced. Sect.~\ref{sec:StateDynamics} is devoted to the continuous-time dynamics of the controlled state process which is goverened by a system of ODEs, SDEs and a PDE.  In Sect.~\ref{State-approx}, we consider the approximation of the dynamics of the spatial temperature distribution in the GES by a low-dimensional system of ODEs  In Sect.~\ref{sec:Discrite-OP}, we formulate the stochastic optimal control problem as an MDP. An approximate method to solve the MDP which is based on a state discretization is studied in Sect.~\ref{sec:StateDiscretization}.  In Sect.~\ref{OPt-Num-rsult}, we present  numerical results which show  properties of the value function and the optimal decision rules. The paper concludes in Sect.~\ref{sec:summary} with a short summary and outlook. An appendix provides a list of notations and collects proofs and technical results that have been removed from the main text.

\section{Residential Heating System}
\label{sec:Heating-S}
A residential heating system is designed to provide thermal energy for heating and hot water supply of a building. Here, the notion ``building''  is used for single family homes, office buildings, small companies  or even small districts with a couple of buildings sharing a common heat and water supply. 

\paragraph{ Residual demand} The building is equipped with some local production units for thermal energy such as solar collectors or other units using renewable energies. The supply of theses units usually does not meet exactly the demand of thermal energy due to the immanent temporal fluctuations and seasonality effects in both supply and demand.  We call that imbalance the residual demand, and model it by a stochastic process $\Resi=(\Resi(t))_{t\in[0,T]}$,  which we decompose as $\Resi(t)=\mu_R(t) + \ResD(t)$. Here, $\mu_R$ denotes the deterministic seasonality, a regular, repeating pattern  known from historical data, while the stochastic process $\ResD$ is the deseaonalized  component describing the uncertain future deviations from the seasonality pattern. More  details will be given below in Subsect.~\ref{ResidualD1}.

\paragraph{Internal energy storage}
If the demand exceeds the supply from local production then the excess demand $\Resi(t)>0$ is satisfied by an internal buffer storage. It also stores overproduction of thermal energy  which is not needed to satisfy the demand. In that case $\Resi(t)$ is negative. This internal storage is designed to balance supply and demand on shorter time scales of some hours or only  few days.  However, the capacity is not sufficient to serve as buffer for seasonal fluctuations on time scales of weeks and months. 

As often observed in reality, we assume the internal buffer storage to be a stratified hot water storage tank. The warmest storage layer is at the top and below there are colder layers through natural layering. For simplicity, we assume that the storage can keep a constant  temperature  $\overline p$ on the top and also a constant   temperature  $\underline p<\overline p$  at the bottom. We do not model the vertical temperature profile in the storage  but consider only the storage's average temperature which is denoted by $P(t)$ for time $t\in[0,T]$. 

\paragraph{Fuel-fired boiler}
Due to its limited capacity, the internal storage cannot provide the necessary heat supply for a permanent or very strong unsatisfied demand of the building. Therefore, the system is equipped with another production unit, which is able to generate enough heat also on short time scales and to prevent the internal storage to become completely empty. This unit may fire fossil fuels (gas, oil, coal), convert  electricity to heat using an immersion heater or obtain additional heat from a  district heating system as in \cite{some2025prosumers}.  In all cases, this heat production comes with additional costs arising from the consumption of fuel, or electricity, or other respective heat sources.  To be more specific, we call this additional production unit a ``fuel-fired boiler‘’ and the price of the respective heat source a ``fuel price‘’, keeping alternative energy sources in mind. We denote this  fuel price  at time $t$ by $\Fu(t)$. Uncertainties about the future prices will be captured by modeling $\Fu$ as a stochastic process,  which we decompose, as in the case of the residual demand, into $\Fu=\mu_F+\FuD$ with the deterministic seasonality pattern $\mu_F$ and deseaonalized stochastic component $\FuD$. For more details we refer to Subsect.~\ref{ResidualD1}.

\paragraph{Geothermal energy storage}
In periods of permanent or strong overproduction, the internal storage may reach its capacity limits and can no longer accommodate more leftover heat from the local production. In order to enable  a later usage of that leftover heat, the  heating system  is equipped with an additional external thermal storage, which in this work is a geothermal storage.	Compared to the internal storage, its capacity is much larger, but it is also characterized by a lower temperature level. Therefore, heat pumps are required for transferring heat from the geothermal to the internal storage.  Further, the transfer of thermal energy to and from the external storage depends on the often slow operation of heat exchangers.  
The geothermal storage is characterized by a nonhomogeneous spatial and temporal temperature distribution. We will work with a spatially two-dimensional model and denote by  $Q(t,x,y)$ the GES temperature at time $t$ and the point  $(x,y)$. More details are provided in  Subsec.~\ref{Geothermal_S}. 

If the internal storage is already (almost) fully charged and there is still overproduction of thermal energy in the building, then heat can be transferred from the internal to the external storage.  This is achieved by sending a fluid of high temperature from  the internal storage tank through the heat exchanger pipes of the geothermal storage tank.  The fluid arrives at the (possibly multiple) inlets of the pipe heat exchangers (PHXs) with the inlet  temperature denoted by $\Qin(t)$. After passing through the geothermal storage, the fluid will leave the \phxs with a lower temperature. The average temperature at the (possibly multiple) outlets is denoted by $\Qout(t)$, see below Eq.~\eqref{Average-Outlet} for details.  This is also the temperature at which the fluid returns to the internal storage. Since the efficiency of charging the geothermal storage is improved by increasing the inlet temperature $\Qin(t)$, we assume in this work that during charging, this temperature is equal to the maximum available temperature provided by the system,  which we denote by the constant $\Qin_C$. 

On the other hand, if the internal storage  is (almost) empty and there is still unsatisfied demand in the building, then instead of producing heat from firing fuel,  thermal energy can be also be transferred from the geothermal storage to the internal storage. For that process, the system uses a \textit{heat pump }for raising the temperature of the fluid arriving from the outlet of the geothermal storage to a higher level $\Pin>\underline p$. 	
Here $\Pin$ is a pre-specified temperature at which the fluid coming from the heat pump arrives at the internal storage. For simplicity, we assume that  $\Pin$ is constant.  

The heat pump connects two circuits in which moving fluids carry heat. A first circuit
is connected to the geothermal storage. The fluid arrives from the storage's outlet at the inlet of the heat pump with temperature $\Qout$. 
The heat pump withdraws heat from the fluid so that it leaves the pump at time $t$ with  the temperature $\Qout(t)-\Delta T_{HP}$, and returns to the inlet of the geothermal storage.
The quantity $\Delta T_{HP}>0$ is called heat pump spread and assumed to be a given constant.  The thermal energy extracted from the fluid of the first circuit is transferred to the fluid in the second circuit. The latter connects the heat pump with the internal storage. At the pump's inlet arrives cold water of temperature $\Pout$, which is raised  to the temperature $\Pin>\Pout$,   that is  suitable for the heating system, using the extracted heat in the first circuit and additional electrical energy. At this temperature, the fluid returns to the internal storage. 

\section{Control System}
\label{Control-System}
In this section we setup the model for  continuous-time $t\in [0,T]$, where $T>0$ is a finite time horizon.

\paragraph{State}
The {state} of the control system at $t\in [0,T]$ is given by the following   four quantities. 
First, $P$  the average temperature in the IES and $ Q$, the spatio-temporal temperature distribution in the geothermal  storage are the controlled or endogeneous state components. Further,  the  deseasonalized residual demand $R$ and 
the  deseasonalized  fuel price $F$ are two uncontrolled or exogeneous state components which will be modeled  as stochastic processes.	
We define the state process by $X=(R,F,P,Q)^\top$ where $x^\top$ denotes the transpose of the vector $x$. Further details and the description of the state dynamics are given below in Sect.~\ref{sec:StateDynamics}.  

\paragraph{Control} The operator or decision maker for the considered residential heating system has several control actions at his disposal which relate to the charging and discharging operations of the two storage facilities.
The GES is charged and discharged via the \phxs connected to the IES and filled with some fluid. We assume that during such  charging and discharging operations the fluid in the \phx is always moving with the constant velocity  $\vconst>0$. The GES is charged by transferring heat from the IES, and vice versa it is discharged  by transferring heat to the IES. 
The IES can also be charged by firing fuel. 
We assume that simultaneous discharging or charging the GES and firing fuel is not allowed, since it will not generate minimum cost and will be therefore not optimal.
If both the IES and GES are full, and there is an overproduction of heat by the solar collector, that is the residual demand is negative, then the operator no longer can store this energy but  has to discard it. We will  call this control action  \emph{over-spilling} and assume that it is associated with zero costs and that during over-spilling  the temperature of the IES will remain constant at the maximum temperature $\overline{p}$.

We formalize this by introducing the \textit{set of  feasible actions} 
\begin{align}
	\actionspace=\{\Spill, \Discharg,\Wait,\Charg,\Fuel\}
	\label{feasible A}
\end{align}
from which the decision maker can select at any time $t\in[0,T]$ an control action and form the  
\textit{continuous-time control process}    $u:[0,T]\to \actionspace$.
The interpretation of the actions in $\actionspace$ is as follows. The action
$\Charg$ indicates charging the IES   by transferring heat from the  GES,  whereas  $\Discharg$ denotes discharging the IES  by transferring heat to the  GES. In both cases the fuel-fired boiler is off.
The action $\Fuel$  indicates that heat is generated by firing fuel to charge the IES, with the \phx pumps switched off. The action $\Wait$ means wait or suspend, whereby both the \phx pumps and the fuel-fired boiler are switched off. Finally, $\Spill$ denotes over-spilling.

\section{State Dynamics}
\label{sec:StateDynamics}
In this section, we describe the dynamics of each individual component of the state process $\State$. We start with the residual demand $R$ and the fuel price $F$ which we model as stochastic processes satisfying SDEs. Then we derive a PDE for the spatio-temporal temperature distribution in the geothermal storage given by  $Q$.  Finally, we consider the average temperature in the internal storage $P$ which is governed by an ODE.  

\subsection{Residual Demand and Fuel Price} 
\label{ResidualD1}
%\paragraph{Uncertainties}
The two exogenous states describing the deseasonalized residual demand $R$  and the the deseasonalized fuel price $F$
are modeled as stochastic processes defined on a filtered probability space  $(\Omega,\mathcal{G},\mathbb{G},\mathbb{P})$.
In particular, that space supports a two-dimensional standard Wiener  process $W=(W_R,W_F)^\top$ on $[0,T]$ driving the SDEs \eqref{Residual_deseason} for $R$ and  \eqref{Fuel_F} for  $F$ given below.
The filtration $\mathbb{G}$  is assumed to be generated by $W$, that is, $\G=\G^W=(\mathcal{G}^W(t))_{t\in [0, T]}$  with the $\sigma$-algebras $\mathcal{G}^W(t)=\sigma\{W(s), s\le t\}$, augmented by the $\mathbb{P}$-null sets, so that, $\mathbb{G}$ satisfies the usual assumptions   of right-continuity and completeness.

\paragraph{Residual demand}
Recall, the residual demand $\Resi(t)$ describes the imbalance, i.e., the difference between  the demand for thermal energy in the building and supply of thermal energy  provided by the local production units  at time $t \in [0,T]$, and is  measured in $kW$. It can take positive as well as negative values in $\R$.  The residual demand is positive if demand exceeds supply, and negative  if supply exceeds demand (overproduction). For the formulation of the stochastic optimal control problem it will be useful to  decompose  the residual demand on $[0,T]$ as follows
\begin{align}
	\Resi(t)=\mu_R(t)+\ResD(t).
	\label{Residual_D}
\end{align}  
Here, $\widetilde{\mu}_R: [0,T] \rightarrow \R $ is a bounded deterministic function  describing the residual demand's seasonality, and $\ResD=\Resi-\mu_R$ is the deseasonalized residual demand, which we model as an   Ornstein-Uhlenbeck  process which is mean-reverting to zero, and defined by the SDE
\begin{align}
	d\ResD(t)=-\beta_R\ResD(t)dt+\sigma_R(t) dW_R(t), \qquad \ResD(0)=\resd_0 \in \R,
	\label{Residual_deseason}
\end{align}
where $W_R$ is a standard Wiener process. The parameter $\beta_R>0$ is called mean-reversion speed and  $\sigma_R:[0,T] \rightarrow [\underline{\sigma}_R,\overline{\sigma}_R]$  is a deterministic and bounded  volatility function  with some constants $0<\underline{\sigma}_R\le \overline{\sigma}_R$.  The use of a time-dependent volatility function $\sigma_R$ makes it possible to link the stochastic variability of $\Resi$ to certain seasonal patterns.  
Simple but already meaningful examples of the above-mentioned seasonality function have the form  
\begin{align}
	\label{seasonality}
	\mu_R(t)&=
	\mu_R^0   + \sum\limits_{i=1}^M \mu_R^i\cos \frac{2\pi (t-t^i_R) }{\delta^i_R},
\end{align}
where the constant $\mu_R^0 \in \R$ denotes the long-term mean, and  $\mu_R^i>0,~ i=1,\ldots,M$, 
the amplitude of the seasonality  component, $\delta_R^i$ the length of the seasonal period
and  $t_R^i$ some time shift parameter for the $i$-th seasonality component (representing the time of the seasonal peak of the residual demand) 
and $M$ is the number of components. 

\begin{figure}[h!]
	\centering 	
	\includegraphics[width=0.49\textwidth,height=0.32\linewidth]{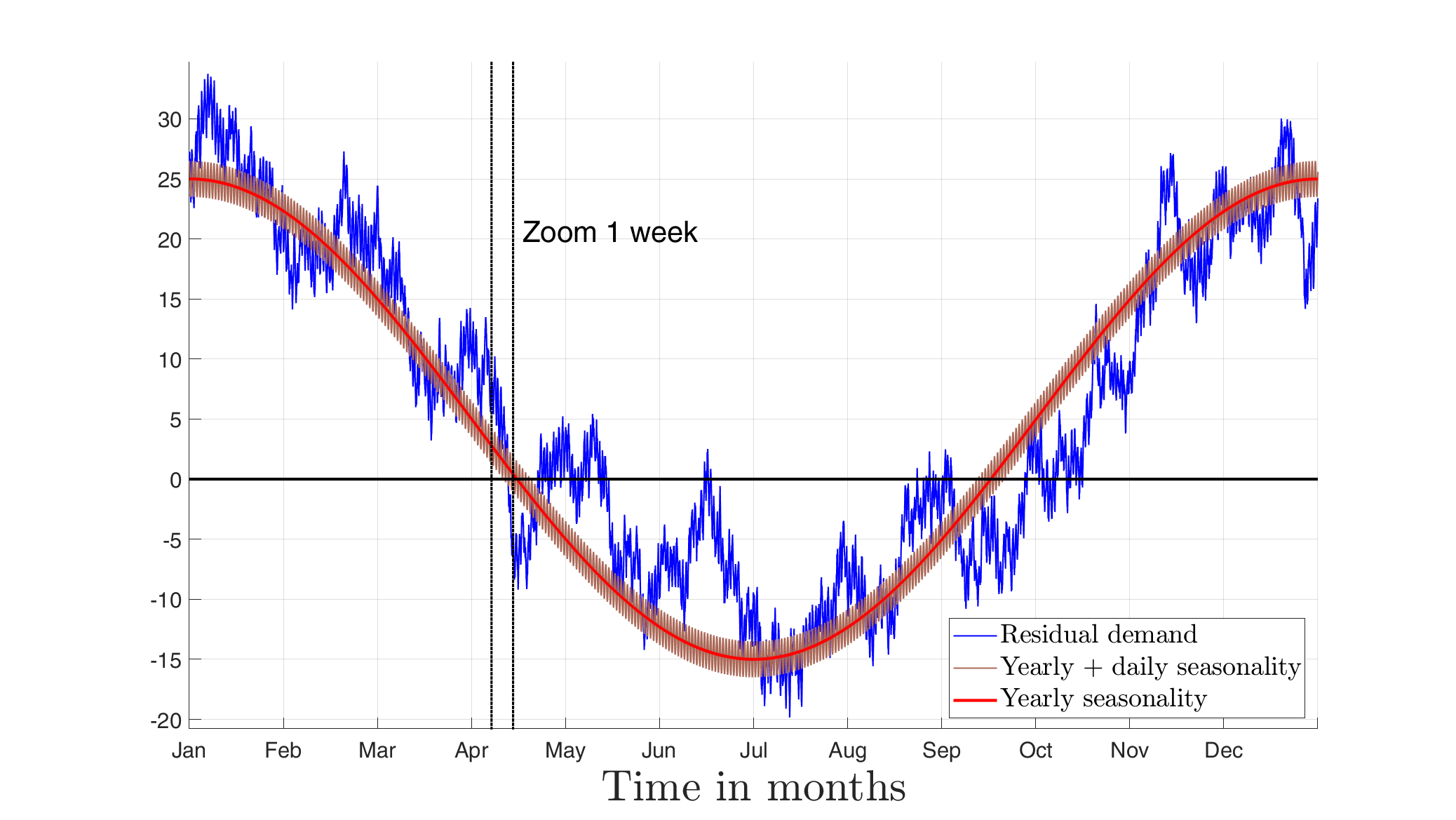}
	\hspace*{-0.06\textwidth}
	\includegraphics[width=0.49\textwidth,height=0.32\linewidth]{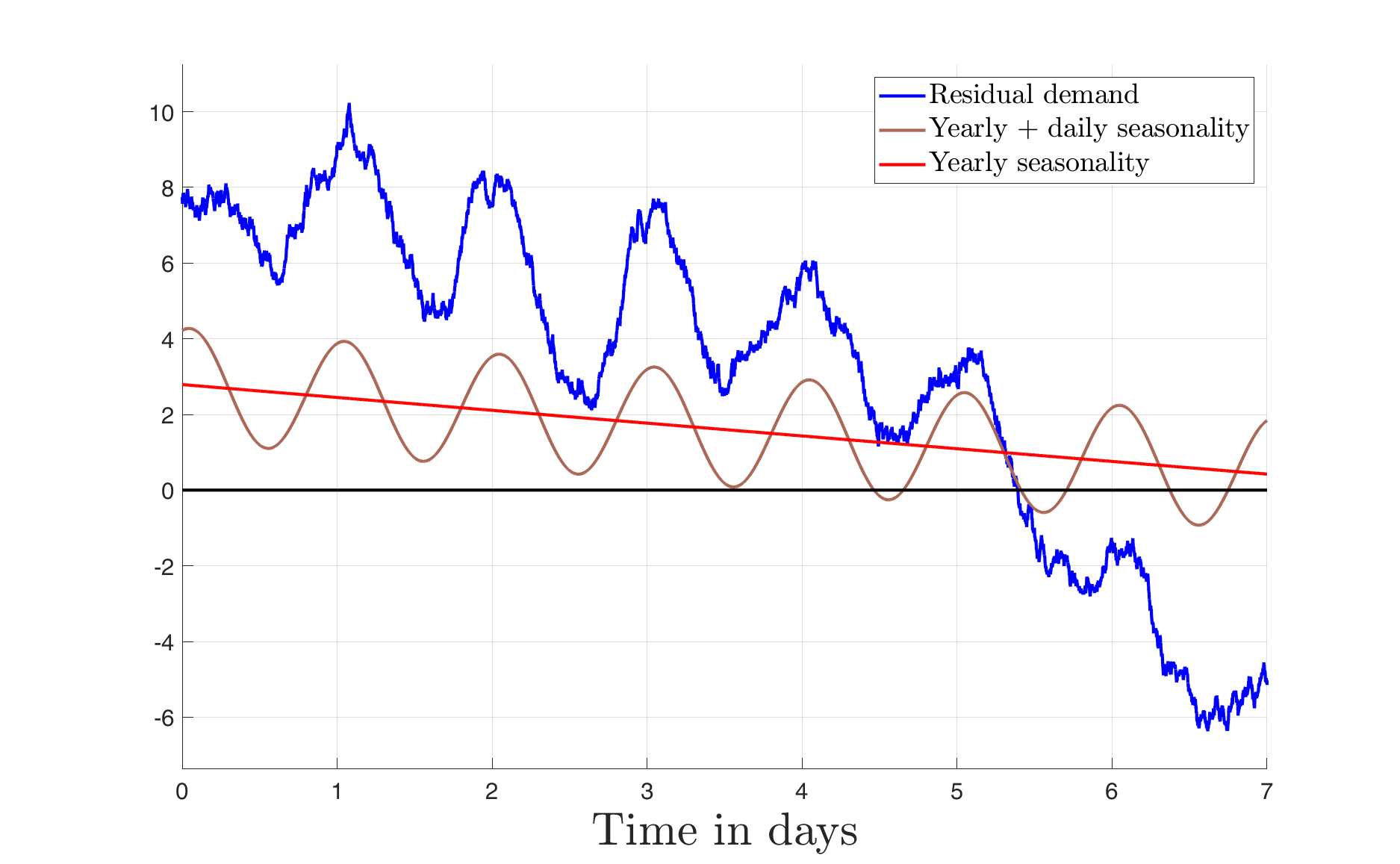}
	{\caption[Residual demand. Left: Over a period of one year. Right: One week zoom in.]{Residual demand $\Resi$ over a period of one year (left) and   a zoom on a week in mid-April (right) with parameters $\beta=0.5$, $\sigma^R=2$,  $\mu_R^0=5$, $\mu_R^2=20$, $\mu_R^2=1.5$, $\delta_R^1=365$ days, and $\delta_R^2=1$ day. Blue solid line for residual demand $ \Resi$, red and brown solid lines show the yearly component and the seasonality function $\mu_R$ with combined yearly and daily components. The long-term mean level  $\mu_R^0$ is shown as a black solid line.		
	}}
	\label{fig:resdemand}
\end{figure}

\paragraph{Fuel price}  Similarly to \eqref{Residual_D}, the fuel price $\Fu $ is a stochastic  process  which can be decomposed as $\Fu(t)=\mu_F(t)+\FuD(t)$  	
where,  $\mu_F$ is a bounded deterministic seasonality function. Typical examples are functions  as in \eqref{seasonality} but with only a yearly component and no daily component. The deseasonalized fuel price 	 $\FuD$ is modeled by  an Ornstein-Uhlenbeck  process which is mean reverting to zero. It captures the random effects in the fuel price, and  is defined by the SDE 
\begin{align}
	\label{Fuel_F}		
	d\FuD(t)&=-\beta_F\FuD(t)dt+\sigma_F(t)dW_F(t), \quad \FuD(0)=\fud_0 \in \R.
\end{align}

Here,  $\beta_F>0$ deontes the mean-reversion speed, and 	$\sigma_F:[0,T] \rightarrow  [\underline{\sigma}_F,\overline{\sigma}_F]$ is  a deterministic and bounded  volatility function  with some constants $0<\underline{\sigma}_F\le \overline{\sigma}_F$.
For further simplification, we can assume that the fuel price is a known deterministic function of time  or even constant. Then  the fuel price can be removed from the state variables of the control problem which reduces the dimension of the state space  by one.

\subsection{ Spatio-temporal Temperature Distribution in the Geothermal Storage}
\label{Geothermal_S}
The dynamics of the  spatial temperature distribution in a geothermal storage can be described mathematically by a linear heat equation with  convection term and appropriate boundary and interface conditions.    

\subsubsection{Two-dimensional Model}
\label{model2D}
The setting is based  on \cite[Sect.~2]{TakamWunderlichPamen2023}. For self-containedness  and the convenience of the reader,  we recall in this subsection the description of the model.
The GES area is assumed to be a cuboid for which we consider a two-dimensional rectangular cross-section. As mentioned above 
$Q=Q(t,x,y)$ is the temperature at time $t \in [0,T]$ in the point $(x,y)\in \Domainspace=(0,l_x) \times (0,l_y)$ where $l_x,l_y$ are the width and height of the storage. 
Fig.~\ref{bound_cond1} depicts the domain $\Domainspace$ and its boundary $\partial \Domainspace$.  The domain $\Domainspace$ is divided into three parts. The first $\Dm$ is  filled with the storage  medium (soil)  which   is assumed to be homogeneous for simplicity,  and characterized by the constant material parameters $\rhom, \kappam$, and $\cpm$ denoting  mass density,    thermal conductivity and   specific heat capacity, respectively. The second is $\Df$, it  represents the \phxs and is filled with a fluid (water) with constant material parameters $\rhof, \kappaf$ and $\cpf$. The fluid moves with time-dependent velocity $v_0(t)$ along the \phxsk. For the sake of simplicity, we will restrict ourselves to the case frequently encountered in practice, where the pumps that move the liquid are either switched on or off. Thus, the velocity $v_0(t)$ is piecewise constant taking   only the two values   $\vconst>0$ and zero.  Finally, the third part is the interface $\DInterface$ between $\Dm$ and $\Df$. 
We neglect modeling the wall of the \phx and assume perfect contact between the \phx and the ground for simplicity. Details are given below in \eqref{Interface1} and \eqref{eq: 13f1}. We summarize as follows 

\begin{figure}[h!]		
	\begin{center}
		
		\includegraphics[width=0.9\linewidth]{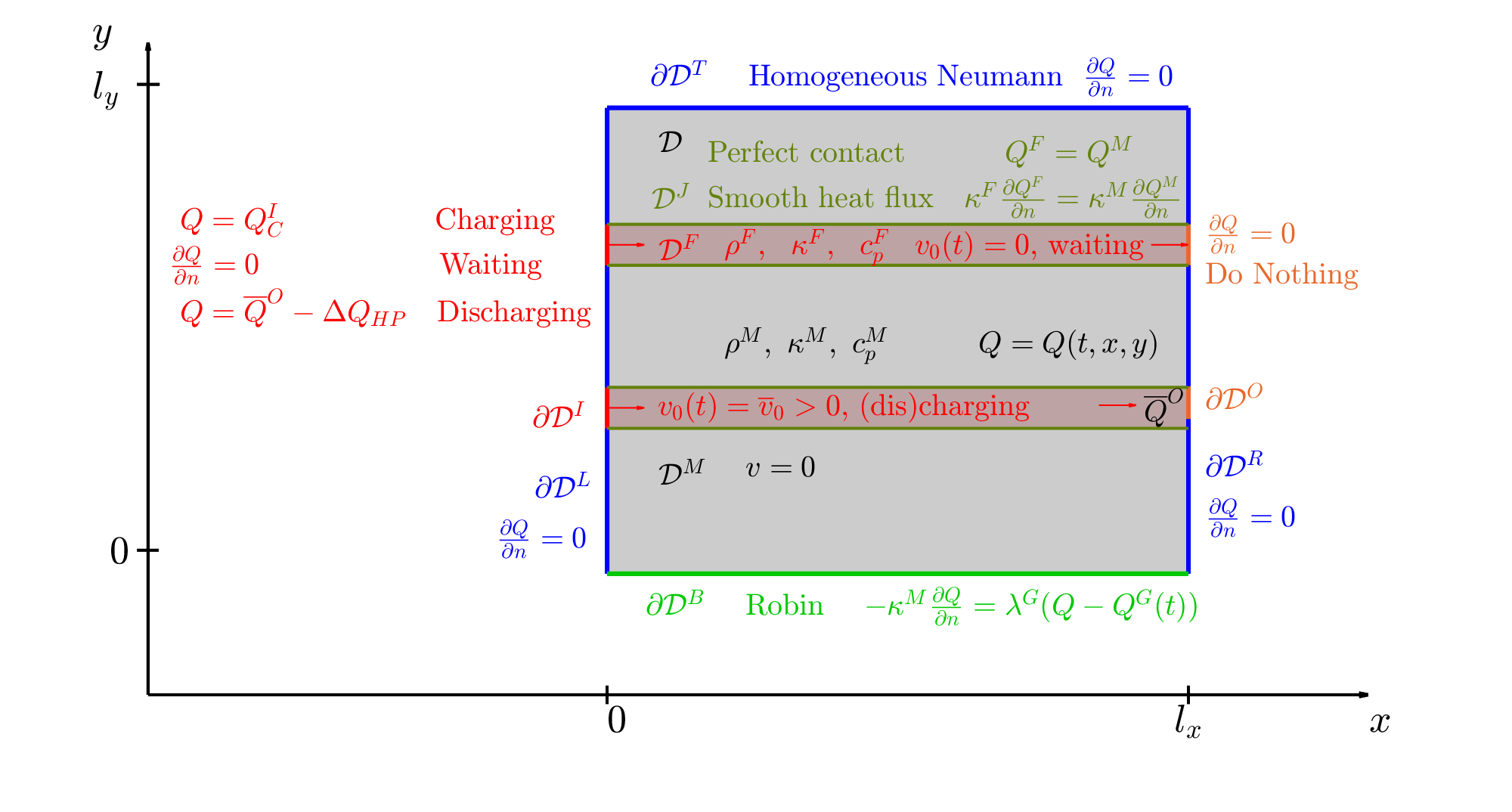}
	\end{center}
	\caption{\label{bound_cond1} Two-dimensional model of the geothermal storage: decomposition of the domain $\Domainspace$,   boundary and interface conditions. Red arrows indicate the direction of the flow.} 
\end{figure}

\begin{assumption}
	\label{assum1}~
	\begin{enumerate}
		\item  Material parameters of the medium   $\rhom, \kappam, \cpm$ in the domain $\Dm$  and of the fluid  $\rhof, \kappaf, \cpf$ in the domain  $\Df$ are constants.
		\item  Fluid velocity is piecewise constant,   that is, $v_0(t)=\begin{cases}
			\vconst>0,  &\text{pump~on,}\\
			0, & \text{pump off.}
		\end{cases}$
		\item  Perfect contact   and impermeability at the interface between fluid and  medium.
		\item  $\Qg(t)\le \QinC$ for all $t\in [0,T]$. 
	\end{enumerate}
\end{assumption}

The last assumption ensures that the GES temperature and thus also the outlet temperature $\Qout$ is always bounded by the inlet temperature during charging $\QinC$.
Note that the  findings derived from our two-dimensional model, where $\Domainspace$ is a rectangular cross-section of a box-shaped storage, can be also used for three-dimensional model under the assumption that the storage domain is a cuboid of depth $l_z$ with a homogeneous temperature distribution in the $z$-direction.

\paragraph{Heat equation}
The temperature $Q=Q(t,x,y)$ in the external storage is governed by the linear heat equation with convection term
\begin{align}
	\rho \cp \frac{\partial Q}{\partial t}={\nabla \cdot (\kappa \nabla Q)}-
	{\rho v \cdot \nabla (\cp Q)},\quad  (t,x,y) \in (0,T]\times \Domainspace \setminus \DInterface,   \label{heat_eq}
\end{align}
where  $\nabla=\big(\frac{\partial}{\partial x},\frac{\partial}{\partial y}\big)$ denotes the gradient operator. The first term  on the right hand side describes diffusion, while the second represents convection of the moving fluid in the \phxsk.
Further,
$v=v(t,x,y)$ $=v_0(t)(v^x(x,y),v^y(x,y))^{\top}$  denotes  the velocity vector with $(v^x,v^y)^\top$ being the normalized directional vector of the flow. 
According to Assumption \ref{assum1}, the material parameters $\rho,\kappa, \cp$  
depend on the position $(x,y)$ and take the values $\rhom,\kappam, \cpm$ for points in $\Dm$ (medium) and  $\rhof, \kappaf, \cpf$ in  $\Df$ (fluid).

Note that there are no sources or sinks inside the storage and therefore the above heat equation appears without forcing term.
Based on this assumption, the heat equation (\ref{heat_eq}) can be written as
\begin{align}
	\frac{\partial Q}{\partial t}={d}\Delta Q-
	{ v \cdot \nabla Q},\quad  (t,x,y) \in (0,T]\times \Domainspace \setminus \DInterface,   \label{heat_eq2}
\end{align}
where $\Delta=\frac{\partial^2}{\partial x^2}+\frac{\partial^2}{\partial y^2}$ is the Laplace operator and  $ d= d(x,y)$ is the thermal diffusivity which is piecewise constant with values  $ d^\dom=\frac{\kappa^\dom}{\rho^\dom \cp^\dom}$  with $\dom=\medium$ for  $(x,y)\in \Dm $ and $\dom=\fluid$   for  $(x,y)\in \Df$, respectively.   
The initial condition $Q(0,x,y)=Q_0(x,y)$ is given by the initial temperature distribution $Q_0$ of the storage.

\subsubsection{Boundary  and Interface Conditions}
For the description of the boundary conditions we decompose the boundary $\partial\Domainspace$ into several subsets as depicted in Fig.~\ref{bound_cond1} representing the insulation on the top and the side,  the open bottom, the inlet and outlet of the \phxsk.  Further, we have to specify conditions at the interface between \phxs and  soil. The inlet, outlet and the interface  conditions model the heating and cooling of the storage via \phxsk. We distinguish between the two regimes ``pump on'' and ``pump off'', where for simplicity, we assume perfect insulation at inlet and outlet if the pump is off.  Since we focus on the heat transfer over the open bottom boundary, we neglect the losses over the insulated top and side and assume perfect insulation at these boundaries. 
This leads to the following boundary conditions.

\begin{itemize}				
	\item \textit{Homogeneous Neumann condition} describing perfect insulation on the top  and the side 
	\begin{align}\frac{\partial Q}{\partial \normalvec}=0,\qquad (x,y)\in 
		\partial \Dtop\cup \partial \Dleft \cup \partial \Dright, 
		\label{Neumann1}
	\end{align}
	where $\partial \Dleft=\{0\} \times  [0,l_y] \backslash \partial \Din$, ~
	$\partial \Dright=\{l_x\} \times  [0,l_y] \backslash \partial \Dout, \partial\Dtop=[0,l_x] \times \{l_y\}$  and $\normalvec$ denotes the outer-pointing normal vector.
	\item \textit{Robin condition} describing heat transfer at the bottom 
	\begin{align}
		-\kappam\frac{\partial Q}{\partial \normalvec}=\heattransfer(Q-\Qg(t)), \qquad  (x,y)\in 
		\partial \Dbottom,
		\label{Robin}
	\end{align}
	with $\partial \Dbottom=(0,l_x) \times \{0\}$, where $\heattransfer>0$ denotes the  heat transfer coefficient  and $\Qg(t)$ the underground temperature.	
	\item 
	\textit{Mixed boundary conditions at the inlet:}   Here one has to distinguish three cases. \\
	(i) \textit{Charging:}	The pump is on ($v_0>0$), the fluid arrives at the storage with the inlet temperature $\QinC$ which is a given constant, and we can impose a Dirichlet boundary condition.\\
	(ii) \textit{Waiting:} The pump is off ($v_0=0$), and we set a homogeneous Neumann condition describing perfect insulation.\\				
	(iii) \textit{Discharging:} In this mode, the pump is switched on ($v_0>0$) and the operation of the heat pump must be taken into account. The fluid from the  \phx outlet  arrives at the inlet of the heat pump with  the  temperature  $\Qout(t)$ and returns to the inlet of the geothermal storage  at  the temperature $\Qout(t)-\Delta T_{HP}$, where  $\Delta T_{HP}>0$ is called heat pump spread and assumed to be a given constant.		
	Mathematically, this leads to a coupling condition which links  the inlet temperature to the average temperature at the \phx outlet.
	% via $\Qin(t)=\QinD(t)=\Qout(t)-\Delta T_{HP}$. 	
	
	Summarizing, we obtain:
	\begin{align}
		%\begin{cases}
		\left\{\!\!	\begin{array}{rll}	
			Q &=~~\QinC, &\text{ ~charging,} \\[0.5ex]
			\frac{\partial Q}{\partial \normalvec}&=~~0, &\text{ ~waiting,} \\[0.5ex]
			Q &=~~ \Qout(t)-\Delta Q_{HP}, &\text{ ~discharging},
		\end{array}
		%\end{cases} 
		\quad (x,y)\in 	\partial \Din .\right.
		\label{input}
	\end{align}

	\item \textit{``Do Nothing'' condition} at the outlet in the following sense. If  the  pump is on ($v_0>0$), then the total heat flux directed outwards can be decomposed into a diffusive  heat flux given by $\kappaf\frac{\partial Q}{\partial \normalvec}$ and a convective  heat flux given by $v_0 \rhof \cpf Q$.  In our model, we can neglect the diffusive heat flux. This leads to a homogeneous Neumann condition
	\begin{align}\frac{\partial Q}{\partial \normalvec}=0,\qquad (x,y)\in 
		\partial \Dout. 
		\label{output}
	\end{align}
	If the pump is off, then we  assume  perfect insulation which is also described by the above condition.	
	
	\item \textit{Smooth heat flux} at interface $\DInterface$ between fluid and soil leading to a coupling condition
	\begin{align}
		\kappaf\frac{\partial \Qff}{\partial \normalvec }=\kappam\frac{\partial \Qmm}{\partial \normalvec },
		\qquad  (x,y)\in \DInterface.
		\label{Interface1}
	\end{align}	
	Here, $\Qff, \Qmm$ denote the temperature of the fluid inside the \phx and of the soil outside the \phxk, respectively.
	Moreover, we assume that the contact between the \phx and the medium is perfect which leads to a smooth transition of a temperature, that is, we have 
	\begin{align}
		\Qff=\Qmm,\qquad  (x,y)\in \DInterface. \label{eq: 13f1}
	\end{align} 
\end{itemize}

\subsubsection{Aggregated Characteristics}	
\label{sec:Aggregate}
The solution of the heat equation \eqref{heat_eq2} allows to describe the spatio-temporal temperature distribution in the GES.  However, for the optimal management of a residential heating system equipped with a GES, which we consider in this article, it is not necessary to know the complete information about this distribution,  that is, $Q(t,x,y)$ at each individual grid point.	
Instead, it is sufficient to know only the dynamics of some aggregated quantities of the temperature distribution, e.g. the average temperature in the storage medium, in the \phx and at the outlet boundary of the \phxk. They can be computed by post-processing after solving the PDE. Some of these aggregated characteristics are presented below.
We start with  the average temperature in the medium and the fluid. They are  given by
\begin{align}
	\Qav^\dom(t) &= \frac{1}{|\mathcal{D}^\dom|} \iint_{\mathcal{D}^\dom} Q(t,x,y) dxdy, \quad \dom=\medium,\fluid,
	\label{Average-M-F}
\end{align}
where $\Qm(t)$ denotes the  average temperature in the medium and $\Qf$ the average temperature of the fluid of the \phxsk. These quantities allow to determine the amount of thermal energy stored. Below in Subsect.~\ref{Sec-contol-Discre}, we will  require that $\Qm$ satisfies the state constraint $\Qm(t) \in [\underline{q}, \overline{q}]$ for all $t\in[0,T]$ with $\underline{q}< \overline{q}$, which leads to  a state-dependent control constraint.

Another aggregated characteristics is the average temperature at the outlet boundary given by 
\begin{align}
	\label{Average-Outlet}
	\Qout(t) &=	 \frac{1}{|\partial \Dout|} \int_{\partial \Dout} Q(t,x,y)\mathrm{d}s.
\end{align}
We refer to  \cite{takam2025energies} and \cite[Section 5]{TakamWunderlichPamen2023} for more details. 

\subsection{Average Temperature in the Internal Storage}\label{InternalStorage} The IES is assumed to be a stratified water tank. For the ease of exposition we assume that the technical implementation is such that there is a constant and known bottom temperature $\underline{p}$ and top temperature $\overline{p} > \underline{p}$.  Further, we only consider the  spatially averaged temperature of the IES  as state variable, which we denote by $P$. 

We assume that charging the GES by transferring heat from the IES is such that a (conventional) pump sends fluid with an inlet temperature $\QinC$ from the IES to the GES.  After passing the GES  the fluid returns to the IES with the  GES outlet temperature $\Qout(t)$.  Charging the IES by  transferring heat from the GES is such that heat pump raises the temperature of the fluid from $\Pout$, to a given constant and known temperature $\Pin$ with $\Pin >\underline{p}$.  In this process,  the heat pump uses the heat extracted from the GES and additional electrical energy. More details about charging and discharging processes are given in Sect. \ref{sec:Heating-S}  
\begin{figure}[h!]
	\centering  
	\includegraphics[width=12cm,height=6cm]{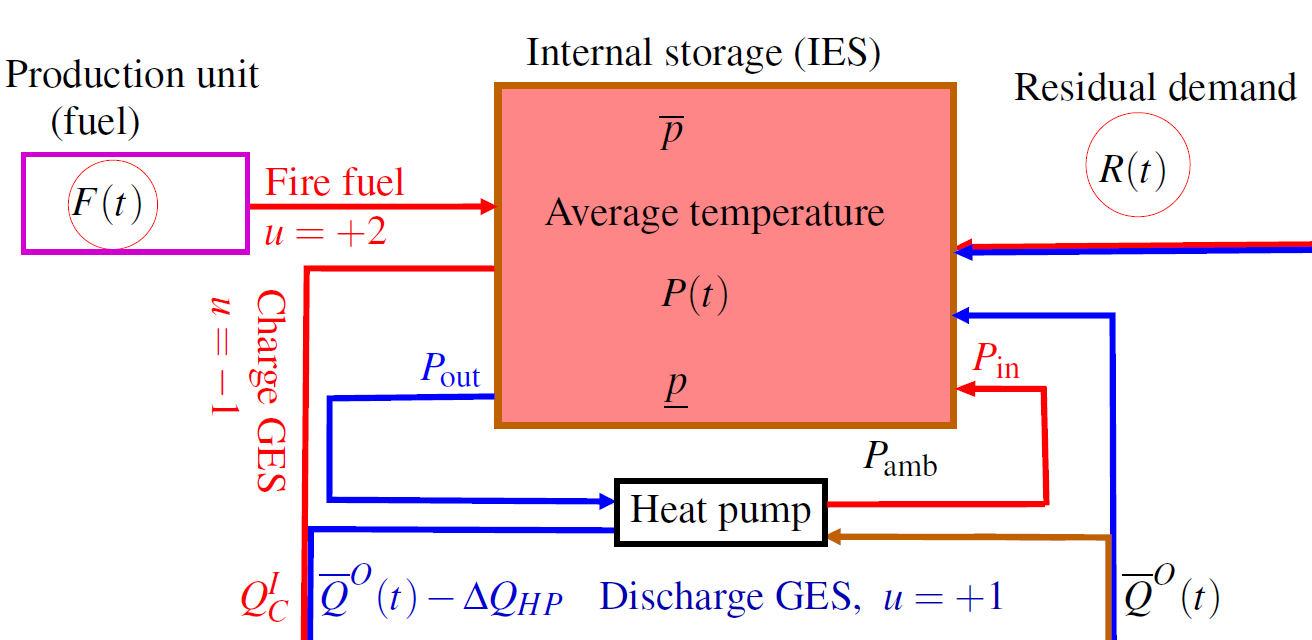}
	\caption{Changes of thermal energy in the internal storage  }
	\label{fig:internal_P}
\end{figure}

The changes of thermal energy in the IES are  either due to inflow of energy from overproduction, or from the GES, or by firing fuel. Further, there may be  an outflow of energy to satisfy the positive residual demand, or to the GES, or due to the loss to environment as depicted in Fig.~\ref{fig:internal_P}. The environment is assumed to be at   temperature $\Pamb(t)< \underline{p}$  which may vary over time.
The dynamics of the IES is then given by 
\begin{align}
	dP(t)=(\drift_P(t,\ResD(t),\Qout(t),u(t))-\gamma(P(t)- \Pamb(t)))dt, ~ P(0)=p_0 \in  [\underline{p},\overline{p}],
	\label{Internal_S}
\end{align}
where $\ResD$ is the  deseasonalized residual demand given by equation \eqref{Residual_deseason} and $\Qout$ is the average temperature of the fluid at the outlet of the \phxs   given in \eqref{Average-Outlet}. The quantity $-\gamma(P-\Pamb)$ is the heat loss to the environment at time $t$, where  $\gamma={\kappa^P A^P}/({m^P\cpw})$ is a constant with $m^P$ the mass of the water in the IES, $\cpw$ the specific heat capacity of the water, $\kappa^P$ the  heat transfer coefficient, $A^P$ the total surface of the IES. The function $\drift_P$ is given by
\begin{align}
	\drift_P(t,\resd,q,\afix)= \begin{cases}
		- \zeta_P {(\mu_R(t)+\resd)}+ \zeta_F,& \afix=\Fuel,\\
		- \zeta_P {(\mu_R(t)+\resd)} + \zeta_C(\Pin-\Pout),& \afix=\Charg,\\
		- \zeta_P {(\mu_R(t)+\resd)},	&\afix=\phantom{+}\Wait, \\
		- \zeta_P {(\mu_R(t)+\resd)}- \zeta_D(\QinC-q),&\afix=\Discharg,\\
		~~ 0, &\afix=\Spill,
	\end{cases}
	\label{Internal_Psi}
\end{align}
for  $t\in[0,T], r\in\R,q\in[\underline{q},\overline{q}], \afix\in\feasible$. Here,  $ \zeta_P=(m^P\cpw)^{-1},~ \zeta_D=  \zeta_PL_D$, $ \zeta_C= \zeta_PL_C$, and $ \zeta_F= \zeta_P L_F$ are energy conversion factors with some positive constants $L_D,~L_C,~L_F$. 
The increment of the total thermal energy in the IES at time $t$ is then given by $m^P\cpw dP(t)$. In the dynamics of the IES we assume that  $\Pin > \Pout$.  The residual demand  $\resi=\mu_R(t)+\resd$ appears in the dynamics of $P$ with negative sign because a positive residual demand decreases the temperature in the IES and a negative residual demand increases temperature in the IES.  Note that during over-spilling ($\afix=\Spill$)  the ODE \eqref{Internal_S} for  $P$  reads $dP=-\gamma(P-\Pamb) dt$.  It describes the cooling of IES by the colder environment. For simplicity,  any other inflows and outflows of energy are neglected by setting $\drift_P=0$. This will facilitate the time discretization below in Subsect.~\ref{sec:TimeDiscretization} and the construction of a transition operator with additive Gaussian noise in Proposition \ref{prop:state_recursion}.   

Below in Subsect.~\ref{Sec-contol-Discre}, we will  require that $P$ satisfies the state constraint $P(t) \in [\underline{p}, \overline{p}]$ for all $t\in[0,T]$ which leads to  a state-dependent control constraint.
Note that \eqref{Internal_S} is a random ODE, since the drift coefficient depends on the stochastic process $\ResD$. In contrast to an SDE, it does not contain a diffusion term and is not driven by a Wiener process. 

\section{Approximation of State Dynamics}
\label{State-approx}

We recall that  the dynamics of $R$ and $F$ are given by the  SDEs \eqref{Residual_deseason} and \eqref{Fuel_F}, and  $P$ is governed by the random ODE \eqref{Internal_S}.
The state  $Q$, the temperature distribution in the GES,  is governed the initial boundary value problem for the   linear parabolic PDE \eqref{heat_eq2}
\begin{align*}
	\frac{\partial}{\partial t} Q(t,x,y)&=d(x,y)\Delta Q(t,x,y)-
	v(t,x,y) \cdot \nabla Q(t,x,y), \quad  (t,x,y) \in (0,T]\times \Domainspace \backslash \DInterface, \\
	~Q(0,x,y)& =Q_0, \quad (x,y)\in \Domainspace \\
	\text{+ boundary}	& \text{ and interface conditions given in \eqref{Neumann1} through \eqref{eq: 13f1}}. 
\end{align*}
Giving that one of the state components, the temperature $Q=Q(t,x,y)$ in the GES, depends not only on time $t$ but also on spatial variables and its dynamics is governed by a PDE,  the state process $\State=(R,F,P, Q)^\top$ takes values in an infinite-dimensional space. This leads to a \emph{non-standard} stochastic optimal control problem for which we do not expect to find a tractable solution. We therefore consider the following finite-dimensional approximation.

\paragraph{Approximation of the temperature distribution in the GES}
A detailed inspection of the control system, and in particular   the  state and control constraints in Subsect.~\ref{Sec-Discret-STOPT} and the associated performance criterion in Subsect.~\ref{Sec-contol-Discre}  show that,  we do not really need to know the complete spatio-temporal temperatures distribution. As already outlined in Subsection \ref{sec:Aggregate}, where we introduced aggregated quantities of the temperature distribution,  it is sufficient to know only the dynamics of the average temperature in the storage medium, in the \phx and at the outlet boundary of the \phxk, which we denoted by $\Qm,\Qf,\Qout$, respectively.  Let $\overline Z:[0,T]\to \R^{\nO}$ denote vector function containing as entries the desired above mentioned aggregated characteristics describing the dynamics of the spatial temperature distribution in the GES. We denote by $\overline Z$ the output variable. A typical example for $\nO=3$ is $\overline Z=(\Qm,\Qf,\Qout)^\top$.

In  the previous  work \cite{TakamWunderlich_redu2024}, we have shown that by applying model order reduction techniques to the spatially discretized heat equation \ref{heat_eq2}, one can find quite accurate approximations of the dynamics of the output $\overline Z$  from a suitable chosen low-dimensional system of ODEs.  We only briefly outline the corresponding approximation steps here.
\begin{enumerate}
	\renewcommand{\labelenumi}{(\roman{enumi})}
	\item  The first step is to approximate the original mathematical model given by the PDE \eqref{heat_eq2}, by an ``analogous model‘’ with time-invariant dynamics. We assume that, in contrast to the original model, the fluid always moves at a constant velocity $\vconst$, even during the pump-off periods. 
	During these periods, the fluid in the original model is at rest and is only subject to diffusive heat propagation.   In order to imitate this behaviour of the fluid at rest by a moving fluid, we assume that the temperature 	at the inlet of the \phx is equal to the average temperature of the \phx fluid  $\Qf$.  In this way, the average temperature of the fluid is maintained. Approximation errors may occur, as a possible temperature gradient along the \phx is not maintained and is replaced by an almost flat temperature distribution. 
	
	Mathematically, in this setting, the  initial boundary value problem for the heat equation \eqref{heat_eq2} now contains a modified boundary condition at the inlet. During pump-off periods,  the homogeneous Neumann boundary condition in \eqref{input} is replaced  by a nonlocal coupling condition  similar the condition  we already impose during discharging the GES. For further details,  we refer to our work \cite[Section 5]{TakamWunderlich_redu2024}. 
	
	\item The next step is a   spatial discretization of the initial boundary value problem for the PDE \eqref{heat_eq2} for $Q$, modified as described above. Following our  previous work \cite{TakamWunderlichPamen2023}, this  leads to  a  system of ODEs for a vector function $\widetilde Y:[0,T]\to\R^n$ with a  high  dimension $n$. For a finite difference discretization as in \cite{TakamWunderlichPamen2023}, the entries $\widetilde Y_i(t)$ represent the temperature at the $i$-th grid point.	The desired aggregated characteristics, that is, the entries of the output $\overline Z$,  can be approximated by  linear combinations of the entries of the $\widetilde Y$. Mathematically, this can be expressed as  $\overline Z(t) \approx \widetilde Z(t)= \widetilde C \widetilde Y(t)$, where  $\widetilde Z\in \R^{\nO}$  denotes  the output approximation,  and $\widetilde C$ an $\nO\times n$ matrix.
	
	\item Finally, we apply balanced truncation model order reduction as presented in \cite[Section 6]{TakamWunderlich_redu2024} to the resulting  $n$-dimensional  system of ODEs. This leads to a system of ODEs of dimension $\ell \ll n$ for an $\ell$-dimensional reduced-order state $Y$ from which approximations of the output variables are obtained by linear combination of the entries of $Y$,   that is, $Z(t)\approx \overline Z(t)= \Cav Y(t)$, where $\Cav$ is a $\nO\times \ell$ matrix. For the above mentioned example with $\nO=3$ and output  $\overline Z=(\Qm,\Qf,\Qout)^\top$, we find the approximations of the form $\Qav^\dom \approx \QavAppr^\dom =\Cav^\dom Y(t)$, $\dom=\medium,\fluid,\outlet$, where the $\nO$-dimensional  vectors $\Cav^\dom$ form the rows of the matrix $\Cav\in\R^{3 \times \ell}$.  
\end{enumerate}

To summarise, the dynamics of the GES and the  response of the aggregated characteristics to the control  process $u(t)$ can be  represented approximately as follows
\begin{align}
	\label{reduced-S}
	\begin{split}
		dY(t)& =\drift_{Y}(t,Y(t),u(t))dt, \quad Y(0)=\Rstate_0 \in  \R^\ell,\\
		Z(t)& =\Cav Y(t),
	\end{split}
\end{align}
where
$\drift_{Y}(t,y,\afix)=\mat{A}(\afix)y+\mat{B} g(t,\afix) $   for $\afix\in\actionspace$, 
$\Rstate=(\rstate)^{\top} \in \R^\ell$ denotes the reduced-order state, and $\Rstate_0$ its initial value at time $t=0$. The second equation in \eqref{reduced-S} is an algebraic equation called output equation in which  $\Cav$ is the $\nO\times \ell$-output-matrix introduced above.  
Its entries depend on the type of information the manager wishes to get from the system, and $Z \in \R^{\nO}$ is the vector of approximated  aggregated characteristics.   The control-dependent  system matrix $A(\afix)$ is given by (see \cite[Section 5.1]{takam_PhD_2023})
\begin{align}
	&A(\afix)=\begin{cases}
		\overline{\mat{A}}, & \afix =\Discharg~ (\text{charge GES}),\\
		\overline{\mat{A}}+B^1\OutputOut ,& \afix =\Charg~(\text{discharge GES}),\\
		\overline{\mat{A}}+B^1\OutputF ,& \text{otherwise},
	\end{cases} 
	%	B(\afix)=\begin{cases}
		%		B_\ell \qquad \afix \in \{\Charg,\Discharg\},\\
		%		B^2_\ell \qquad \text{otherwise},
		%	\end{cases}
	\label{Matrix_AB}
\end{align}
with $\overline{\mat{A}} \in \R^{\ell \times \ell}$ and $B=(B^1, B^2)\in \R ^{\ell \times 2}$ is the input matrix,  resulting from the application of balanced truncation model order reduction.
The input function $g$ is given by
\begin{align}
	g(t,\afix)= {\Qin_{\afix}(t) \choose Q_{G}(t)}  ~~\text{with}~ ~
	\Qin_{\afix}(t)=\begin{cases}
		~~~\QinC,            &\afix=\Discharg~ (\text{charge GES}),\\
		- \HPspread,   &\afix=\Charg~(\text{discharge GES}),\\
		~~~~0, & \text{else}.		
	\end{cases}
	\label{Input}
\end{align}
Recall, $\QinC$ is the GES inlet temperature during charging, $\HPspread$ the heat pump spread, and $Q_G(t)$ the underground temperature.

% This is now a control problem  in \emph{standard form}, while we are facing the curse of dimensionality. 

\paragraph{Approximated state dynamics}  Now, we are able  to approximate the dynamics of the infinite dimensional state  $(R,F,P, Q)^\top$ by the dynamics of a finite-dimensional state   $\State = (R,F,P, Y^{\top})^\top$ taking values in a state space $\Statespace\subset \R^{l+3}$, where $\ell$ denotes the dimension of the vector function $Y$ which replaces and approximates the last state component, the spatial temperature distribution $Q$.   The continuous-time dynamics of this state process $\State$  is given by  the SDE \eqref{Residual_deseason}  for the residual demand $R$, the SDE \eqref{Fuel_F} for fuel price $F$, the random ODE \eqref{Internal_S} for $P$, the average temperature in the IES, and finally, the $\ell$-dimensional  system of ODEs \eqref{reduced-S} for $Y$. 

\begin{remark}\label{rem:cont_SOC}
	In such a continuous-time setting, the cost-optimal management of a residential heating system with a GES can be formulated as an optimization problem for a (degenerate) controlled diffusion process.  Applying the standard dynamic programming solution approach  results in the Hamilton-Jacobi-Bellman (HJB) equation, a highly nonlinear PDE that serves as a necessary optimality condition. Apart from various analytical problems arising from the fact that the state process is a degenerate diffusion process, this approach suffers from the curse of dimensionality. The main problem is that the HJB equation can only be solved numerically, and the computational cost of solving such nonlinear PDEs with $\ell+3$ state variables is prohibitively high. Therefore, in the next section, we proceed with a  more tractable approach based on time discretization.
\end{remark}

\section{Stochastic Optimal Control Problem}
\label{sec:Discrite-OP}
In this section, we formulate the cost-optimal management problem  for a residential heating system mathematically as a discrete-time stochastic optimal control problem.  The derivation is based on a  time discretization of the continuous-time state dynamics studied above.
The starting point is the assumption, often fulfilled in reality, that the control can only be changed at a few discrete points in time and held constant in between. 
To avoid discretization errors, it is assumed that the state variables evolve continuously in time between discrete time points and are determined by the ODEs and SDEs derived in sections \ref{sec:StateDynamics} and \ref{State-approx}, which capture the dynamics of the system.
In the derived control system, however, only the state values at the discrete points in time are used for the control decisions. 
The control problem is then  treated as a MDP and solved using dynamic programming techniques.

\subsection{Time Discretization}
\label{sec:TimeDiscretization}
Recall that, the state process $X=(R,F,P,Y^\prime)^\prime$ takes values in  the state space $\Statespace \subset \R^d$ of dimension $d=\ell+3$, where $\ell$ is the dimension of the reduced-order state $Y$ from the approximation of the GES temperature distribution, see \eqref{reduced-S}. We subdivide the planning horizon  $[0,T]$ into $N$  uniformly spaced subintervals of length $\Delta t = T/{N}$ and define the time grid points  $t_n = n \Delta t$ for $n =0,\ldots,N$.	

\paragraph{Discrete-time control and state process}
Now, we want to study the dynamics of the state process $X$ sampled at the discrete time points $t_n, n=0,\ldots,N$, under the  following assumption.
\begin{assumption}(Piecewise constant control).
	\label{Ass:ConstantControl}\\
	The control $u$ is kept constant between two consecutive grid points of the time discretization, that is 
	\begin{align}
		u(t) = u(t_n)=: \ud_n,\quad \text{for }~ t \in [t_n,t_{n+1}),~~n=0,\ldots,N-1,
	\end{align}		
	and denote by ${\aprocess}=(\ud_n)_{n=0,\ldots,N-1}$ the sequence of controls or actions $\ud_n \in \actionspace=\{\Spill,\Charg,\Wait,\Discharg,\Fuel\}$ chosen by the decision maker in period $[t_n,t_n+1]$.	
\end{assumption}

Next, we derive the discrete-time dynamics of the  sampled state process $X=(X_n)_{n=0,\ldots,N}$ in terms of a recursion. 
Starting point is the dynamics of the the continuous-time state process given  by  the SDE \eqref{Residual_deseason} for  $R$,  SDE \eqref{Fuel_F} for  $F$, the random ODE \eqref{Internal_S} for $P$, and  the $\ell$-dimensional  system of ODEs \eqref{reduced-S} for $Y$. Since all these equations are linear, we can benefit from the availability of closed-form solutions that allow us to avoid discretization errors in the derivation of the discrete-time state dynamics. To simplify the calculations, we assume the following for the time-dependent model parameters.
\begin{assumption}(Piecewise constant model parameters).
	\label{Ass:ConstantParameter}\\
	The time-varying seasonalities $\mu_R,\mu_F$, volatilities $\sigma_R,\sigma_F$, the ambient temperature $\Pamb$,  	and the undergound temperature $\Qg$ are constant between two consecutive grid points of the time discretization, that is
	\begin{align}
		\mu_{\dom}(t)& =\mu_{\dom}(t_n)=\mu_{\dom,n}, &  \sigma_{\dom}(t)& =\sigma_{\dom}(t_n)=\sigma_{\dom,n},\quad \dom=R,F,\\
		\Pamb(t)& =\Pamb(t_n)=\Pambn{n}, & \Qg(t)& =\Qg(t_n)=\Qg_{n},
	\end{align}		
	$\text{for }~ t\in[t_n,t_{n+1}),~~n=0,\ldots,N-1.$	
\end{assumption}
Note that this assumption is not very restrictive, since in reality these parameters can be expected to change only slowly over time and to be nearly constant on short time scales such as the periods during which the control is held constant.

\begin{proposition}[Transition operator]
	\label{prop:state_recursion}
	Under Assumptions \ref{Ass:ConstantControl} and \ref{Ass:ConstantParameter}, there exists a sequence of independent three-dimensional standard normally distributed random vectors 
	for $(\Noise_n)_{n=1,\ldots,N}$ with $\Noise_{n}= (\Noise^R_n,\Noise^F_n,\Noise^P_n)^\top\in\mathcal{N}(0_3,\mathds{I}_3)$ such that the state process $X=(X_n)_{n=0,\ldots,N}$ satisfies the recursion 
	\begin{align}
		\label{state_recursion}
		X_{n+1}=  \mathcal{T}(n,X_n,\ud_n,\Noise_{n+1}), \quad X_0=X(0)=x_0. 
	\end{align} 
	for $n=0,\ldots N-1$. Here,    $\mathcal{T}:\{0,\ldots,N-1\}\times \Statespace\times\actionspace\times \R^3 \to \Statespace $ 
	is  the transition operator which is for $n\in \{0,\ldots,N-1\}$, $x=(r,f,p,y^\top)^\top\in\Statespace,\afix\in \actionspace$, $\noise=(\noise^R,\noise^F,\noise^P)^\top\in \R^3$  defined as 	$\mathcal{T}=(\mathcal{T}^R,\mathcal{T}^F,\mathcal{T}^P,(\mathcal{T}^Y)^\top)^\top$ with 
	\begin{align}
		\label{state_recursion2}
		\begin{split}
			\mathcal{T}^R(n,x,\afix,\noise) &= r \mathrm{e}^{-\beta_R \Delta_N}+\Sigma_{R}(n)\, \noise^R,\\ 
			\mathcal{T}^F(n,x,\afix,\noise) &= f\mathrm{e}^{-\beta_F \Delta_N}+\Sigma_{F}(n)\, \noise^F,\\
			\mathcal{T}^P(n,x,\afix,\noise) & = \mathrm{e}^{-\gamma\Delta_N}p +\mathcal{H}(n,r,y,\afix) 
			+\Sigma_{P}(n,\afix) \big(\sqrt{1-\rho^2(n,\afix)}\, \noise^P + \rho(n,\afix)\,\noise^R\big),\\
			\mathcal{T}^Y(n,x,\afix,\noise) & = \overline{\mathcal{T}}^Y(n,y,\afix)=  \mathrm{e}^{A({\afix})\Delta_N}y ~~+\big(\mathrm{e}^{A({\afix})\Delta_N}-\mathds{I}_{\ell}\big)A^{-1}B\,g_n(\afix),
		\end{split}	
	\end{align}		
	where  $\mathds{I}_{\ell}$  is an $\ell \times \ell$ identity matrix, $A(\afix), B$ are given in \eqref{Matrix_AB},   the functions $\Sigma_{R},\Sigma_{F},\Sigma_{P},\rho$ and  $\mathcal{H}$ are given in Appendix \ref{app:trans_op_details}.  Further, $g_n(\afix)=g(t_n,\afix)$ denotes the constant input function in $[t_n,t_{n+1})$ with $g(t,\afix)$ given in \eqref{Input}.
	
\end{proposition}
\begin{proof}
	The SDE  \eqref{Residual_D}  for  $R$ and  the random ODE \eqref{Internal_S}  for $P$ enjoy closed-form solutions yielding that  $R_{n+1}$ and $P_{n+1}$   can be expressed in terms of stochastic integrals with respect to Wiener process $W_R$ with a deterministic integrand  from which it can be derived that the conditional distribution of  $R_{n+1},P_{n+1}$ given $R_n=r,P_n=p$ and $\ud_n=\afix$ is bivariate Gaussian. Computing mean and the covariance matrix of this distribution leads to the given expressions. Note that  $\Sigma^2_{R}(n)$ and $\Sigma^2_{P}(n)$ are the  variances, and $\rho(n)$ represents the correlation coefficient of this distribution. Similarly, solving the SDE \eqref{Fuel_F} on $[t_n,t_{n+1}]$ shows that  $F_{n+1}$ can be also expressed in terms of a stochastic integral  with respect to Wiener process $W_F$ with a deterministic integrand. This implies that the conditional distribution of  $F_{n+1}$ given $F_n=f$ is  Gaussian  and independent of $R_{n+1},P_{n+1}$ since $W_R,W_F$ are independent.
	Finally, $Y_{n+1}$ can be obtained by solving the system of linear ODEs \eqref{reduced-S} for $Y(t)$ on $[t_n,t_{n+1}]$ with initial value $Y_n=y$. Under the Assumptions \ref{Ass:ConstantControl} and \ref{Ass:ConstantParameter}, this is for each fixed control $\afix$, a system of  autonomous ODEs with a constant forcing term. From the closed-form solution follows the expression for $\mathcal{T}^Y$. For more details, we refer to \cite[Chapter 6]{takam_PhD_2023}. 
\end{proof}

In view of the recursion \eqref{state_recursion} we can consider the  sampled state process $X=(X_n)_{n=0,\ldots,N}$ as a discrete-time stochastic process, which is defined  on  the filtered probability space  $(\Omega,\mathcal{F},\mathbb{F},\mathbb{P})$ with the \textit{filtration} $\mathbb{F}=(\mathcal{F}_n)_{n=0,\ldots,N}$. Here,  $\mathcal{F}_n=\sigma(\{\Noise_1, \ldots, \Noise_n\})$ is the sigma-algebra generated by the first $n$ random variables $\Noise_1,\ldots, \Noise_n$ and $\mathcal{F}_0=\{\varnothing,\Omega\}$ is the trivial $\sigma$-algebra.

Further, the recursion \eqref{state_recursion} for the state process $X$ shows that the conditional distribution of the state  $X_{n+1}$ at the end of the period $[t_n,t_{n+1}]$, given the state $X_n$ and the control $\ud_n$ at the beginning of that period, is Gaussian. This property will be helpful for the formulation the control problem as Markov decision process below in Subsect.~\ref{Sec-Discret-STOPT} , and the derivation of the associated transition kernel. Note that the Gaussian distribution is degenerate since the last state component $Y$ follows  deterministic dynamics. In contrast, the conditional distribution of the first three components $R,F,P$ is a non-degenerate three-dimensional Gaussian distribution, where $F$ is independent of $R,P$.\\

\subsection{State-Dependent Control Constraints}
\label{Sec-contol-Discre}
The residential heating system under consideration is subject to various operational constraints.   A first one is a box constraint for the state component $P$, the IES average temperature. As already mentioned in Subsection \ref{InternalStorage}, it is required that $P(t) \in [\underline{p}, \overline{p}]$,  for all $t\in[0,T]$ with $\underline{p}<\overline{p}$, where $\underline{p}$ and $\overline{p}$ represent an empty and a full IES, respectively.  A second condition is that the authorities only permit the operation of a GES for environmental reasons, if the GES temperature does not exceed a certain range. We model this by    a box constraint to the average temperature in the medium of the form  $\QmAppr(t) \in [\underline{q}, \overline{q}]$ with $\underline{q}< \overline{q}$, where $\underline{q}$ and $\overline{q}$ represent an empty and a full GES, respectively.
This is sufficient, as an inhomogeneous spatial temperature distribution is averaged after some time due to the diffusive propagation of heat in the GES. 
Since $\QmAppr=\OutputM Y$,  the above constraint requires the reduced-order state $Y(t)$ to take values  in the subset $\mathcal{Y}\in\R^\ell$ between the two hyperplanes  defined by  $\OutputM Y=\underline{q} $ and  $\OutputM Y=\overline{q} $ for all $t\in[0,T]$, that is  in
\begin{align}
	\label{RstateSpace}
	\mathcal{Y}\in\R^\ell =\{y\in\R^\ell : \underline{q} \le \OutputM  y \le \overline{q}\}.  
\end{align}
In a model in which the controls can be continuously changed over time, such restrictions mean that charging a storage (IES or GES) is no longer permitted when the storage is full, while discharging can no longer be selected for an empty storage. However, due to Assumption \ref{Ass:ConstantControl},  we are limited to controls that are kept constant between two consecutive points in time of the time grid. In contrast to continuous operation, this means that charging or discharging is no longer permitted when the storage  is “almost” full or empty. 

Therefore, for each time step $n \in \{0,1,\ldots, N-1\}$, one has to derive a subset of the set of all feasible actions $\actionspace$ given in \eqref{feasible A}, that contain the feasible actions available to the controller depending on the state $X_n$ at time $n$. The latter should be defined in such a way that the IES and GES do not become full or empty within the next period $[t_n,t_{n+1}]$.  In view of the dynamics of the state components $P$ and $\RState$ describing the IES and GES,  it is sufficient to consider the state of IES and GES at the end of the period at time $t_{n+1}$. This lead to the following still rather implicit definition  
$$\feasible_0(n,x)=\{\afix \in \actionspace \mid \mathcal{T}(n,x,\afix,\Noise_{n+1}) ~\text{is such that IES and GES are not full or empty} \}$$  
for $n=0,\ldots,N-1$ and $x\in\mathcal{X}$.

Now,  we will show how this simple and intuitive idea can be formulated in a mathematically rigorous way.
Recall, that we have from recursion \eqref{state_recursion}, $P_{n+1}=\mathcal{T}^P(n,X_n,\aprocess_n,\Noise_{n+1})$. Further, the conditional distribution of the IES temperature  $P_{n+1}$ given $X_n$ is Gaussian. This prohibits to satisfy the box constraint
$P_{n+1} \in [\underline{p}, \overline{p}]$ with certainty and requires a relaxation. Therefore, we allow over- and undershooting, that is, $P_{n+1} >\overline{p}$ and $P_{n+1}<\overline{p}$, but constrain the probabilities for these events by some ``small'' tolerance value $\epsilon < 1$.
Furthermore, the recursion \eqref{state_recursion} for $P$ shows that,  the Gaussian distribution of $P_{n+1}$ depends on the deseaonalized residual demand $R_n$ at time  $t_n$. As a realization of a Gaussian random variable, these values are potentially unbounded,  making it impossible to define reasonable relaxed constraints to $P_{n+1}$. Note that $R$ is an exogenous state and  is not subject to control.  Therefore, we restrict the following derivations to a bounded interval $[\underline{r},\overline{r}]\subset \R$ with $\underline{r}<\overline{r}$ in which the random variables $R_n$ take values with very high probability.	
Since the sampled continuous-time process $R$ is a Gaussian Ornstein-Uhlenbeck process for which the marginal distribution of $R(t)$ converges asymptotically for $t\to\infty$ to the stationary distribution $\mathcal{N}(0, \sigma_R^2/(2\beta_R))$. We apply the $3\sigma$-rule and choose 
$[\underline{r},\overline{r}]= [-3{\sigma_{R}}/(2 \beta_R)^{1/2}, +3 {\sigma_{R}}/(2 \beta_R)^{1/2}]$, which carries $99.97\%$ of the probability mass of this distribution. 	

The above mentioned restrictions are formalized by the following truncation operator $\truncop:\mathcal{X}\to \mathcal{X}_\truncop= [\underline{r},\overline{r}]\times [\underline{p},\overline{p}]\times\R^\ell$, which is defined for $x=(r,p,y^\prime)^\prime\in \mathcal{X}$ by 
$$\truncop(x)=(\underline{r}\one_{(-\infty),\underline{r})}(r) +r\one_{[\underline{r},\overline{r}]}(r) +\overline{r}\one_{(\overline{r},\infty)}(r),~ \underline{p}\one_{(-\infty),\underline{p})}(p) +p\one_{[\underline{p},\overline{p}]}(p) +\overline{p}\one_{(\overline{p},\infty)}(p),~y^\prime)^\prime, $$
where $\one$ denotes the indicator function. 
This operator maps the values of $r$ and $p$ outside $[\underline{r},\overline{r}]$ and $[\underline{p},\overline{p}]$ to the nearest boundaries of the intervals. Then the desired set of feasible actions is expressed for $n=0,\ldots,N-1$ and $x\in\mathcal{X}$ by 

\begin{align}
	\label{FeasibleActions}
	\feasible(n,x)& =\feasible^P(n,x) \cap\feasible^Y(n,x), \qquad \text{where} \\
	\begin{split}
		\feasible^P(n,x)&=\big\{\afix \in \actionspace \mid~
		\mathbb{P}\big( \mathcal{T}^P(n,\truncop(x),\afix,\Noise_{n+1}) >\overline{p}\big)\le\eps \qquad  \text{ and }\\
		&\phantom{=\big\{\afix \in \actionspace \mid~ }~~ \mathbb{P}\big( \mathcal{T}^P(n,\truncop(x),\afix,\Noise_{n+1}) <\underline{p}\big)\le\eps\big\}, 
	\end{split}	
	\label{box-constP}\\
	\feasible^Y(n,x)&=\big\{\afix \in \actionspace \mid~\OutputM \overline{\mathcal{T}}^Y(n,y,\afix)\in [\underline{q},\overline{q}]\big\}.
	\label{box-constY}
\end{align}

The  explicit description of the above subsets describing state-dependent control constraints is based on  the following  assumptions to the model parameters. 

\begin{assumption}~\label{Ass-model-par}
	%	We assume that the following hold
	\begin{enumerate}
		\item For all $n=0,\ldots,N-1,   x\in \mathcal{X}_\truncop, \noise \in \R^3$ it holds
		\begin{align}
			\mathcal{T}^P(n,x,\Fuel,\noise) > \mathcal{T}^P(n,x,\Charg,\noise) > \mathcal{T}^P(n,x,\Wait,\noise) > \mathcal{T}^P(n,x,\Discharg,\noise).			
			\label{P_monoton}
		\end{align}
		\item \label{deomp_IES}
		For all $n=0,\ldots,N-1$ it holds that within the period $[t_n,t_{n+1}]$  \\		
		\begin{tabular}[t]{l}
			a full GES cannot  be completely discharged: \\
			\hspace*{4em}$\mathcal{Y}^C(n)=\big\{y\in \mathcal{Y}: ~\OutputM\overline{\mathcal{T}}^Y(n,y,\Charg) \ge \underline{q}\big\}\neq \varnothing$, \\
			an empty  GES  cannot be fully charged:\\
			\hspace*{4em} $\mathcal{Y}^D(n)=\big\{y\in \mathcal{Y}: ~\OutputM\overline{\mathcal{T}}^Y(n,y,\Discharg) \le \overline{q}\big\} \neq \varnothing$,\\
			there are states of the GES for which the action $\afix=\Wait$ (wait) is feasible\\
			\hspace*{4em} $\mathcal{Y}^W(n) =\mathcal{Y}\backslash (\mathcal{Y}^C(n) \cup \mathcal{Y}^D(n)) \neq \varnothing$.
		\end{tabular}\\[0.5ex]		
		\item 
		For all  $x=(r,p,y^\prime)^\prime \in \mathcal{X}_\truncop$, and  $n=0,\ldots,N-1$  it holds  that within the period $[t_n,t_{n+1}]$  \\ 
		\begin{tabular}[t]{l}
			a full IES cannot  be completely discharged:  \\	
			\hspace*{4em}$\big\{p\in[\underline{p},\overline{p}]: ~
			\mathbb{P}\big( \mathcal{T}^P(n,x,\Discharg,\Noise_{n+1}) \ge\underline{p}\big)\ge 1-\eps\big) \big\}\neq \varnothing$   \\[0.5ex]			
			an empty  IES  cannot be fully charged:\\ 
			\hspace*{4em}$\big\{p\in[\underline{p},\overline{p}]: ~
			\mathbb{P}\big( \mathcal{T}^P(n,x,\Fuel,\Noise_{n+1}) \le\overline{p}\big)\ge 1-\eps\big) \big\}\neq \varnothing$ 
		\end{tabular}	\\[0.5ex]		
		\item \label{ass_Qg} The lower and upper bounds for $\QmAppr$ ar such that the underground temperature satisfies\\ 
		\hspace*{4em}$\underline{q} \le \Qg(t)   \le \overline{q}$ \quad for all $t\in[0,T]$.
	\end{enumerate}
\end{assumption}	
\begin{remark}\label{rem:assInterpretation} 
	The first assumption reflects a natural condition for the model parameters, which are such that for a given state $x$ at time $t_n$ and a realization $\noise$ of the random variable $\Noise$, the realizations of the IES temperature at time $t_{n+1}$, $P_{n+1}=p$, satisfy the given order with respect to the different actions $\afix\in \feasible$. Thus, firing fuel leads to a higher $p$ than the transfer of heat from GES to IES. The latter dominates waiting, and waiting leads to a higher $p$ than the transfer of heat from IES to GES.	
	
	The second and third assumptions can always be fulfilled if the length of the time intervals $\Delta_N$ is chosen small enough. Finally, the fourth 
	assumption  states that a fully discharged GES is colder than the underground, so that a heat flow to the GES can be induced as intended, and the GES acts as a heat production unit. Furthermore, the average temperature of a full GES dominates the temperature of the underground.
\end{remark}

Next, we describe the $\feasible^Y(n,x)$ and $\feasible^P(n,x)$ separately. 

\paragraph{Description of the set of feasible actions ${\feasible^Y(n,x)}$}
As already outlined above,  to satisfy the box constraint to the average temperature in the GES medium
$	\QmAppr(t)=\OutputM\RState(t) \in [\underline{q},\overline{q}]  \text{ for all }~t \in [0,T]$,
the reduced-order state $\RState(t)$ has to take values in the subset $\mathcal{Y}\in\R^\ell$ given in  \eqref{RstateSpace}.  
The restriction to piecewise constant controls requires to stop charging or discharging, if at a grid point $t_n, n=0,\ldots,N-1,$  the average temperature $\QmAppr_n=\QmAppr(t_n)$ is already close to $\overline{q}$ or  $\underline{q}$, respectively, such that it will not exceed these boundaries  in the period between $t_n$ and $t_{n+1}$. 
\begin{figure}[h!]
	\centering  
	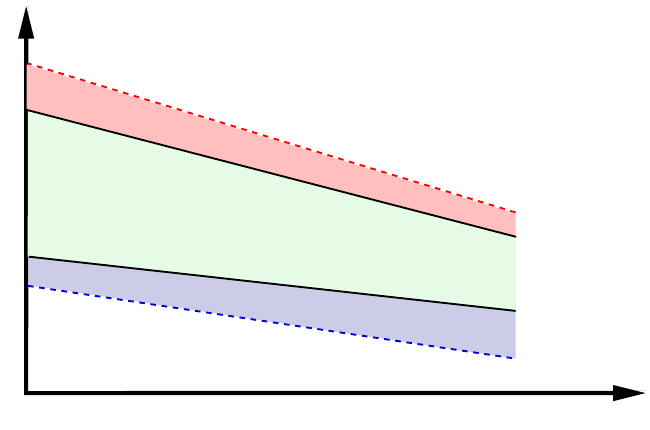
	\caption{Characterization of the set of feasible control $\feasible^Y(n,x)$ for $\ell=2$ }
	\label{Feasible-Y-Char}
\end{figure}
Using the recursion $Y_{n+1}=\overline{\mathcal{T}}^Y(n,Y_n,\afix)$  for the   sequence $(\RState_n)$ of reduced-order states given in \eqref{state_recursion2},
this leads to a decomposition of the  subset $\mathcal{Y}$  of the  form $\mathcal{Y}=\mathcal{Y}^C(n) \cup\mathcal{Y}^D(n) \cup \mathcal{Y}^W(n)$ where
$\mathcal{Y}^C,\mathcal{Y}^D,\mathcal{Y}^W$ are given in Assumption \ref{Ass-model-par},\ref{deomp_IES}.
The subsets  $\mathcal{Y}^C(n)$ and $\mathcal{Y}^D(n)$ contain all states of $\mathcal{Y}$ for which charging and discharging the GES is no longer allowed in period $[t_n,t_{n+1})$, respectively. This decomposition is is sketched for the case $\ell=2$ in Figure \ref{Feasible-Y-Char}.
The set $\feasible^Y(n,x)$ of feasible actions introduced in \eqref{box-constY} is then given for $x=(r,f,p,y)$ as 
\begin{align}
	\feasible^Y(n,x) = \begin{cases}
		\actionspace\setminus \{-1\}, & \text{for } y\in  \mathcal{Y}^C(n),\\
		\actionspace\setminus \{+1\}, & \text{for } y\in \mathcal{Y}^D(n),\\
		\actionspace, & \text{for } y\in  \mathcal{Y}^W(n).	
	\end{cases}
	\label{control_constr_Y}
\end{align}
Note that under Assumption \ref{Ass-model-par} the above subsets $\mathcal{Y}^W(n)$ are non-empty for all $n=,\ldots,N-1$.

\paragraph{Description of the set of feasible actions $\feasible_P(n,x)$} 
We denote  the probabilities appearing in the definition of $\feasible_P(n,x)$ in \eqref{box-constP} by
\begin{align*}
	\overline{\pi}^{\afix}(n,x)=\mathbb{P}\big( \mathcal{T}^P(n, \mathcal{R}(x),\afix,\Noise_{n+1}) >\overline{p}\big),~~\text{ and } ~~
	\underline{\pi}^{\afix}(n,x)= \mathbb{P}\big( \mathcal{T}^P(n,\mathcal{R}(x),\afix,\Noise_{n+1}) <\underline{p}\big),
\end{align*} 
which can also be interpreted as the conditional probabilities that the average IES temperature $P_{n+1}$ at the end of the period $[t_n,t_{n+1}]$ exceeds the boundaries $\underline{p},\overline{p}$ if at the beginning of the period, at time $t_n$, the state is $X_n=x$ and the action $\ud_n=\afix$ is selected.

In view of the inequalities \eqref{P_monoton} given in  Assumption \ref{Ass-model-par}, the following monotonicity properties of $\underline{\pi}^{\afix}$ and $\overline{\pi}^{\afix}$ hold  for all $(n,x)$
\begin{align*}
	&\overline{\pi}^{\fuelf}>\overline{\pi}^{\charge}>\overline{\pi}^{\waite}>\overline{\pi}^{\discharge} ~~\text{and}~~
	\underline{\pi}^{\fuelf} <\underline{\pi}^{\charge}<\underline{\pi}^{\waite} <\underline{\pi}^{\discharge}.
\end{align*}
Let us  define the following subsets of the   state space $\Statespace$ by
\begin{align*}
	\overline{\Statespace}_P^{\afix}(n)&=\{x \in \Statespace \mid \overline{\pi}^{\afix}(n,\truncop(x)) \leq \epsilon \}~~\text{and}~~
	\underline{\Statespace}_P^{\afix}(n)=\{x \in \Statespace \mid \underline{\pi}^{\afix}(n,\truncop(x)) \leq \epsilon \}.
\end{align*}%\pause
Then for a given state $X_n=x$ at time $t=t_n$, an action $\ud_n=\afix$ is feasible  with respect to the constraints  to $P$ given in \eqref{box-constP} if $x \in \overline{\Statespace}_P^{\afix}(n) \cap \underline{\Statespace}_P^{\afix}(n)$. Thus, the set of feasible controls $\feasible_P$  can be expressed  as
\begin{align*}
	\feasible_P(n,x)=\bigcup_{\afix: ~x \in \overline{\Statespace}_P^{\afix}(n) \cap \underline{\Statespace}_P^{\afix}(n)} \{\afix\},
\end{align*}
and contains those actions $\afix$ for which $x$ is in the above mentioned intersection. 
In view of our model setting and the dynamics of the state process, it can be deduced that the projections of the above subsets onto the sub-state spaces 
%\begin{align*}
$	\Statespace_{RP}=\{(r,p) \mid x=(r,f,p,\Rstate) \in \Statespace\},$
%\end{align*}
have the following form
\begin{align}
	\overline{\Statespace}_{RP}^{\afix}(n)=\{(r,p) \mid p \le \overline{h}_n^{\afix}(r,\Rstate)\},\quad\text{and}\quad 
	\underline{\Statespace}_{RP}^{\afix}(n)=\{(r,p) \mid p \ge \underline{h}_n^{\afix}(r,\Rstate)\},
\end{align}
for some continuous functions $\overline{h}_n^{\afix},\underline{h}_n^{\afix}: \R \times \R^\ell \to \R$ which are sketched in 
Fig.~\ref{Projectino-feasi}.
\begin{figure}[h!]
	\centering  
	\includegraphics[width=0.5\textwidth]{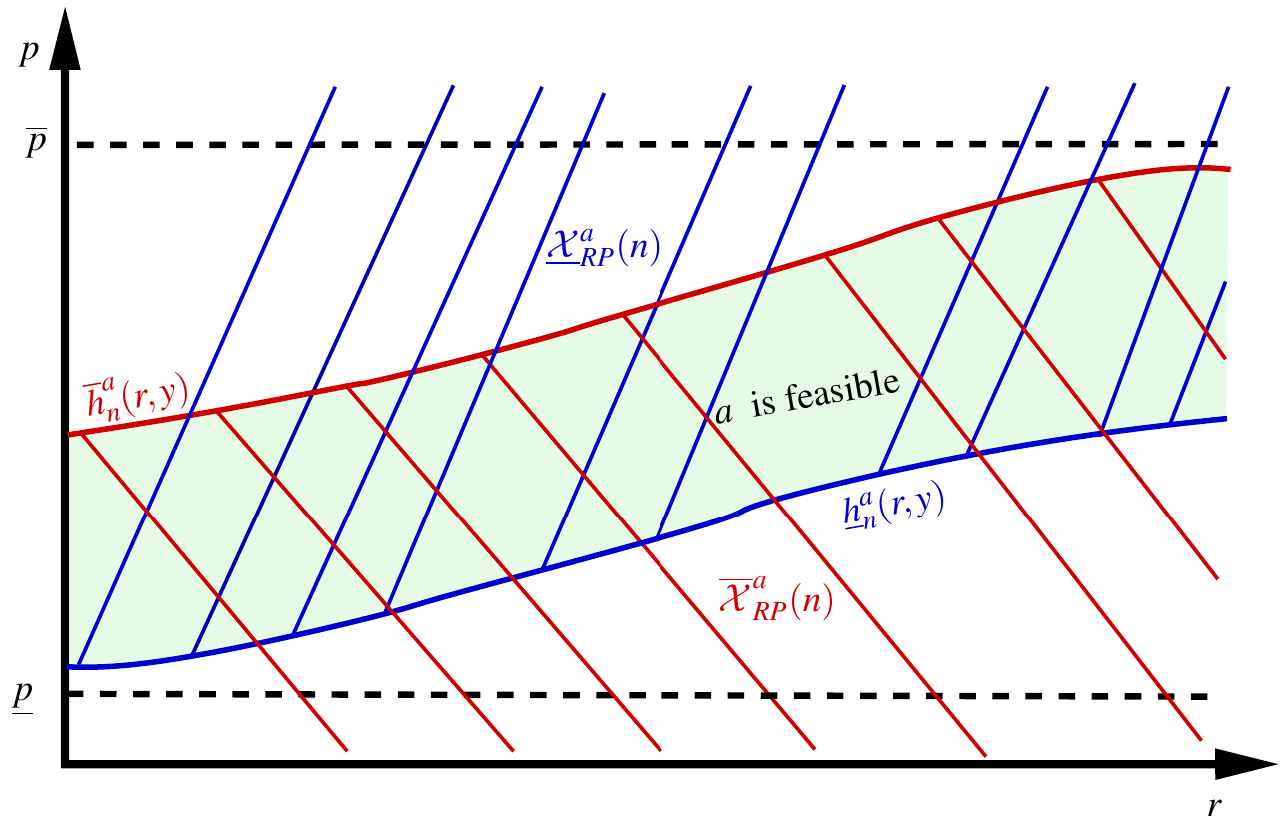}	
	\caption{Projection of $\underline{\Statespace}_P^{\afix}(n)$ and $\overline{\Statespace}_P^{\afix}(n)$ onto $\Statespace_{RP}$	}
	\label{Projectino-feasi}
\end{figure}
Note that only for $\afix=\Discharg$  there is dependence of the functions $\overline{h}_n^\afix, \underline{h}_n^\afix$ on the variable $\Rstate$ since during discharging the IES its average  temperature $P$  depends on the outlet temperature $\QoutAppr=\OutputOut \RState$ of the GES, see \eqref{Internal_S} and \eqref{Internal_Psi}.  It follows, that the  set $\feasible_P(n,x)$ can be subdivided into  8 subsets, which follow from a decomposition of the state space $\Statespace_{RP}$ for the pair of state components $(R,P)$, which is depicted in  Fig~\ref{Feasible-P-Char}. 

\begin{figure}[h!]
	\centering  
	\input{Feasible_SetUp.pdftex_t}
	\caption{Projection of the set of feasible controls $\feasible_P(n,x)$	onto $\Statespace_{RP}$}
	\label{Feasible-P-Char} 
\end{figure}

\subsection{Performance Criterion}
\label{sec:Performance}
We now want to describe the costs arising in the operation of the residential heating system and  derive a performance criterion for the optimization problem. This criterion summarises the expected total discounted costs from the operation of the system and the expected discounted terminal  costs from the evaluation of the stored thermal energy in the IES and GES. 

\paragraph{Admissible controls}
We denote by $\admiss$ the \textit{set of admissible controls} $\aprocess=(\ud_0,\ldots,\ud_{N-1})$ for which, we want to define below the performance criterion that will be minimized  within this set. Since we want to apply dynamic programming methods for solving the optimization problem, we restrict ourselves to Markov or feedback controls defined by $\ud_n=\amarkov(n,x)$ with a measurable function   $\amarkov:\{0,\ldots,N-1\}\times \Statespace\to \actionspace$ which is called \textit{decision rule}. Note that such Markov controls $\aprocess$ are by construction adapted to the  filtration $\mathbb{F}$.
Taking into account the state-dependent control constraints derived in Subsect.~\ref{Sec-contol-Discre} the set of admissible controls is given by  
\begin{align}
	\label{admiss_control_D}
	\begin{split}
		\admiss=\Big\{\aprocess=(\ud_0,\ldots, \ud_{N-1})|~ &
		\aprocess \text{ is a Markov control with }  \ud_n=\amarkov(n,\State^\ud_n)~  \text{for all }n=0,\ldots, N-1,\\
		&	\text{and } \amarkov(n,x)\in \feasible(n,x) \text{ for all } (n,x) \in \{0,1,\ldots, N-1\}\times \Statespace\Big\}.
	\end{split}
\end{align}

\paragraph{ Continuous-time performance criterion}
Let us identify the sequence $\aprocess=(\ud_n)_{n=0,\ldots,N-1}$ defining an admissible  discrete-time control process with the piecewise constant continuous-time control $\overline u(t) =\sum_{n=0}^{N-1} \ud_n \one_{[t_n,t_{n+1})}(t) $,  see Assumption \ref{Ass:ConstantControl}. Then, given the continuous-time state process $X=X^{\overline \uprocess}=(X^{\overline \uprocess}(t))_{t\in[0,T]}$ starts at initial time at $X^{\overline \uprocess}(0)=x_0 \in\Statespace$, the performance of an admissible control $\aprocess \in \admiss$ is given by 
\begin{align}
	\label{performance_cont}
	\mathcal{J}_0(x_0;\aprocess)=
	\mathbb{E}\bigg[\int_{0}^{T}\mathrm{e}^{-\delta s}\runCCont(X^{\overline \uprocess}(s),\overline u(s))\mathrm{d}s + \mathrm{e}^{-\delta  T}\termC(X^{\overline \uprocess}(T))\bigg] 
\end{align}	
where $\delta\ge 0$ denotes the discount factor, and $\runCCont$ and $\termC$ denote the running and terminal costs, respectively, which we specify below. 
Working with discounted costs takes into account the fact that a GES is designed to store thermal energy over longer periods of time (many weeks and months instead of days). Therefore, the economic evaluation should also take into account the time value of money expressed by the discount factor $\mathrm{e}^{-\delta s}$.

\paragraph{Running costs}
These costs rates measured in monetary unites per unit of time, which are  defined by 
\begin{align}
	\runCCont(x,\afix) = \begin{cases}
		\price_Ff, &\afix=\Fuel,\\
		\price_{HP}(\Pin-\OutputOut y)+\price_{P}, &\afix=\Charg,\\
		\price_{P},  &\afix=\Discharg,\\
		0,  & \afix=\{\Wait,\Spill\}.
	\end{cases}
	\label{Running Cost Psi}
\end{align}
Here,  $\price_F$ is the fuel consumption rate of the fuel-fired boiler, $\price_{HP}$  the price of electricity per unit of time consumed by the heat pump  to increase the temperature by $1K$, and  $\price_{P}$ the price of electricity per unit of time consumed by  the pumps moving the fluid between IES and GES.

\paragraph{Terminal costs}	
At the end of the planning period, a terminal cost function $\termC: \Statespace \to \mathbb{R}$ can be used to evaluate the terminal state of the system, in particular the amount of thermal energy stored in the IES and GES, in monetary terms. A typical example are \textit{penalty and liquidation payments} that are applied if the IES temperature $P_N$  and the average temperatures in the GES medium  $\QmAppr_N=\OutputM Y_N$ are below or above a certain user-defined critical value $p_\text{ref} \in [\underline p,\overline p]$ and $q_\text{ref} \in [\underline q,\overline q]$, respectively. Suppose that the terminal state is $X_N=x =(r,f,p,y^\top)^\top$, then the terminal cost is defined by
\begin{align}
	\termC(x) &= 
	m^Qc_p^M\Big( \price^Q_{\text{pen}} (\QmAppr_N-q_{\text{ref}})^- -  \price^Q_{\text{liq}} (\QmAppr_N -q_{\text{ref}})^+ \Big) + m^P\cpf\Big( \price^P_{\text{pen}}(P_N-p_{\text{ref}})^- -  \price^P_{\text{liq}}(P_N-p_{\text{ref}})^+\Big)
	\label{terminal_penalty}
\end{align}
Note that $	m^Qc_p^M(\QmAppr_N-q_{\text{ref}})$ and $m^P\cpf(P_N-p_{\text{ref}})$ describe the amount of thermal energy required to adjust the average temperature in the GES medium from $\QmAppr_N$ to $q_\text{ref}$ and  the average IES temperature from $P_N$ to $p_\text{ref}$, respectively. If these quantities are negative, a penalty is due at a fixed price $ \price^\dom_{\text{pen}}\ge 0, \dom=P,Q,$ per unit of energy, which is added to the total operating costs.  Conversely, if the quantity is positive, a revenue for the liquidation of the residual energy at a fixed price $ \price^\dom_{\text{liq}}\ge 0$ per unit of energy reduces the total costs at terminal time.
The  terminal cost function given in \eqref{terminal_penalty} includes the special cases of (i) only penalty payments for $ \price^\dom_{\text{pen}}>0,  \price^\dom_{\text{liq}}=0$, (ii)  only  liquidation revenues for $ \price^\dom_{\text{pen}}=0,  \price^\dom_{\text{liq}}>0$, and (iii) zero terminal costs for $ \price^\dom_{\text{pen}}=  \price^\dom_{\text{liq}}=0$, $\dom=P,Q$.
An example for the choice of the critical values are  the initial average temperatures in IES and GES which follow from  \eqref{Internal_S} and \eqref{reduced-S}, that is  $p_{\text{ref}}=P_0=p_0$ and $q_{\text{ref}}=\QmAppr_0=\OutputM y_0$. In that case,  \eqref{terminal_penalty} can be used to valuate the storage facilities at horizon time $T$, which can be required if the  management of the residential heating system is transferred to a new owner.

\paragraph{Discrete-time performance criterion}
The performance criterion $\mathcal{J}_0(x_0;\aprocess)$ given in \eqref{performance_cont}	can be rewritten in terms of the sequence of  values $(X_n^\aprocess)_{n=0,\ldots,N}$ obtained from sampling  the continuous-time state process $(X^{\overline \uprocess}(t))_{t\in[0,T]}$. Since  $(X^{\overline \uprocess}(t))$  is defined as solution of a system SDEs and ODEs, and admissible controls are of Markov or feedback type, the state process is a Markov process. This and an application of the tower property of conditional expectation yields
\begin{align}
	\mathcal{J}_0(x_0;\aprocess)& =
	\mathbb{E}\bigg[
	\sum_{k=0}^{N-1} 	\mathbb{E}\Big[\int_{t_k}^{t_{k+1}}\mathrm{e}^{-\delta s}\runCCont(X^{\overline \uprocess}(s),\overline\uprocess(s))\mathrm{d}s\Big| \mathcal{F}_{t_k}\Big] + \mathrm{e}^{-\delta  T}\termC(X_N^{\aprocess})\bigg] \\
	&= \mathbb{E}\bigg[
	\sum_{k=0}^{N-1} \runCDis(k,X_k^{\aprocess},\ud_k) +\termD(X_N^{\aprocess})	\bigg],
	\label{performance_discr}
\end{align}			
where 	for $k=0,\ldots,N-1, x\in\Statespace$ and  $\afix\in \actionspace$
\begin{align}	
	\runCDis(k,x,\afix) &= \mathbb{E}_{k,x}\bigg[\int_{t_k}^{t_{k+1}}  \mathrm{e}^{-\delta s}\runCCont(X^{\overline \uprocess}(s),\afix) \mathrm{d} s \bigg]\quad\text{and}\quad \termD(x)  = \mathrm{e}^{-\delta  T}\termC(x).
	\label{cost_functional}		
\end{align}	
Here,  $\mathbb{E}_{k,x}(\cdot) = \mathbb{E}(\cdot| X_k = x)$ denotes the conditional expectation given that at time $t_k$ the  state is $X_k=x$. 

\paragraph{One-period running costs}
The next lemma shows that the conditional expectation appearing in the definition   of the one-period running costs $	\runCDis(k,x,\afix)$  in \eqref{cost_functional} can be given in closed-form. Hence, the transition from the continuous-time to discrete-time setting does not suffer from additional discretization errors.

\begin{lemma}
	\label{lem:RunningCost-closed}
	The one-period running costs $	\runCDis(k,x,\afix)$	 defined in  \eqref{cost_functional} are given for $k=0,\ldots,N-1$, $x = (r,f,p,y^\top)^\top\in\Statespace$, $\afix \in \actionspace$,  and $\delta >0$ by $\runCDis(k,x,\afix)=\mathrm{e}^{-\delta k\Delta_N}\overline{\runCDis}(k,x,\afix)$ with  
	\begin{align}
		\overline{\runCDis}(k,x,\afix)=\begin{cases}
			\frac{\price_F \mu_{F,k}}{\delta}\big(1-e^{-\delta\Delta_N}\big)+\frac{\price_F  f}{\delta+\beta_F}\big(1-e^{-(\delta+\beta_F) \Delta_N}\big), &\afix=\Fuel,\\[1ex]
			\frac{( \price_{HP}\Pin+ \price_{P})}{\delta}\big(1-e^{-\delta \Delta_N}\big) -
			\price_{HP} \OutputOut \Big\{\big(e^{(A-\delta \mathds{I}_\ell)\Delta_N}-\mathds{I}_\ell\big)\big(A-\delta \mathds{I}_\ell\big)^{-1}y &\\
			+\Big[\big(e^{(A-\delta \mathds{I}_\ell)\Delta_N}-\mathds{I}_\ell\big)\big(A-\delta \mathds{I}_\ell\big)^{-1} -\frac{1}{\delta}\big(1-e^{-\delta \Delta_N}\big)\mathds{I}_\ell\Big]A^{-1}B{g_k(+1)}\Big\}, &\afix=\Charg,\\[1ex]
			\frac{\price_{P}}{\delta}\big(1-e^{-\delta \Delta_N}\big),  & \afix=\Discharg,\\[1ex]
			0,  &  \afix=\{\Wait,\Spill\}.
		\end{cases}
		\label{cost_interm}
	\end{align}
\end{lemma}
The proof of this Lemma is available in Appendix \ref{Cost-closed-f1}.
\begin{remark}
	The non-discounted case,  that is $\delta=0$, is obtained by passing to the limit $\delta\to 0$ and using 	that  $(1-e^{-\delta \Delta_N})/{\delta} \to \Delta_N$. 
\end{remark} 

\subsection{Optimization Problem}
\label{sec:OP}
The aim of the residential heating system's manager is to find an admissible control proces $\aprocess$ defined by the associated decision rule $\amarkov$ that  minimizes the expected total discounted costs arising from the operation of the system and the evaluation of the stored thermal energy in IES and GES at terminal time. These costs have been derived in the above subsection and are given by the  performance criterion $\mathcal{J}_0(x_0;\aprocess)$ in   \eqref{performance_discr}. Thus,  the optimization problem reads
\begin{align}
	\mathcal{V}_0(x_0)=\inf_{\aprocess \in \admiss} \mathcal{J}_0(x_0;\aprocess)
	\label{OP}
\end{align}	
with $ \admiss$ given in \eqref{admiss_control_D}. We call $\mathcal{V}_0(x_0)$ the \textit{value function} of the problem. It represents the minimum expected costs described by the performance criterion.

\subsection{Markov Decision Process}
\label{Sec-Discret-STOPT} 
To solve the optimal control problem \eqref{OP} we will apply the dynamic programming approach using MDP theory. This 
requires to embed \eqref{OP} into a family of optimization problems with variable initial time $n=0,\ldots,N$ and initial state $X_n=x\in\Statespace$. For each of theses problems we define the performance of an admissible control $\aprocess \in \admiss$ and the value function  as
\begin{align}
	J(n,x;\aprocess)& =
	\mathbb{E}_{n,x}\bigg[
	\sum_{k=n}^{N-1}  \runCDis(k,X_k^{\aprocess},\ud_k) +\termD(X_N^{\aprocess})	\bigg] \quad \text{and} \quad V(n,x)=\inf_{\aprocess \in \admiss} J(n,x;\aprocess).
	\label{performance_value_MDP}
\end{align}	
A control $\aprocess^*=(\ud_0^*,\ldots,\ud_{N-1}^*) \in \admiss$ is called optimal control if  $V(n,x)=J(n,x;\aprocess^*)$ for all $n=0,\ldots,N$ and $x\in\Statespace$. The associated decision rule $\amarkov^*=\amarkov^*(n,x)$ defining $\aprocess^*$ is called \textit{optimal decision rule}.

The control problem in \eqref{performance_value_MDP} is a MDP with finite time horizon $N$, an $\ell+3$-dimensional state process $X$ with dynamics given by the recursion  \eqref{state_recursion} which reads $X_{n+1}=  \mathcal{T}(n,X_n,\ud_n,\Noise_{n+1}), ~ X_0=X(0)=x_0, $  It is driven by driven by a sequence of independent three-dimensional standard normally distributed random vectors $(\Noise_n)_{n=1,\ldots,N}$ that appear in the recursion as additive noise. Thus the transition kernel of the MDP describing the conditional distribution of $X_{n+1}$ given $X_n$ and $\ud_n$, is given by a multivariate  Gaussian distribution.  The transition operator $\mathcal{T}$ is linear in the state variable, and the control takes values in the finite  action space  $\actionspace=\{\Spill, \Discharg,\Wait,\Charg,\Fuel\}$ and is subject to state-dependent control constraints described by the family of subsets $\feasible(n,x)\subset \actionspace $ given in \eqref{FeasibleActions}.

\paragraph{Dynamic programming equation}
Solving the control problem in \eqref{performance_value_MDP} is based on the Bellman principle which provides  the following necessary optimality condition called  Bellman equation, and constitutes a backward recursion for the value function and the the optimal decision rule.  
We refer to Bäuerle and Rieder \cite{bauerle2011markov},  Puterman \cite{puterman2014markov}, Hern{\'{a}}ndez-Lerma and Lasserre \cite{HernandezLerma1996}, and the references therein for more details in the MDP theory.
\begin{theorem}(Bellman equation)
	\label{theo-Bellman-equ}
	The value function $V$ satisfies  for all $ x \in \Statespace$
	\begin{align}
		\begin{split}
			V(N,x)& =\termD(x),    \\
			V(n,x)& = \min_{a \in  \feasible(n,x)}\Big\{\runCDis(n,x,a)+\mathbb{E} \big[V(n+1,\mathcal{T}(n,x,\ufix,\Noise_{n+1}))\big]\Big\}, \quad n= N-1,\ldots,0.
		\end{split}				
		\label{Bellman}	
	\end{align}
	%The optimal control at time $n=0,1, \ldots, N-1$ is given by  $\ud^*_n =\amarkov^*(n,X_n^{{\aprocess}^*})$, with 
	The optimal decision rule $\amarkov^*$
	is given by the minimizer in \eqref{Bellman}, that is	
	\begin{align}
		\amarkov^*(n,x)=\underset{a \in  \feasible(n,x)}{\mathrm{argmin}}
		\Big\{\runCDis(n,x,a)+\mathbb{E} \big[V(n+1,\mathcal{T}(n,x,\ufix,\Noise_{n+1}))\big]\Big\}.
		\label{minimizer}
	\end{align}
\end{theorem} 
%For the convenience of te reader the proof of this theorem  is given in Appendix \ref{App-DPE}. 
The dynamic programming equation \eqref{Bellman} can be  solved  by backward recursion starting at the terminal time $N$.  
Further details can be found in \cite[Chapter 6]{takam_PhD_2023}. 
\begin{remark}
	\label{rem:verification}
	Note that the controlled state process is defined by a  linear recursion with additive Gaussian noise. Further, the action space $\feasible$ is finite. Hence, the expectation defining the performance criterion of the finite horizon control problem \eqref{performance_value_MDP} is bounded for all admissible controls $\aprocess \in \admiss$,  and the infimum over $\aprocess$ is attained. Therefore the pointwise optimization problem in the Bellman equation \eqref{Bellman}	is formulated with the minimum.  Further, this also allows the application of a verification theorem as in Bäuerle and Rieder \cite[p. 21-23]{bauerle2011markov}, which ensures that the solution of the Bellman equation is indeed the value function and that $\amarkov^*$ in \eqref{minimizer} is optimal.
\end{remark}

\section{State Discretization}
\label{sec:StateDiscretization}
The challenge of the direct implementation of the backward recursion algorithm following from the Bellman equation in  Theorem \ref{theo-Bellman-equ} is that it becomes computationally intractable due to the curse of dimensionality. On the one hand, the dimension $d=\ell+3$  of the state space is already for moderate dimensions $\ell$ of the reduced-order system \eqref{reduced-S} quite high. Note that the choice of $\ell$  determines the quality of the approximation of aggregated characteristics such as $\QmAppr,\QoutAppr$ appearing in the transition operator and the control constraints. On the other hand, due to the state-dependent control constraints, no closed-form expressions for the expectation $\mathbb{E} \big[V(n+1,\mathcal{T}(n,x,\ufix,\Noise_{n+1}))\big]$, which appear in the Bellman equation \eqref{Bellman}, can be expected.	

To overcome these problems, we propose a computationally tractable approximation of these expectations based on a state discretization, which we describe below. 
For the sake of simplicity, we restrict ourselves to a model with a deterministic fuel price $F$, that is also used in our numerical experiments. Then the state of the control problem is reduced to $X=(R,P,Y^\top)^\top$ with values in the state space  
$\Statespace \subset \R^{\ell+2}$.
A generalization to the full model including $F$ as state is straightforward.

The idea is to divide the state space into a finite number of disjoint subsets, each represented by a single point, in which we want to compute approximations of the value function $V$ and the optimal decision rule $\amarkov$. This approach leads to an approximate MDP with a finite state space and a finite-state Markov chain as state process.  

The construction of the state space partition is based on an  $(\ell+2)$-dimensional grid. We choose for each of the $\ell+2$ state variables  $\dom=R,P, \RState^1,\ldots,\RState^\ell$ discretization parameters $N_\dom\in\N$, sets of indices $\mathcal{N}_{\dom}=\{0,...,N_\dom\}$, and define the points   $r_0< \ldots<r_{N_{R}} \in [\underline r,\underline r]$,  $p_0< \ldots<p_{N_P}^{}$   in $[\underline p,\underline p]$,   $y^k_{0}<\ldots<y^k_{N_{Y^k}}$   in $\R$, $k=1,\ldots,\ell$.
Then, the $(\ell+2)$-dimensional discretized state space is given by the grid
\begin{align} 
	\widehat{\Statespace}
	=\{r_0,\ldots,r_{N_R}^{}\} \times \{p_0,\ldots,p_{N_{P}}^{}\} \times \{y^1_0,\ldots,y^1_{N_{Y^1}}\}\times \ldots \times \{y^\ell_0,\ldots,y^\ell_{N_{Y^\ell}}\}.
\end{align}
Each point $x_m=(r_i,p_j,y^1_{k_1},\ldots,y^\ell_{k_\ell})^\prime \in \widehat{\Statespace}$ is identified  by  the multi-index $m=(i,j,k_1,\ldots,k_\ell)\in \widehat{\mathcal{N}}$ where 
$\widehat{\mathcal{N}}= \mathcal{N}_{R} \times \mathcal{N}_{P} \times\mathcal{N}_{Y^1} \times \ldots\times\mathcal{N}_{Y^\ell}$ is the set of $(\ell+2)$-tuples of multi-indices. 
\begin{figure}[h!]
	\centering
	\includegraphics[width=0.8\linewidth,height=0.5\linewidth]{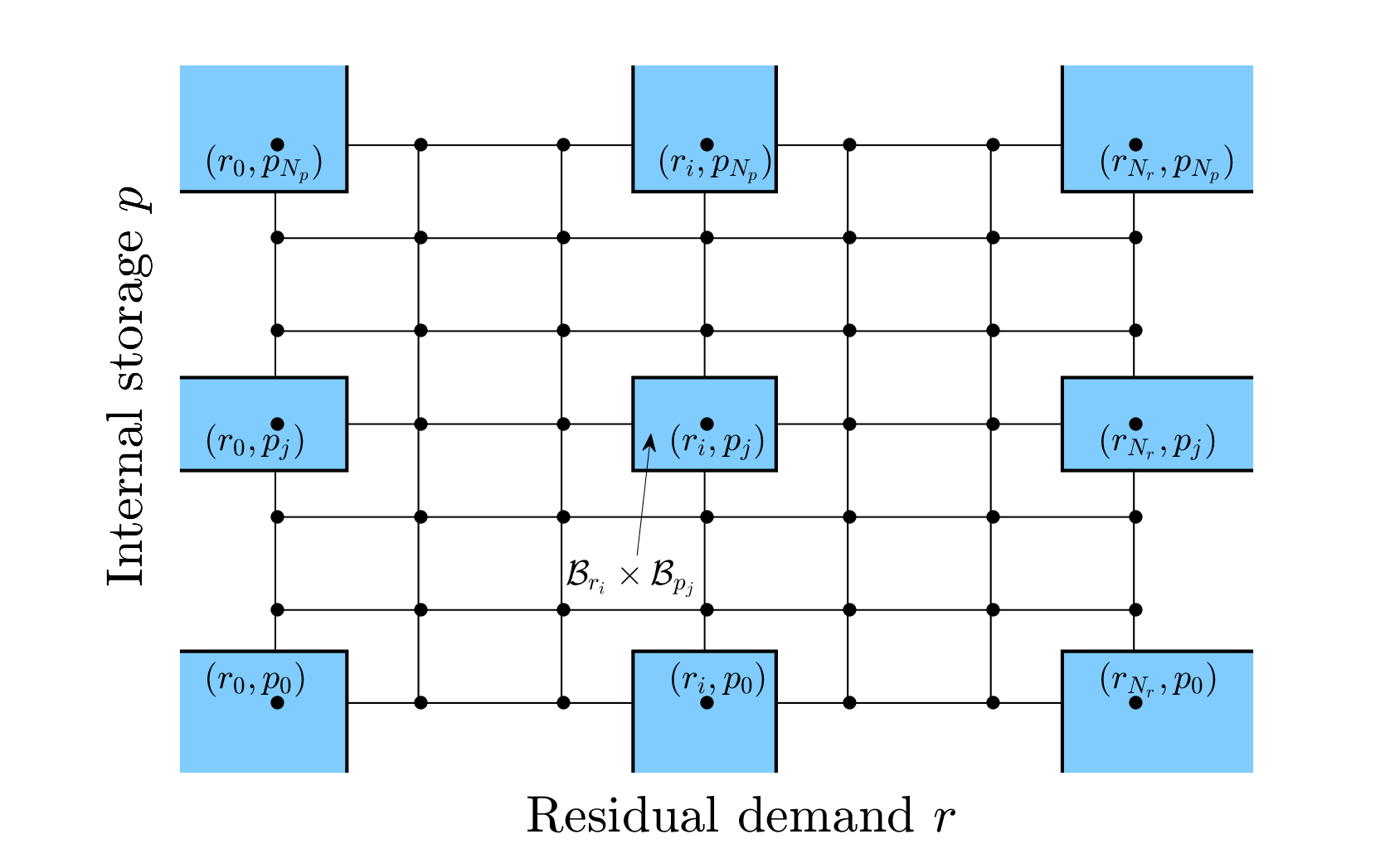}
	\caption{Projection of the  computational grid onto  $(r,p)$-plane $\Statespace_{RP}$  for fixed $\rstate$. }
	\label{Grid-Space}
\end{figure}

For the construction of the discretized state process denoted by $\widehat{X}=(\widehat{R},\widehat{P},\widehat{Y}^\prime)^\prime$ we map all values of the original state $X\in\Statespace$ coordinate-wise to the next grid point. Then, each grid point $x_m$ represents a rectangular neighborhood or box around  $x_m$. For grid points at the boundary these  neighborhoods collect all points of $\Statespace$ beyond the boundary. Figure \ref{Grid-Space} sketches the projection of the  computational grid onto the $(r,p)$-plane $\Statespace_{RP}$ for fixed $\rstate$. 

This approach can be formalized as follows.
We distinguish inner grid points with indices in $\{1,\ldots,N_\dom-1\}$, and boundary grid points  including the corners for which at least one index is $1$ or $N_\dom$, $\dom=R,P, \RState^1,\ldots,\RState^\ell$ . The neighborhoods of inner grid points $r_i,p_j$, and $y^k_{l^k}, k=1,\ldots, \ell$ are defined by 
$\bbox^R_{i}=\big(\frac{1}{2}(r_i+ r_{i-1}),\frac{1}{2}(r_i+ r_{i+1})\big]$,
$\bbox^P_{j}=\big(\frac{1}{2}(p_j+ p_{j-1}),\frac{1}{2}(p_j+ p_{j+1})\big]$,
$\bbox^{Y^k}_{{l_k}}=\big(\frac{1}{2}(y^k_{l^k} + y^k_{l^k-1}),\frac{1}{2}(y^k_{l^k}+ y^k_{l^k+1})\big]$,
for $i=1,\ldots,N_R-1, j=1,\ldots,N_P-1,  l^k=1,\ldots,N_{Y^k}-1, k=1,\ldots,\ell$.
For the boundary grid points, we extend the domain to the  respective boundaries of the original state space $\Statespace$  as follows
\begin{align}	
	\bbox^R_{0}&= \Big(-\infty\frac{1}{2}(r_{0}+ r_{1})\Big], & \bbox^P_{0}&=\Big( -\infty,\frac{1}{2}(p_{0}+ p_{1})\Big],
	&\bbox^{Y^k}_{0}&=\Big(-\infty,\frac{1}{2}(y^k_{0}+ y^k_{1})\Big], 
	\\
	\bbox^R_{N_R}&=\Big(\frac{1}{2}(r_{N_R}^{}+ r_{{N_R}-1}^{}), \infty\Big),& 
	\bbox^P_{N_P}&=\Big(\frac{1}{2}(p_{N_P}^{}+ p_{{N_P}-1}^{}),\infty\Big),& 
	\bbox^{Y^k}_{N_{Y^k}}&=\Big(\frac{1}{2}(y^k_{N_{Y^k}}+ y^k_{N_{Y^k}-1}),+\infty\Big).
\end{align}
The projection of the original state $X_n=(R_n,P_n,Y_n^\top)^\top\in \Statespace$ onto the discretized state   space $\widehat\Statespace$ results in   $\widehat{X}_n=(\widehat{R}_n,\widehat{P}_n,\widehat{Y}_n^\prime)^\prime$  defined  for $n=0,\ldots,N$ by
\begin{align}
	\label{state_projection}
	\widehat{R}_n = \sum_{i=0}^{N_R} r_i \one_{\bbox^R_i}(R_n),\quad 
	\widehat{P}_n = \sum_{j=0}^{N_P} p_j \one_{\bbox^P_j}(P_n),\quad 
	\widehat{Y}^k_n = \sum_{l=0}^{N_{Y^k}} y_{l}^k \one_{\bbox^{Y^k}_l}(Y^k_n),\quad k=1\ldots,\ell.
\end{align}

This discretization converts the given MDP with state space $\Statespace$ into a MDP for a controlled finite-state Markov chain with state space $\widehat{\Statespace}$. Combining the recursion \eqref{state_recursion} for the original state process with the projection formulas in 	\eqref{state_projection}, the dynamics of the the discrete-state approximation $(\widehat X_n)$ is given by the recursion 
\begin{align}
	\label{state_recursion_discrete}
	\widehat X_{n+1}= \widehat{\mathcal{T}} (n,\widehat X_n,\ud_n,\Noise_{n+1}), \quad \widehat X_0=\widehat x_0. 
\end{align} 
for $n=0,\ldots N-1$. Here,    $\widehat{\mathcal{T}}:\{0,\ldots,N-1\}\times \widehat{\Statespace}\times\actionspace\times \R^3 \to \widehat{\Statespace} $ 
is  the discrete-state transition operator. It is for $n\in \{0,\ldots,N-1\}$, $x=(r,p,y^\top)^\top\in\widehat{\Statespace},\afix\in \actionspace$, $\noise\in \R^3$  defined as 	$\widehat{\mathcal{T}}=(\widehat{\mathcal{T}}^R,\widehat{\mathcal{T}}^P,(\widehat{\mathcal{T}}^Y)^\top)^\top$ with 
\begin{align}
	\label{state_recursion2_discrete}
	\begin{split}
		\widehat{\mathcal{T}}^R(n,x,\afix,b)& = \sum_{i=0}^{N_R} r_i \one_{\bbox^R_i}(\mathcal{T}^R(n,x,\afix,b)),\quad
		\widehat{\mathcal{T}}^P(n,x,\afix,b) = \sum_{j=0}^{N_P} p_j \one_{\bbox^P_j}(\mathcal{T}^P(n,x,\afix,b)),\\
		\widehat{\mathcal{T}}^Y(n,x,\afix,b)&= \widetilde{\mathcal{T}}^Y(n,y,\afix) =\sum_{l=0}^{N_{Y^k}} y_{l}^k \one_{\bbox^{Y^k}_l}(\overline{\mathcal{T}}^Y(n,y,\afix)),\quad k=1\ldots,\ell,
	\end{split}
\end{align}
where $\mathcal{T}^R,\mathcal{T}^P, \overline{\mathcal{T}}^Y$ are given in \eqref{state_recursion2}, and $\widehat x_0$ is the projection of the initial state $x_0$ onto $\widehat{\mathcal{X}}$. Note that  the discrete-state approximation $(\widehat X_n)$ inherits the  Markov property from the original state process $(X_n)$. This allows to characterize the transition kernel of the approximate MDP, that describes the conditional distribution of $\widehat X_{n+1}$ given $\widehat X_n$ and $\ud_n$, by the transition probabilities of the Markov chain. They are   given for all multi-indices $m_1,m_2\in \widehat{\mathcal{N}}$ by 
\begin{align}
	\Ptrans^{\afix}_{m_1m_2}=\mathbb{P}(\widehat{X}_{n+1}^{\aprocess}=x_{m_2} ~\mid~\widehat{X}_n^{\aprocess}=x_{m_1},~\ud_n=\afix) = \mathbb{P}(\widehat{\mathcal{T}} (n,x_{m_1},\afix,\Noise_{n+1})=x_{m_2}),
\end{align}
which  are the probabilities that the state moves from $x_{m_1}$ at time $n$ to  $x_{m_2}$ at time $n+1$ under the action $\ud_n=\afix$. For the computation of these probabilities one can utilize that the conditional distribution of the pair  of projected state components  $(\widehat R_{n+1},\widehat P_{n+1})$, given the state $\widehat X_n$ and the decision rule $\ud_n$ at time $n$, follows from the bivariate Gaussian distribution  of the original state variables $(R_{n+1},P_{n+1})$. Further, the corresponding conditional distribution of $\widehat \RState_{n+1}$ is degenerate, since   $Y$ and thus also its projection $\widehat{Y}$ follows  deterministic dynamics.   To be more specific,	
suppose the two multi-indices are given for $\dom=1,2$ as $m_\dom=(i_\dom,j_\dom,l^1_\dom,\ldots,l^\ell_\dom)\in \widehat{\mathcal{N}} $, and denote by $y_{m_\dom}=(y^1_{l^1_\dom},\ldots,y^\ell_{l^\ell_\dom})^\prime$. Then it holds 

\begin{align*}
	\Ptrans^{\afix}_{m_1m_2}
	&=\mathbb{P}\big(\widehat{\mathcal{T}} (n,x_{m_1},\afix,\Noise_{n+1})
	= x_{m_2}=(r_{i_2},p_{j_2}, y^1_{l^1_2},\ldots,y^\ell_{l^\ell_2})^\prime\big) \\	
	&=\mathbb{P}\big((\widehat{\mathcal{T}}^R (n,x_{m_1},\afix,\Noise_{n+1}),
	\widehat{\mathcal{T}}^P (n,x_{m_1},\afix,\Noise_{n+1}))^\prime = (r_{i_2},p_{j_2})\big)~
	\mathbb{P}\big(\widetilde{\mathcal{T}}^Y(n,y_{m_1},\afix)=y_{m_2}\big)\\
	&=\mathbb{P}((\mathcal{T}^R(n,x_{m_1},\afix,\Noise_{n+1}), \mathcal{T}^P(n,x_{m_1},\afix,\Noise_{n+1})) \in \bbox_{r_{i_2}}\times \bbox_{p_{j_2}})~   
	\mathbb{P}\Big(\overline{\mathcal{T}}^Y(n,y_{m_1},\afix)\in \prod_{k=1}^\ell\bbox^{Y^k}_{l^k_2}\Big),
\end{align*}
For more details on the  computation of the first probability in the last line,  we refer to  \cite[Chapter 6]{takam_PhD_2023}.	
For the last probability, the deterministic dynamics of $\RState$ and therefore of $\widehat{\RState}$ yields  
\begin{align}
	\mathbb{P}\big(\overline{\mathcal{T}}^Y(n,y_{m_1},\afix)\in \prod_{k=1}^\ell\bbox^{Y^k}_{l^k_2}) =
	\prod_{k=1}^\ell	\one_{\bbox^{Y^k}_{l^k_2}}\Big(\overline{\mathcal{T}}_k^Y(n,y_{m_1},\afix)\Big). 
\end{align}
We are now able to construct approximations of the value function $V(n,x)$ and the optimal decision rule $\amarkov(n,x)$ of the original MDP in the points of the discretized state space $\widehat{\mathcal{X}}$. These approximations are denoted accordingly by $\widehat V(n,x)$ and $\amarkovdiscrete(n,x)$. They are obtained by solving  the Bellman equation for the approximate MDP.  Analogous to the result given in Theorem 	\ref{theo-Bellman-equ}, this leads to the following backward recursion for all grid points $x_m \in \widehat{\mathcal{X}}$ identified  by an multi-index $m\in \widehat{\mathcal{N}}$. 
\begin{align}
	\begin{split}
		\widehat V(N,x_m) &=\termD(x_m),    \\
		\widehat V (n,x_m) & =\min_{a \in  \feasible(n,x_m)} \Big\{\runCDis(n,x,a)+\mathbb{E} \big[\widehat V(n+1,\widehat{\mathcal{T}}(n,x_m,\ufix,\Noise_{n+1}))\big]\Big\}, \quad n=N-1,\ldots,0.
	\end{split}				
	\label{Bellman_discrete}	
\end{align}
The optimal decision rule $\amarkovdiscrete^*(n,x_m)$
is given by the minimizer in \eqref{Bellman_discrete}, and the expectation in \eqref{Bellman_discrete}  by ~
$	\mathbb{E} \big[\widehat V(n+1,\widehat{\mathcal{T}}(n,x_{m},\ufix,\Noise_{n+1}))\big] 
=\sum\limits_{x_{q} \in  \widehat{\Statespace}} \Ptrans^{\afix}_{m q}\widehat{V}(n+1,x_{q}).$

%		\begin{align}
	%			\mathbb{E} \big[\widehat V(n+1,\widehat{\mathcal{T}}(n,x_{m_1},\ufix,\Noise_{n+1}))\big] %&=\sum_{x_{m_2} \in  \widehat{\Statespace}}\mathbb{P}(\widehat{X}_{n+1}^{\aprocess}=x_{m_2} ~\mid~\widehat{X}_n^{\aprocess}=x_{m_1},~\ud_n=\afix) \widehat{V}(n+1,x_{m_2})\\ &
	%			=\sum_{x_{m_2} \in  \widehat{\Statespace}} \Ptrans^{\afix}_{m_1m_2}\widehat{V}(n+1,x_{m_2}).
	%		\end{align}

\section{Numerical Results}
\label{OPt-Num-rsult}
In this section, we present numerical results that illustrate the solution of the stochastic optimal control problem for the cost-optimal management of a residential heating system. In particular, we show approximations of the value function and the optimal decision rules obtained by solving the Bellman equation \eqref{Bellman_discrete} for the approximate MDP that we studied in Section \ref{sec:StateDiscretization}.

In order to keep the curse of dimensionality under control, we use for the approximation of the GES spatio-temporal temperature distribution $Q$  by a reduced-order system of ODEs of dimension $\ell=4$ and two output variables, representing the average temperature approximations in the storage medium $\QmAppr$ and the average temperature in the \phx fluid $\QfAppr$. The latter serves as approximation of the average outlet temperature $\QoutAppr$, which appears in the transition operator of the MDP. We consider a GES with one \phxk, for which we have found that the 4-dimensional reduced-order system already provides quite good approximations for the two aggregated characteristics  mentioned above. We also found  that the naive approximation $\QfAppr\approx\QoutAppr$ can be significantly improved using the following formula to reconstruct $\QoutAppr$ from $\QfAppr$
\begin{align}
	\QoutAppr(t)\approx f(\QfAppr(t),a)\quad  \text{with} \quad 
	f(q,a)&=\begin{cases}
		2q-\QinC & \afix=-1,\\
		q+\Delta_{HP}/2, &  \afix=+1,\\
		q & \text{otherwise}.
	\end{cases}
\end{align}
It results from the assumption of a perfect linear increase or decrease of the liquid temperature along the way through the \phx from the inlet to the outlet leading to the relation $\QfAppr=(\Qin+\QoutAppr)/2$ with the inlet temperatures $\Qin=\QinC$ and $\QoutAppr -\Delta_{HP}$ for $a=-1,+1$, and $, \Qin=\QoutAppr$ otherwise. Note that including the outlet temperature $\QoutAppr$ directly in the system output requires considerably  larger dimensions $\ell$ of the reduced-order system, which then prevents a computationally tractable solution to the optimization problem. 	

Our numerical experiments are based on a model with a constant and known fuel price $F(t)= f_0>0$. This is the typical case for small and medium sized heating systems operating with fixed tariffs for fuel purchase. This makes it possible to remove the fuel price $F$ from the state variables and reduce the state dimension by one, which in turn again helps to keep the curse of dimensionality under control.  We work with  a planning horizon of $T=3$ days, which is divided into $N=72$ periods of lenght $\Delta_N=1$ hour. Since we consider  a short-term simulation with a small horizon time, we   neglect discounting as well as  the seasonality effect of the residual demand and suppose $\mu_R(t)=\mu^R_0=\textrm{const}$.    

\begin{table}[h]
	\centering	
	{\footnotesize 
		\begin{tabular}[T]{|c|rl||c|rl|} \hline
			Parameters&Values& Units& Parameters&Values& Units\\
			\hline		
			$l_x,l_y,l_z$ &$10,1,10$&$m$ &$T$ &   $72$&$ h$\\
			$h_x,h_y$&$0.1,~~0.01$&$m$ & $N$ & $72$ &\\				
			$ h_P$ & $0.02$&$ m$ & $\Delta_N$ & $1$ &$h$ \\
			$n_P$& $1 $&& $\Pin,\Pout$&40,~30&$\Celsius$ \\\hline 
			$\rhom $ & $2000$ &$kg/m^3$&$\rhof $ &$998$ &$kg/m^3$\\
			$\cpm$ & $800$& $J/(kg\, K)$ & $\cpf=\cpw$&$4182$ & $J/(kg\, K)$\\
			$\kappam $ &$1.59$ & $   J/(s\,  m\,  K)$ & $\kappaf$ &$0.60$ & $   J/( \,m\,  K)$\\
			$ d^M $ & $9.9375 \times 10^{-7}$&$m^2/s$ & $ d^F $& $1.4376\times 10^{-7}$&$m^2/s$\\\hline
			$\vconst$ & $ 10^{-2}$&$m/s$ &   $\QinC$ & $40 $&$\Celsius$\\
			$\heattransfer$ & $10$&$   J/(s\, m^2~ K)$& $\HPspread$ &  $3 $&$K$ \\
			$Q_0$ &  $10 $  &$\Celsius$ & 	$\Qg$ & $15$&$ \Celsius$\\
			$  \sigma_R$  &$232.5$ & $   J/\sqrt{s^3}$&$  \mu_R^0$  &$-4.64 \times 10^3 $ & $   J/s$ \\  
			$\beta_R$  &$0.5$ & $1/h$&$ A^P$  &$9.096$ & $m^2$ \\ 
			$m^P$  &$4000$ & $kg$ & $\kappa^P$  &$12$ & $  J/(s\, m^2\, K)$ \\
			$m^Q$  &$200000$ & $kg$&  $\price_F$&$30$& $l/ h$\\
			$p_\text{ref}$  &$60$ & $\Celsius$&   $\price_{HP}$&3&\texteuro/ $K$\\
			$q_\text{ref}$  &$20$ & $\Celsius$ & $\price_{P}$&$5$& \texteuro/$h$\\
			$\Pamb$  &$20$ & $\Celsius$&  $ \price^P_\text{pen}, \price^Q_\text{pen}$&   $6.7, ~~0.45$& \texteuro/ $kWh$ \\
			$f_0$&$2.25$& \texteuro/$l$&$ \price^P_\text{liq}= \price^Q_\text{liq}$& $0$&\texteuro/ $kWh$\\
			$\epsilon$& $0.05$&&$\overline{p}$& $90$&$\Celsius$ \\
			$\gamma$  &$3.27 \times 10^{-6}$ &   $1/s$&$\underline{p}$& $30$&$\Celsius$ \\
			$L_C$&$ 1.66 \times 10^{3}$&$ J/(K\, s)$&$\underline{q}$& $10$&$\Celsius$ \\
			$L_F$&$7.436 \times 10^{4}$&$  J/s$&$\overline{q}$& $30$&$\Celsius$ \\
			$L_D$&$1.39 \times 10^{3}$&$   J/(K\, s)$&$\underline{r}$&$-16.7$&$ MJ$\\
			$\delta$&0&&$\overline{r}$&$13.4 $&$MJ$\\
			\hline
		\end{tabular}
	}
	%\\[2ex]	
	\caption{Model and discretization parameters}		
	\label{tab:cap1}	
\end{table}

For the computation of approximations of the value function and the optimal decision rule, we use the backward recursion following from  Bellman equation \eqref{Bellman_discrete} for the approximate MDP. 
The state space is discretized with $N_R=8$, $N_P=11$, $ N_{Y^1}=N_{Y^2}=N_{Y^3}=4$, $N_{Y^4}=8$.  The finer discretization of the range of $Y^4$ is motivated by the fact, that we use for the four-dimensional state space $Y$ a coordinate system which is such that $\OutputM \RState= cY^4$ for some constant $c\in \R$. Thus, $Y^4$ is  proportional  to the  average temperature in the storage medium $\QmAppr \OutputM \RState$. Recall, the latter plays a crucial role in the construction of the set  of feasible controls $	\feasible^Y(n,x)$, see \eqref{control_constr_Y}.
The GES storage medium is selected as dry soil, the \phx fluid is water, the \phx height is $  h_P=0.02\; m$.  The GES is a  cuboid with edge lengths $\l_x=10\; m,  l_y =1\;m, l_z = 10 \;m$, the spatial discretization of the heat equation  on the two-dimensional cross-section uses step sizes  $h_x=0.1\;m, h_y=0.01\; m$. The heat pump spread is set to $3K$, as it is often observed in reality. A penalty is chosen for the final cost. It is incurred if the IES and GES are not filled properly. No liquidation revenues are realized for leftover energy. 
The other model parameters are given in Table~\ref{tab:cap1}.

In the following we present results for full and empty storage, characterized by $\QmAppr_N=\overline{q}=30\Celsius$ and $\QmAppr_N=\underline{q}=10\Celsius$, respectively. Furthermore, results are shown for an intermediate grid point for the reduced order state $\Rstate$ which corresponds to the  average  temperature $\OutputM \Rstate=18\Celsius$ of the GES medium,  and is close to the midpoint $(\underline{q}+\overline{q})/2$ and therefore  referred to as ``half full GES''.

\paragraph{Terminal cost function}
The left panel of Fig.~\ref{VfT72_emptyY} shows the terminal cost function $\termC(x)$ as a function  of $(\widetilde{r},p)$.  for an empty, half full   and full GES. Here, $\widetilde{r}=\mu_R^0 +r$ denotes the residual demand including the seasonality component. The right panel of shows the dependence of the terminal costs on the storage level in IES and GES at terminal time, that is on $P_N=p$ and $\QmAppr_N= \overline q^{\medium} = \OutputM y$.
\begin{figure}[h!]
	\centering
	\includegraphics[width=0.49\linewidth,height=0.3\linewidth]{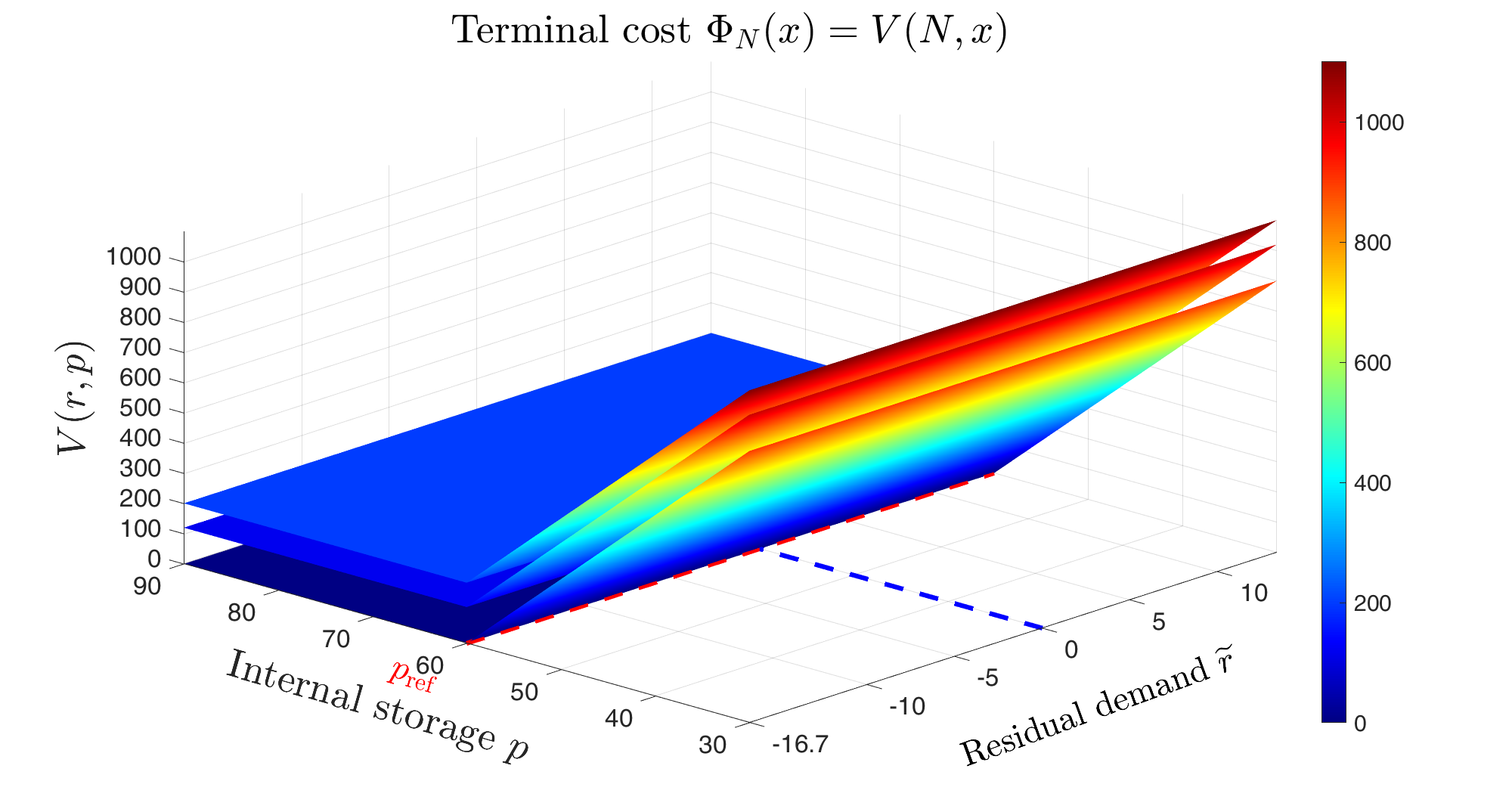}
	\includegraphics[width=0.49\linewidth,height=0.3\linewidth]{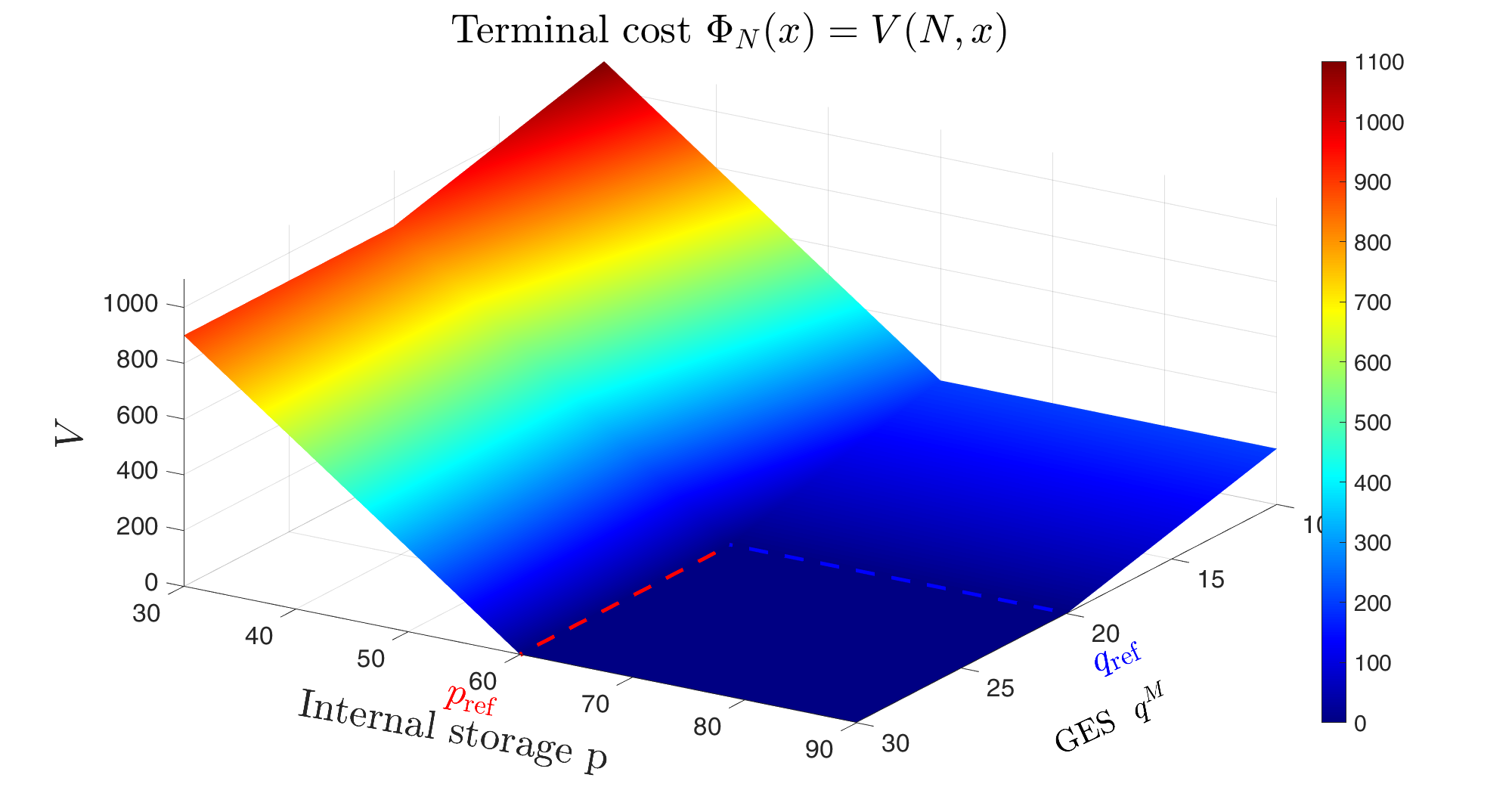}
	\caption[Terminal cost function]
	{Terminal cost function $\termD(x)=V(N,x)$.\\ Left: depending on $(\widetilde{r},p)$ for  empty,  half full  and   full  GES (upper, middle, lower graph).\\Right: depending on $p$ and average GES temperature $\OutputM y$.}		
	\label{VfT72_emptyY}
\end{figure}
This figure shows that the terminal cost function is zero when the average IES and GES temperature are both above the threshold ($P_N\geq p_\text{{ref}}=60~\Celsius$ and $\QmAppr_N \geq q_\text{{ref}}= 20~\Celsius$). However, it begins to increase linearly when these temperatures fall below their respective threshold values.

\paragraph{Value function and optimal decision rule at time $\mathbf{n=N-1}$}
Fig.~\ref{Vf71_RPY30} shows results at the beginning of the last period starting at time $T-\Delta t = 71 h$.
The top left panel depicts the approximations of the value function $V(n,x)$ as a function of $(\widetilde{r},p)$ for a full GES,  a half full GES, and an empty GES. 
The other panels  show the approximate optimal decision rules $\amarkov(n,x)$ as  functions of $(\widetilde{r},p)$,  for a full GES,  a half full GES, and an empty GES in  separate panels.

\begin{figure}[h!]
	\centering
	\includegraphics[width=0.49\linewidth,height=0.32\linewidth]{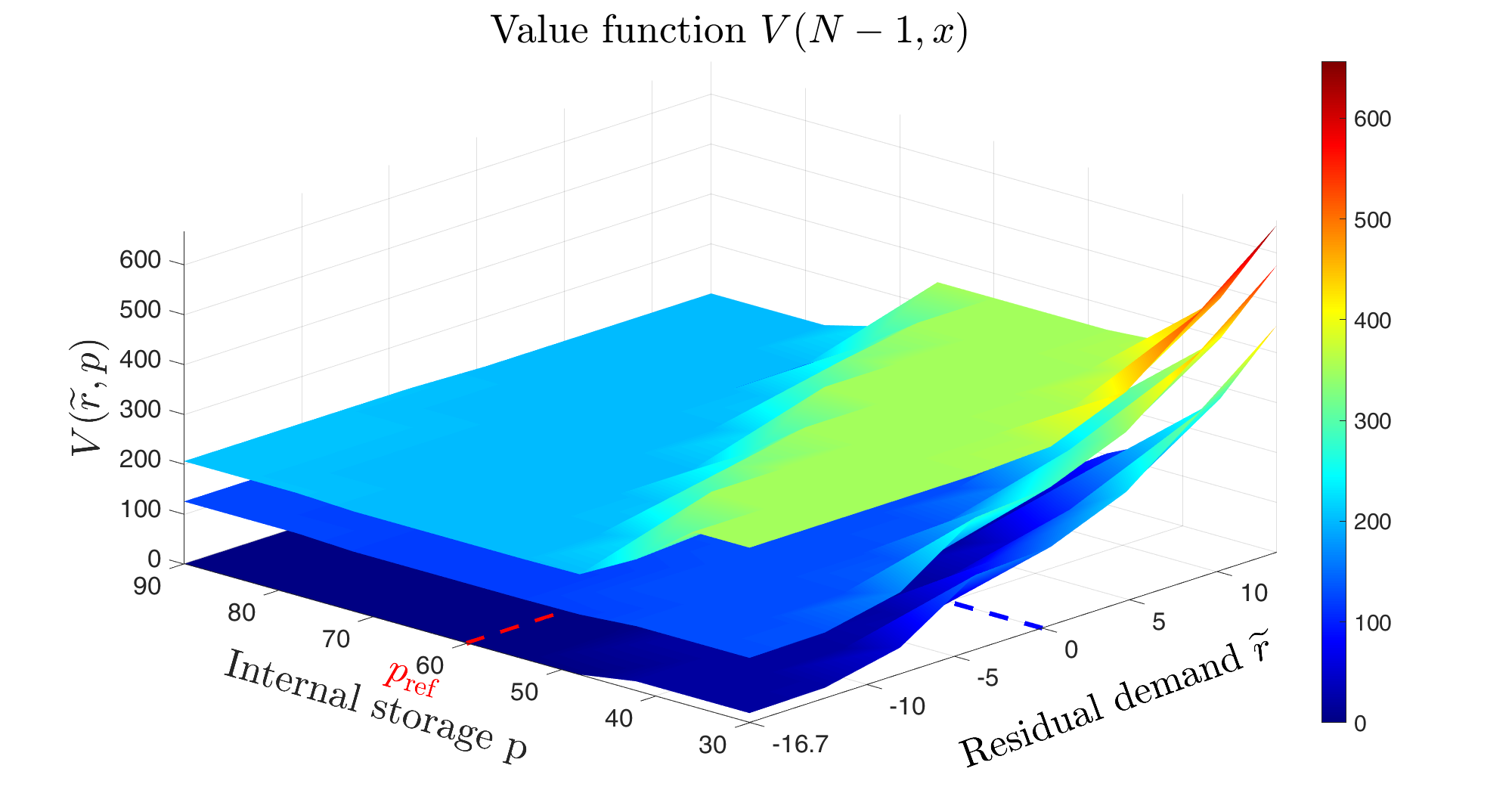}
	\includegraphics[width=0.49\linewidth,height=0.32\linewidth]{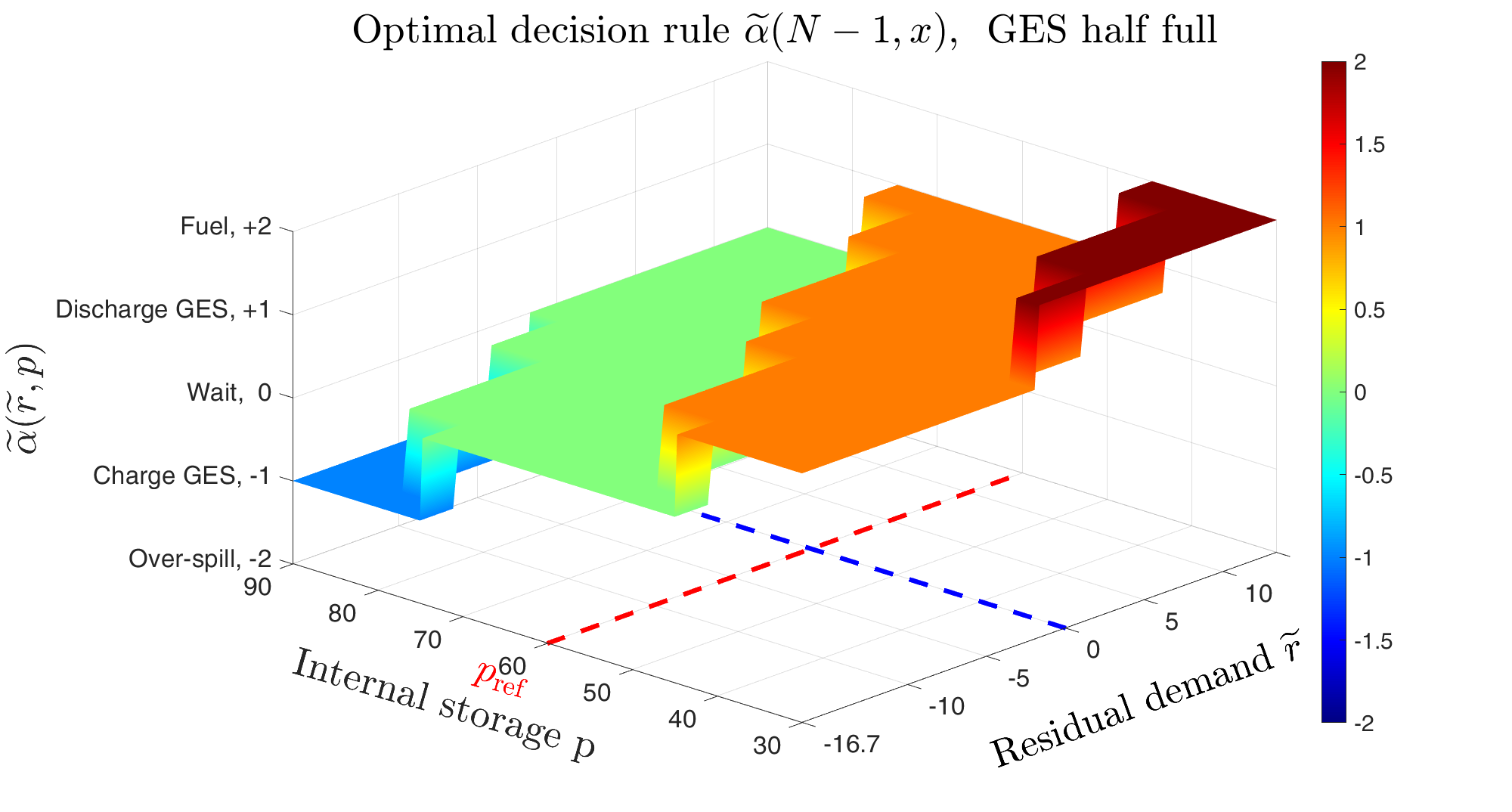}
	\includegraphics[width=0.49\linewidth,height=0.32\linewidth]{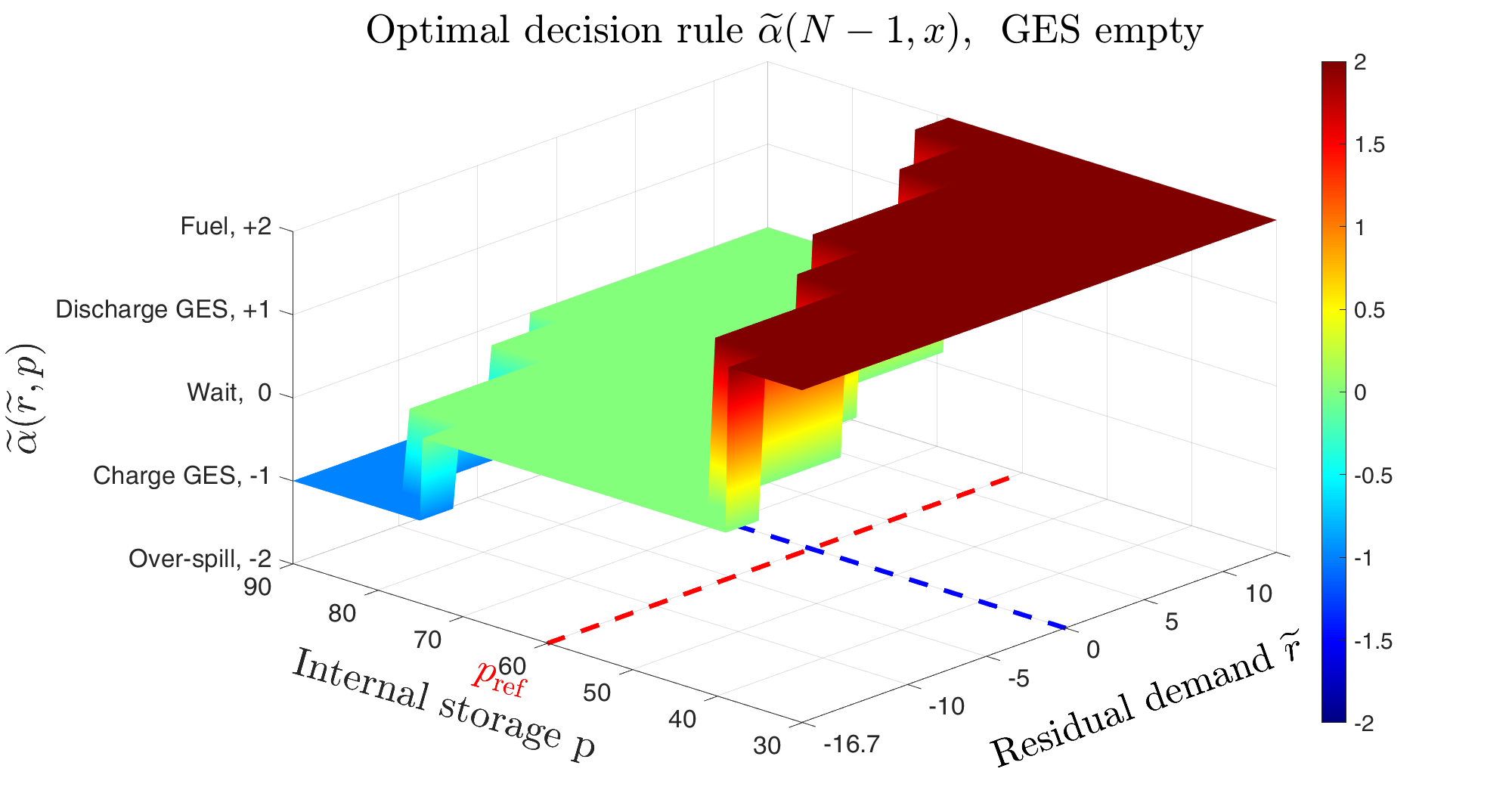}
	\includegraphics[width=0.49\linewidth,height=0.32\linewidth]{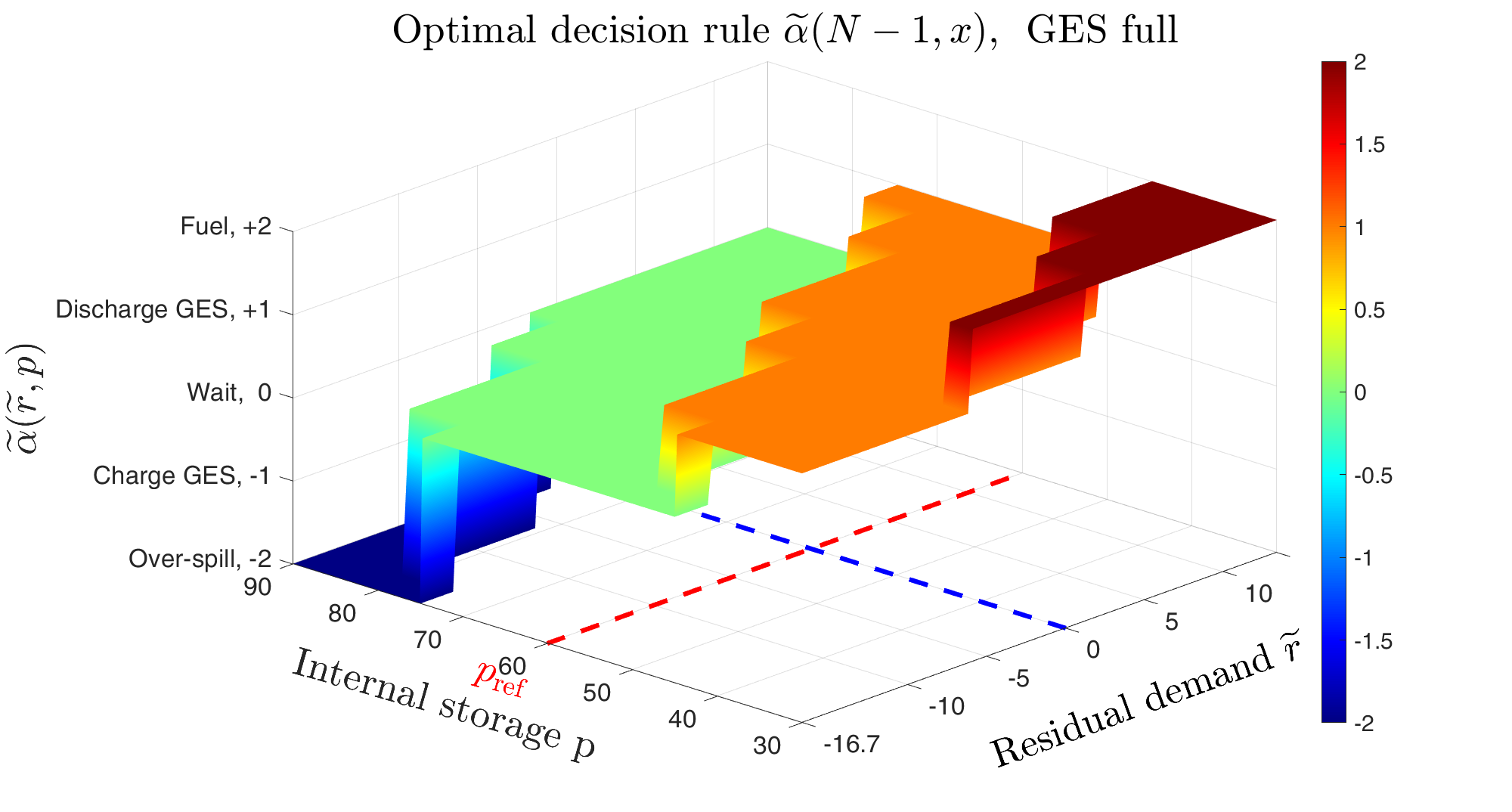}
	\caption[Value function and optimal decision rule at time $N-1$]{
		Value function $V(n,x)$ and optimal decision rule $\amarkov^*(n,x)$ at time $n=N-1$ as a function of $(\widetilde{r},p)$ for different GES filling levels. 
		\\ Top left: Value function for  empty,  half full  and   full  GES (upper, middle, lower graph).\\
		Bottom left, top right, bottom right: Optimal decision rules for  empty,  half full  and   full  GES.}
	\label{Vf71_RPY30}
\end{figure}

For $n=N-1$ the Bellman equation \eqref{Bellman_discrete} yields for the value function 
$\widehat V (N-1,x) =\runCDis(n,x,\afix^*)+\mathbb{E} \big[ \termD(\widehat{\mathcal{T}}(N-1,x,\afix^*,\Noise_{n+1}))\big]$  where $\afix^*$ denotes the optimal action. This decomposition into the running costs in the last period and the expected terminal costs helps to  explain the properties of the value function. 	
It can be observed that the value function is almost constant when  the IES temperature exceeds the penalty threshold $p_\text{ref}=60\,\Celsius$, and the residual demand is negative (overproduction). 
In that case,  it is optimal to wait ($\afix=0$), to transfer thermal energy to the GES ($\afix=-1$), or even to  apply over-spilling when the IES is almost full and there is a large negative residual demand, which indicates a strong overproduction of heat. For these controls  no, or only very small further running costs are incurred in the last period. Therefore, the  value function  represents the expected terminal penalties to be paid if the average GES temperature is below $q_\text{ref}$. The latter does not depend on $(\widetilde{r},p)$, but only on the reduced-order state $y$ which explains the different constant levels of the the value function.

\begin{figure}[ht]
	\centering
	\includegraphics[width=0.49\linewidth,height=0.32\linewidth]{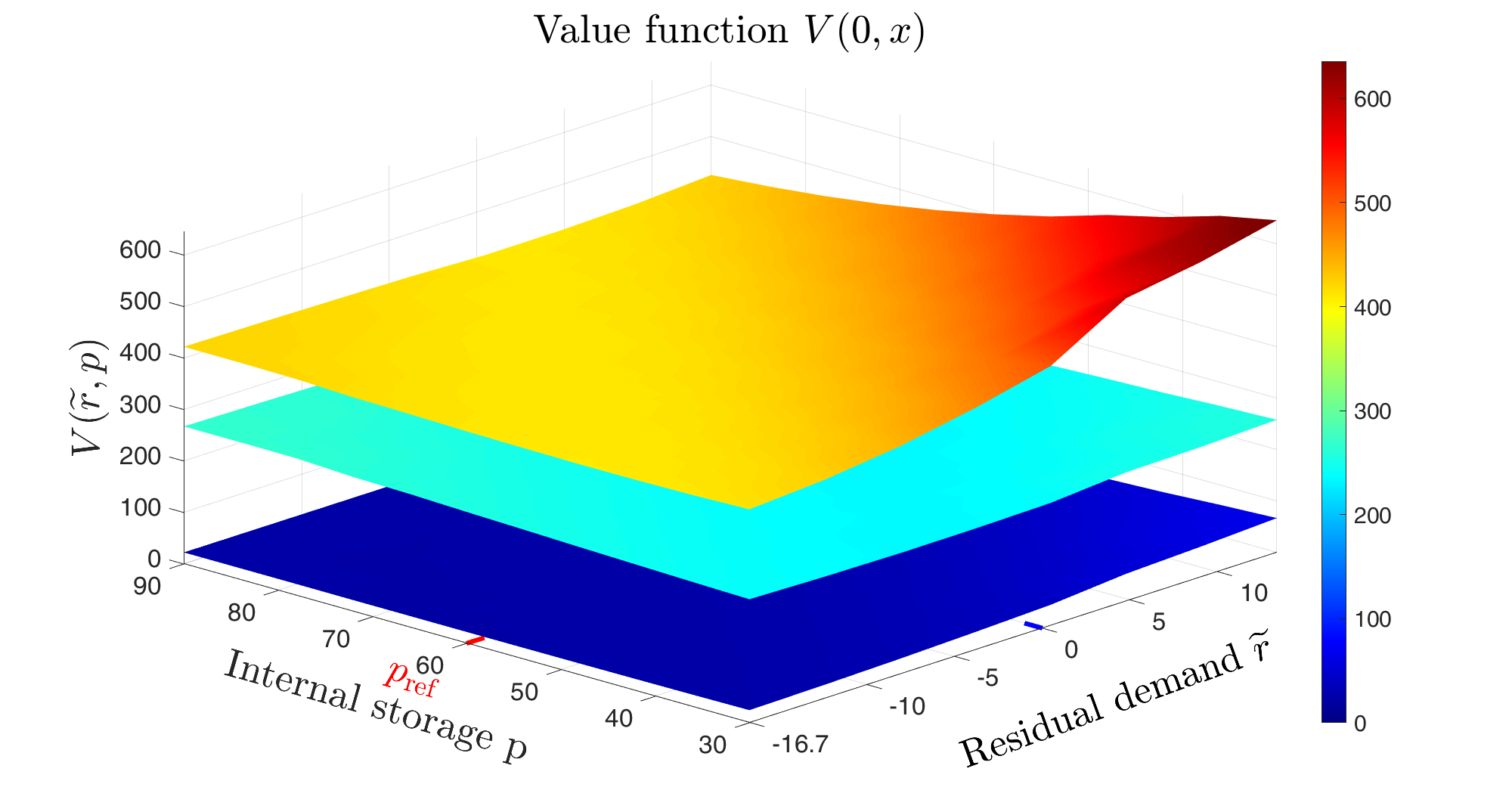}
	\includegraphics[width=0.49\linewidth,height=0.32\linewidth]{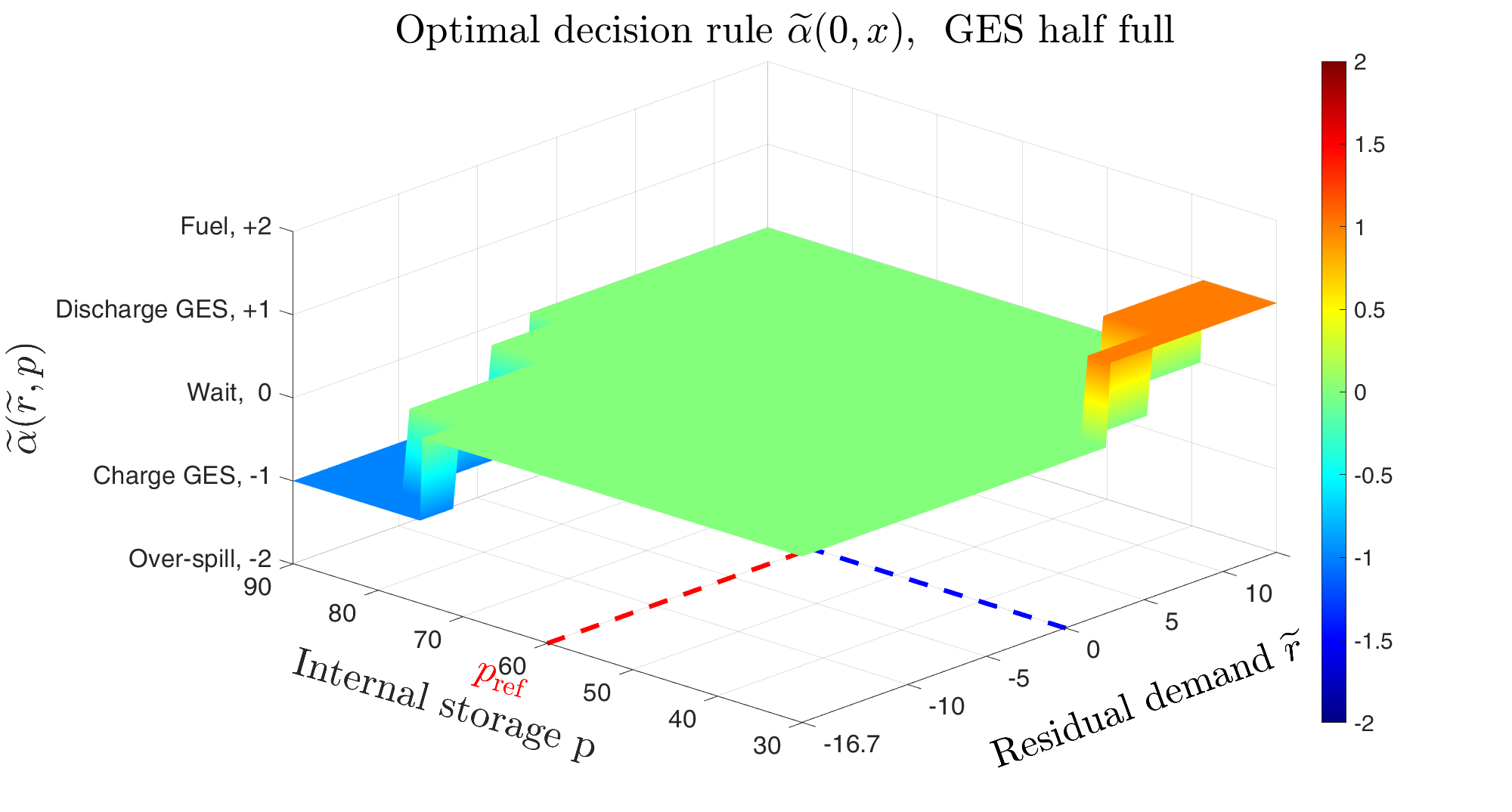}
	\includegraphics[width=0.49\linewidth,height=0.32\linewidth]{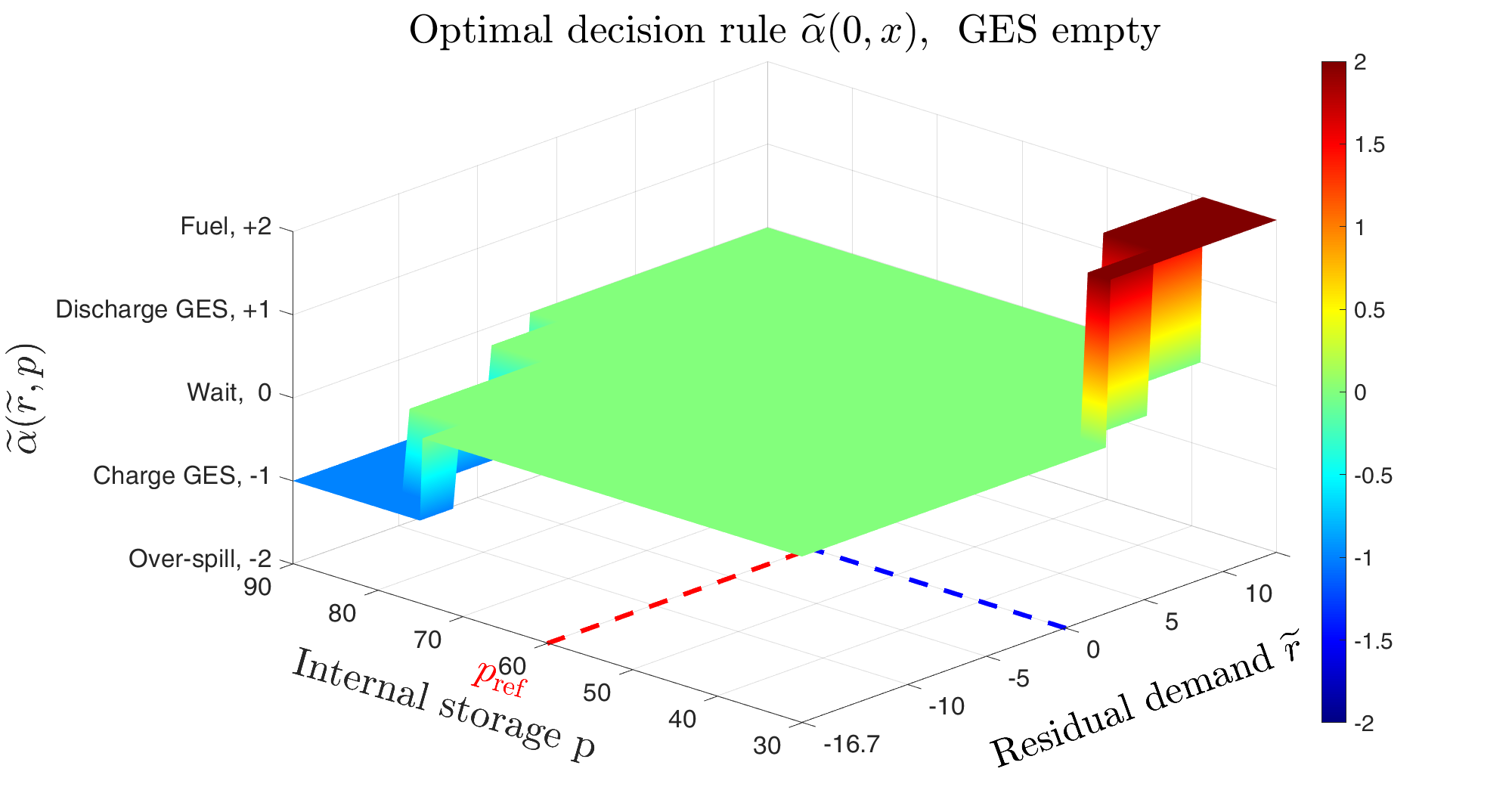}
	\includegraphics[width=0.49\linewidth,height=0.32\linewidth]{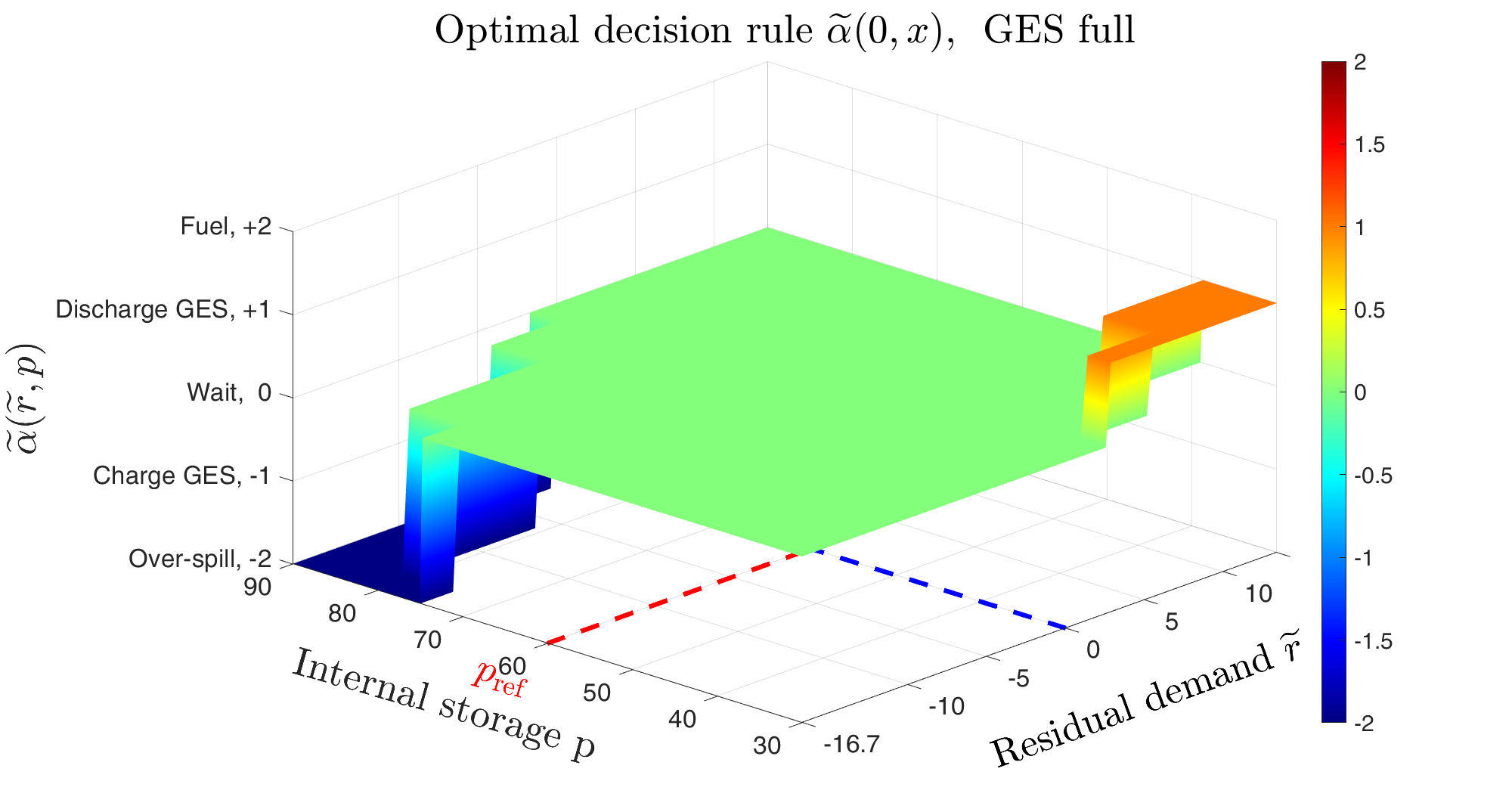}
	\caption[Value function and optimal decision rule at initial time $n=0$]{ Value function $V(n,x)$ and optimal decision rule $\amarkov^*(n,x)$ at initial time $n=0$ as a function of $(\widetilde{r},p)$ for different GES filling levels. 
		\\ Top left: Value function for  empty,  half full  and   full  GES (upper, middle, lower graph).\\
		Bottom left, top right, bottom right: Optimal decision rules  for empty,  half full  and   full  GES.}
	\label{Vf0_RPY30}
\end{figure}
When the residual demand becomes positive and when the IES temperature falls below $p_\text{ref}$ then it is optimal to charge the IES in order to avoid penalty payments at the end of the period if the IES temperature remains below $p_\text{ref}$.	
Considering the chosen parameters for the operating costs, burning fuel ($\afix=+2$) to charge the IES is more expensive than transferring heat from the GES to the IES with the heat pump ($\afix=+1$). Therefore, as expected, it is optimal to use the heat pump to charge the IES when its temperature is moderately  low,  especially when the residual demand is positive, which means that the heat supply does not cover the demand. Only when the IES is almost completely empty or the residual demand is very high, it is optimal to choose the expensive option of fuel firing. This also explains that the value function increases with decreasing IES temperature $p$ and increasing residual demand $\widetilde{r}$. The strong growth of the value function with strong unsatisfied demand and almost empty IES results from the penalties that must be paid if the IES temperature cannot be raised to $p_\text{ref}$ within the last period. Note that when the GES is empty then firing fuel is the only available option to charge the IES, as can be seen in the top right-hand panel.

\paragraph{Value function and optimal decision rule at initial time $\mathbf{n=0}$}
Fig.~\ref{Vf0_RPY30} shows the approximate value function and optimal decision rules  as functions of $(\widetilde{r},p)$ at the initial time $t=0~h$ for a full GES,  a half full GES, and an empty GES. In contrast to the results for $n=N-1$, the value functions shown in the top left-hand box now take on larger values for IES temperatures above $p_\text{ref}$, as there are $72$ periods ahead in which there is a potential imbalance between supply and demand and therefore ongoing costs could be incurred.  On the other hand, the maximum values which are attained for very small IES temperatures are smaller now, since penalities at terminal time can be avoided by an active storage management.
The increase with decreasing IES temperature $p$ is much less pronounced and visible only for positive residual demand.
In contrast, the value function increases visibly with increasing residual demand $\widetilde{r}$, especially for very high unsatisfied demand and low IES temperatures. This is due to the fact that in this case the transfer of heat from the GES to the IES with the help of the heat pump or even by firing  fuel in the next, or one of the next periods, is optimal and causes corresponding running costs.  The optimal decision rules show that it is only not optimal to wait when the IES is almost full or almost empty. In these cases, heat exchange with the GES is optimal unless the GES is full or empty. With a full GES, the only feasible and therefore optimal action is over-spilling, while with an empty GES, firing  fuel is the only option.

\paragraph{Optimal Paths of the State Process}
% starting with  almost empty IES and GES
%	\label{Opt-paths} 
\begin{figure}[h!]
	\centering
	\includegraphics[width=0.95\linewidth,height=0.40\linewidth]{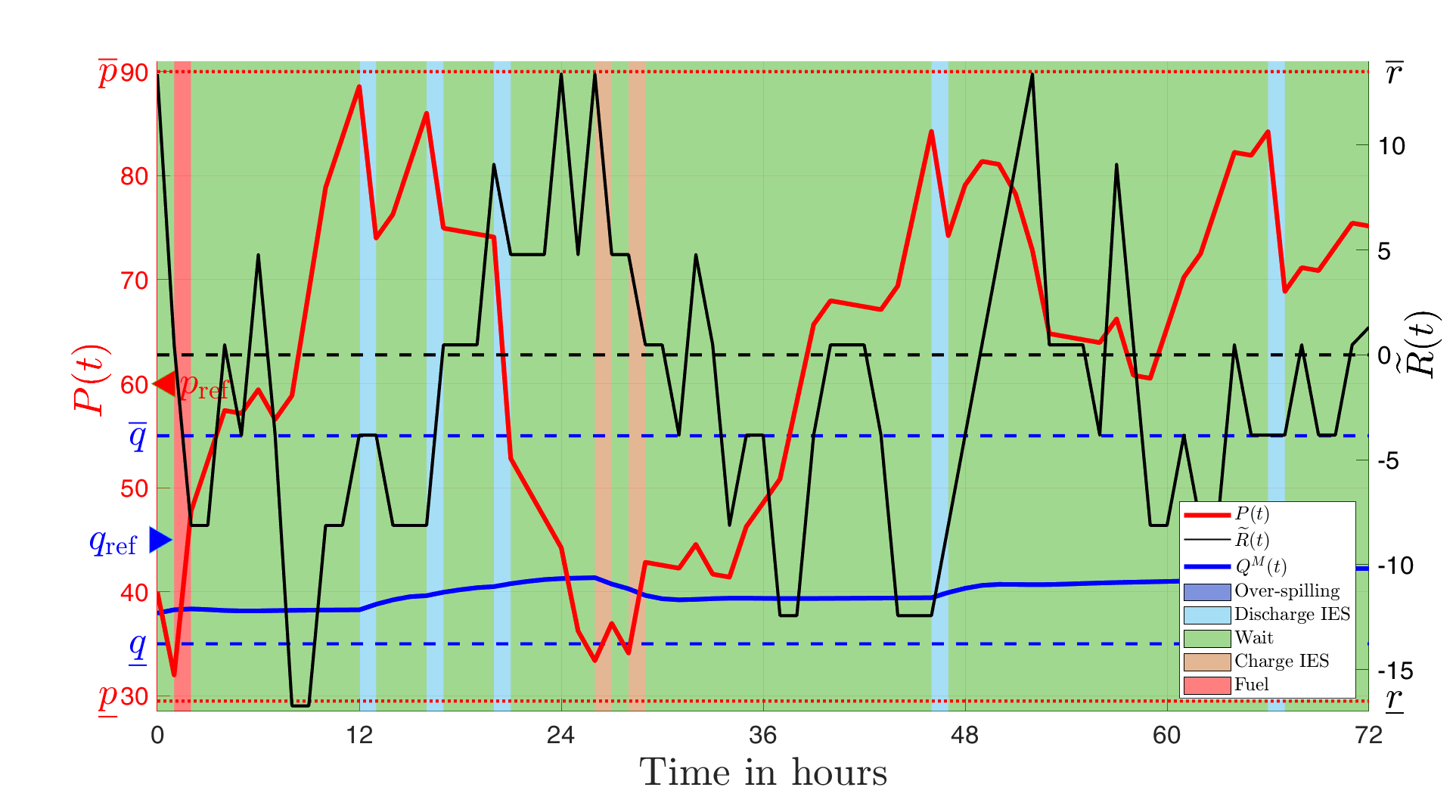}
	\caption[Paths of state variables for almost empty IES and GES]{  Optimal paths of IES temperature $P$ (red), GES average temperature $\QmAppr$ (blue) for initially almost empty IES $(P_0=40~\Celsius)$ and GES $(\QmAppr_0=12~\Celsius)$. 
		Red dashed lines at $\underline{p}=30\Celsius$ and $\overline{p}=90\Celsius$ represent state constraints for $P$, as well as the truncation values for the residual demand $\widetilde{R}$ at $\mu_R^0+\underline{r}=-16.7$ and $\mu_R^0+\overline{r}=13.4$. The path of $\widetilde{R}$ (black) starts at $\widetilde{R}(0)=\overline{r}$. The zero level of $\widetilde{R}$ is shown by a black dashed line. 		
		Blue dashed lines depict the state constraint $\QmAppr\in [\underline{q}, \overline{q}]$.  
		The background colors represent the optimal controls as shown in the legend.
	}
	\label{AvTFullP_emptyY}
\end{figure}
Fig.~\ref{AvTFullP_emptyY}  shows optimal paths of the average temperatures in the IES and GES together with the residual demand when we start with  almost empty IES and  GES.  Then in the first periods  it is  optimal to wait when there is overproduction $(\widetilde{R}<0)$ and to fire fuel when there is unsatisfied demand $(\widetilde{R}>0)$.  After the initial phase, with an almost empty IES and for small positive residual demand, it is optimal to discharge the GES and do not fire fuel until it is not almost  empty. When the IES is almost full and the residual demand is negative, it is optimal to transfer heat from the IES to the GES.

\section{Summary and Outlook}
\label{sec:summary}
We have investigated a stochastic optimal control problem for the cost-optimal management of a residential heating system  under stochastic supply-demand imbalance and stochastic energy prices.  As a special feature, the heating system is equipped with a geothermal storage  with a large capacity, whose response to charging and discharging decisions, however, depends on the spatial temperature distribution in the storage. The  dynamics of the  latter is described by a heat equation with a convection term that models the fluid movement in the \phx.  This leads to a challenging, non-standard mathematical optimization problem.  We first performed a semi-discretization of the heat equation and used model reduction techniques to reduce the dimension of the associated  high-dimensional system of ODEs. Finally, a time discretization leads to a MDP, which was finally approximated by an MDP with a discretized state space. In this way, the curse of dimensionality  is  overcome and the optimal control problem is solved numerically with a reasonable computational effort. 

We are currently working on more realistic mathematical models of the geothermal storage taking into account  the three-dimensional spatial temperature distribution in the storage and a refined topology of the \phxsk.     
However, this will then lead to considerably higher dimensions of the state space of the control problem, and require more efficient approximation techniques to solve the control problem, such as optimal quantization, see Pages et al.~\cite{pages2004optimal}  and reinforcement learning methods   such  as Q-learning, see Powell \cite{powell2007approximate} and Sutton and Barto \cite{sutton2018reinforcement}.  For some  first results in this direction, we refer to 
Pilling et al.~\cite{Pilling_arxiv2024}.

\bigskip
\begin{small}
	\smallskip\noindent\textbf{Acknowledgments~}
	The authors thank    Martin Redmann (Martin-Luther University Halle--Wittenberg),  Olivier Menoukeu Pamen (University of Liverpool), Gerd Wachsmuth,   Eric Pilling and Andreas Künnemann  (BTU Cottbus--Senftenberg)
	for valuable discussions that improved this paper.

	\smallskip\noindent
	\textbf{Funding~} P.H.~Takam gratefully acknowledges the  support by the German Academic Exchange Service (DAAD), award number 57417894.	
	R.~Wunderlich gratefully acknowledges the  support by the Federal Ministry of Education and Research (BMBF),  award number	05M2022.

\end{small}

%	% Reduce spacing
%	\let\oldbibliography\thebibliography
%	\renewcommand{\thebibliography}[1]{%
	%		\oldbibliography{#1}%
	%		\setlength{\itemsep}{-0.9ex plus .05ex}%
	%	}	
%
%	%\bibliographystyle{plainnatmod}
%	%\bibliographystyle{amsplain}
%	%\bibliographystyle{plain}
\bibliographystyle{acm}  
%\bibliography{TakamDissert}

\begin{thebibliography}{10}
	
	\bibitem{bahr2022efficient}
	{\sc B{\"a}hr, M., and Breu{\ss}, M.}
	\newblock Efficient long-term simulation of the heat equation with application
	in geothermal energy storage.
	\newblock {\em Mathematics 10}, 13 (2022), 2309.
	
	\bibitem{bahr2017fast}
	{\sc B{\"a}hr, M., Breu{\ss}, M., and Wunderlich, R.}
	\newblock Fast explicit diffiusion for long-time integration of parabolic
	problems.
	\newblock In {\em International Conference of Numerical Analysis and Applied
		Mathematics (ICNAAM 2016)\/} (2017), vol.~1863, p.~410002.
	
	\bibitem{bauerle2011markov}
	{\sc B{\"a}uerle, N., and Rieder, U.}
	\newblock {\em {M}arkov {D}ecision {P}rocesses {W}ith {A}pplications to
		{F}inance}.
	\newblock Springer Science \& Business Media, 2011.
	
	\bibitem{bauerle2016gas}
	{\sc B{\"a}uerle, N., and Riess, V.}
	\newblock Gas storage valuation with regime switching.
	\newblock {\em Energy Systems 7}, 3 (2016), 499--528.
	
	\bibitem{bazri2022thermal}
	{\sc Bazri, S., Badruddin, I.~A., Usmani, A.~Y., Khan, S.~A., Kamangar, S.,
		Naghavi, M.~S., Mallah, A.~R., and Abdelrazek, A.~H.}
	\newblock Thermal hysteresis analysis of finned-heat-pipe-assisted latent heat
	thermal energy storage application for solar water heater system.
	\newblock {\em Case Studies in Thermal Engineering 40\/} (2022), 102490.
	
	\bibitem{chen2008semi}
	{\sc Chen, Z., and Forsyth, P.~A.}
	\newblock A semi-{L}agrangian approach for natural gas storage valuation and
	optimal operation.
	\newblock {\em SIAM Journal on Scientific Computing 30}, 1 (2007), 339--368.
	
	\bibitem{chen2010implications}
	{\sc Chen, Z., and Forsyth, P.~A.}
	\newblock Implications of a regime-switching model on natural gas storage
	valuation and optimal operation.
	\newblock {\em Quantitative Finance 10}, 2 (2010), 159--176.
	
	\bibitem{dincer2021thermal}
	{\sc Dincer, I., and Rosen, M.~A.}
	\newblock {\em {T}hermal {E}nergy {S}torage: {S}ystems and {A}pplications}.
	\newblock John Wiley \& Sons, 2021.
	
	\bibitem{ehsan2019scenario}
	{\sc Ehsan, A., and Yang, Q.}
	\newblock Scenario-based investment planning of isolated multi-energy
	microgrids considering electricity, heating and cooling demand.
	\newblock {\em Applied energy 235\/} (2019), 1277--1288.
	
	\bibitem{fleming2006control}
	{\sc Fleming, W.~H., and Soner, H.~M.}
	\newblock {\em {C}ontrolled {M}arkov {P}rocesses and {V}iscosity {S}olutions},
	vol.~25.
	\newblock Springer Science and Business Media, 2006.
	
	\bibitem{some2025prosumers}
	{\sc Ganet~Somé, M.}
	\newblock Stochastic optimal control of prosumers in a district heating system.
	\newblock {\em arXiv.2501.09088\/} (2025).
	
	\bibitem{gu2014modeling}
	{\sc Gu, W., Wu, Z., Bo, R., Liu, W., Zhou, G., Chen, W., and Wu, Z.}
	\newblock Modeling, planning and optimal energy management of combined cooling,
	heating and power microgrid: A review.
	\newblock {\em International Journal of Electrical Power \& Energy Systems
		54\/} (2014), 26--37.
	
	\bibitem{guelpa2019thermal}
	{\sc Guelpa, E., and Verda, V.}
	\newblock Thermal energy storage in district heating and cooling systems: A
	review.
	\newblock {\em Applied Energy 252\/} (2019), 113474.
	
	\bibitem{HernandezLerma1996}
	{\sc Hernández-Lerma, O., and Lasserre, J.~B.}
	\newblock {\em {D}iscrete-{T}ime {M}arkov {C}ontrol {P}rocesses}.
	\newblock Springer New York, 1996.
	
	\bibitem{KITAPBAYEV2015823}
	{\sc Kitapbayev, Y., Moriarty, J., and Mancarella, P.}
	\newblock Stochastic control and real options valuation of thermal
	storage-enabled demand response from flexible district energy systems.
	\newblock {\em Applied Energy 137\/} (2015), 823 -- 831.
	
	\bibitem{kuang2019stochastic}
	{\sc Kuang, J., Zhang, C., and Sun, B.}
	\newblock Stochastic dynamic solution for off-design operation optimization of
	combined cooling, heating, and power systems with energy storage.
	\newblock {\em Applied Thermal Engineering 163\/} (2019), 114356.
	
	\bibitem{major2018numerical}
	{\sc Major, M., Poulsen, S.~E., and Balling, N.}
	\newblock A numerical investigation of combined heat storage and extraction in
	deep geothermal reservoirs.
	\newblock {\em Geothermal Energy 6}, 1 (2018), 1--16.
	
	\bibitem{oksendal2019stochastic}
	{\sc Oksendal, B., and Sulem, A.}
	\newblock {\em {A}pplied {S}tochastic {C}ontrol of {J}ump {D}iffusions}.
	\newblock Springer International Publishing, 2019.
	
	\bibitem{pages2004optimal}
	{\sc Pages, G., Pham, H., and Printems, J.}
	\newblock An optimal {M}arkovian quantization algorithm for multi-dimensional
	stochastic control problems.
	\newblock {\em Stochastics and dynamics 4}, 04 (2004), 501--545.
	
	\bibitem{pham2009c}
	{\sc Pham, H.}
	\newblock {\em {C}ontinuous-{T}ime {S}tochastic {C}ontrol and {O}ptimization
		{W}ith {F}inancial {A}pplications}, vol.~61.
	\newblock Springer Science and Business Media, 2009.
	
	\bibitem{Pilling_arxiv2024}
	{\sc Pilling, E., Bähr, M., and Wunderlich, R.}
	\newblock Stochastic optimal control of an industrial power-to-heat system with
	high-temperature heat pump and thermal energy storage.
	\newblock {\em arXiv.2411.02211\/} (2024).
	
	\bibitem{powell2007approximate}
	{\sc Powell, W.~B.}
	\newblock {\em {A}pproximate {D}ynamic {P}rogramming: {S}olving the {C}urses of
		{D}imensionality}.
	\newblock Wiley, Aug. 2011.
	
	\bibitem{puterman2014markov}
	{\sc Puterman, M.~L.}
	\newblock {\em {M}arkov {D}ecision {P}rocesses: {D}iscrete {S}tochastic
		{D}ynamic {P}rogramming}.
	\newblock John Wiley \& Sons, 2014.
	
	\bibitem{Regnier_et_al_2022}
	{\sc Regnier, G., Salinas, P., Jacquemyn, C., and Jackson, M.}
	\newblock Numerical simulation of aquifer thermal energy storage using
	surface-based geologic modelling and dynamic mesh optimisation.
	\newblock {\em Hydrogeology Journal 30}, 4 (2022), 1179--1198.
	
	\bibitem{shardin2017partially}
	{\sc Shardin, A.~A., and Wunderlich, R.}
	\newblock Partially observable stochastic optimal control problems for an
	energy storage.
	\newblock {\em Stochastics 89}, 1 (2017), 280--310.
	
	\bibitem{soltani2019comprehensive}
	{\sc Soltani, M., Moradi~Kashkooli, F., Dehghani-Sanij, A., Nokhosteen, A.,
		Ahmadi-Joughi, A., Gharali, K., Mahbaz, S., and Dusseault, M.}
	\newblock A comprehensive review of geothermal energy evolution and
	development.
	\newblock {\em International Journal of Green Energy 16}, 13 (2019), 971--1009.
	
	\bibitem{sutton2018reinforcement}
	{\sc Sutton, R.~S., and Barto, A.~G.}
	\newblock {\em {R}einforcement {L}earning: {A}n {I}ntroduction}.
	\newblock MIT press, 2018.
	
	\bibitem{takam_PhD_2023}
	{\sc Takam, P.~H.}
	\newblock {\em Stochastic optimal control problems of residential heating
		systems with a geothermal energy storage}.
	\newblock PhD thesis, BTU Cottbus-Senftenberg, 2023. \url{
		https://doi.org/10.26127/BTUOpen-6450}.
	
	\bibitem{TakamWunderlich_redu2024}
	{\sc Takam, P.~H., and Wunderlich, R.}
	\newblock Model order reduction for the input–output behavior of a geothermal
	energy storage.
	\newblock {\em Journal of Engineering Mathematics\/} (2024).
	
	\bibitem{takam2025energies}
	{\sc Takam, P.~H., and Wunderlich, R.}
	\newblock Numerical simulation of the input-output behavior of a geothermal
	energy storage.
	\newblock {\em Under revision\/} (2025).
	
	\bibitem{TakamWunderlichPamen2023}
	{\sc Takam, P.~H., Wunderlich, R., and Pamen, O.~M.}
	\newblock Modeling and simulation of the input-output behavior of a geothermal
	energy storage.
	\newblock {\em Mathematical Methods in the Applied Sciences\/} (2023).
	
	\bibitem{testi2020stochastic}
	{\sc Testi, D., Urbanucci, L., Giola, C., Schito, E., and Conti, P.}
	\newblock Stochastic optimal integration of decentralized heat pumps in a smart
	thermal and electric micro-grid.
	\newblock {\em Energy Conversion and Management 210\/} (2020), 112734.
	
	\bibitem{ware2013accurate}
	{\sc Ware, A.}
	\newblock Accurate semi-{L}agrangian time stepping for stochastic optimal
	control problems with application to the valuation of natural gas storage.
	\newblock {\em SIAM Journal on Financial Mathematics 4}, 1 (2013), 427--451.
	
	\bibitem{zayed2023recent}
	{\sc Zayed, M.~E., Shboul, B., Yin, H., Zhao, J., and Zayed, A.~A.}
	\newblock Recent advances in geothermal energy reservoirs modeling: Challenges
	and potential of thermo-fluid integrated models for reservoir heat extraction
	and geothermal energy piles design.
	\newblock {\em Journal of Energy Storage 62\/} (2023), 106835.
	
	\bibitem{zhong2021distributed}
	{\sc Zhong, J., Tan, Y., Li, Y., Cao, Y., Peng, Y., Zeng, Z., Nakanishi, Y.,
		and Zhou, Y.}
	\newblock Distributed operation for integrated electricity and heat system with
	hybrid stochastic/robust optimization.
	\newblock {\em International Journal of Electrical Power \& Energy Systems
		128\/} (2021), 106680.
	
\end{thebibliography}
%\addcontentsline{toc}{section}{References}

%\newpage
\appendix
\small
\section{List of Notations}
\label{Notations}
\renewcommand{\arraystretch}{0.8}	

\vspace*{-4ex}
\begin{longtable}{p{0.3\textwidth}p{0.68\textwidth}l}		
	$t,T$ & time,  horizon time &\\
	$R, \widetilde{R}$ & residual demand: deseaonalized/including seasonality\\
	$F, \widetilde{F}$ & fuel price: deseaonalized/including seasonality\\
	$P$ &average  temperature in the IES&\\				
	$Q=Q(t,x,y)$ & temperature in the GES &\\	
	$l_x$,~$l_y,l_z$ &width, height and depth of the storage &\\
	$h_P$  & \phx height &\\
	$\mathcal{D} =(0, l_x) \times (0,l_y)$ &two-dimensional GES domain  &\\
	$\Dm, ~\Df, \DInterface$ &   domain  of storage medium,  \phx fluid, interface in between  &\\		
	$\partial \mathcal{D}$ &boundary of GES domain &\\
	$\partial \Din$,~$\partial \Dout$ & inlet and outlet boundaries of \phx&\\
	$\partial \Dleft, \partial \Dright, \partial \Dtop$,~$\partial \Dbottom$  & left, right, top and bottom boundaries of the domain&\\		
	$v=v_0(t)(v^x, v^y)^{\top}$ & time-dependent velocity vector&\\
	$\vconst$ &  constant velocity during pumping&\\
	$\cpm,\cpf,\cpw$ & specific heat capacity   of the storage fluid, \phx fluid, IES fluid &\\
	$\rhof$,~$\rhom$ & mass density of the fluid and medium&\\
	$\kappaf$,~$\kappam$ &thermal conductivity of the fluid and medium&\\
	$\af$,~$\am$& thermal diffusivity of the fluid and medium&\\	
	%\end{tabular}
	%
	%\begin{tabular}{lll}
	$\heattransfer$ & heat transfer coefficient between storage and  underground &\\
	%$\lambda >0$ & heat flux&\\
	$Q_0, \Qg$ &initial and underground temperature distribution of the GES &\\
	$\QinC$ &inlet    temperature of the \phx  during charging &\\
	$\Qm, \Qf, \Qout$ & average temperature in  GES storage medium, \phx fluid, GES outlet&\\
	$\QmAppr, \QfAppr, \QoutAppr$ &  respective  approximations from reduced-order system&\\
	$\HPspread$ &  heat pump spread &\\				
	$ \Pin,\Pout,\Pamb$ & IES   inflow, outflow and ambient temperature  &\\
	$\gamma$&rate of IES heat loss to the environment&\\	
	$\drift_{\dom}$ & drift coefficient $\dom=P,Y$\\
	$\mat{A}, \mat{B}, \mat{\omatrix}$ & $ n \times n$  system matrix, $n \times m$  input matrix, $\nO \times  n$ output matrix & \\
	$Y, Z$&state and output    of reduced-order system&\\
	$\dimred, \nO$& dimension of state $Y$ and  output $Z$ \\
	$g$& input variable of the system&\\						
	$\mathbb{P}, \mathbb{E}$ & probability measure,  expectation  &\\
	$W_{R},W_{F}$ & Wiener processes&\\
	$\mu_{R/F}$, $\beta_{R/F}$, $\sigma_{R/F}$ & mean reversion level, speed and volatility for residual demand/ fuel &\\		 
	$N, \Delta_N$& number of time intervals, time step size&\\ 		
	$\Statespace,~\widehat{\Statespace}$ &state space of  continuous and  discretetized  state &\\
	$X=X(t)=(R,F,P,\RState^\top)^\top$ &continuous-time state process at time $t$&\\
	$X_n=(R_n,F_n,P_n,\RState_n^\top)^\top  $&  discrete-time state process at  time $t_n$&\\
	$\widehat{X}_n=(\widehat{R}_n,\widehat{F}_n,\widehat{P}_n,\widehat{\RState}_n^\top)^\top  $& discrete-time state process on discretized state space &\\	
	$\Noise=(\Noise_n)_{n=1,\ldots,N}$ &sequence of independent Gaussian  random  vectors&\\
	$m_{\dom}(n),~\Sigma^2_{\dom}(n)$ & conditional mean and variance of the random variables $\dom_{n+1}$, with $\dom=R,F$&\\
	$\Sigma_{P}^2(n,\afix), \rho(n,\afix)$&conditional ovariance,  correlation coefficient of $R_{n+1}$ and $P_{n+1}$&\\
	$\mathcal{T}=(\mathcal{T}^R,\mathcal{T}^F,\mathcal{T}^P,(\mathcal{T}^Y)^\top)^\top$&transition operator&\\			 		
	$\uprocess=(\uprocess(t))_{t\in[0,T]}$&   continuous-time  control process&\\
	$\feasible(n,x),  \feasible_P(n,x), \feasible_Y(n,x)$& state-dependent  set of feasible actions 	&\\					
	$\aprocess=(\ud_1,\ldots,\ud_{N-1}), \amarkov(n,x)$&  discrete-time  control process, decision rule&\\
	$\admiss$ & set of admissible controls &\\
	$\price_\dom$ & price, price rates \\
	$\runCCont,\runCDis~\termC,\termD$& continuous-time/ discrete-time running and terminal cost &\\
	$J({n},x;u),~V({n},x) $ &performance criterion and value function at time $n$&\\
	$	\mathcal{N}_\dom=\{0,1,...,N_\dom\}$ & set of indices for  states $\dom= R, P,\RState^1,\ldots,\RState^\ell$&\\
	$\mathcal{N}_{R} \times \mathcal{N}_{P} \times\mathcal{N}_{ Y^1} \times \ldots\times\mathcal{N}_{ Y^\ell}$& set of multi-indices for points of  discretized space $\widehat{X}$&\\
	$x_m=(r_i,p_j,y^1_{k_1},\ldots,y^\ell_{k_\ell})$&point of  discretized space $\widehat{X}$&\\	
	$\Ptrans^{\afix}_{{m_1},{m_2}}$& transition probability &\\ 
	$\bbox_{\dom_i}$&  neighborhood of $\dom_i,~i=0,\ldots,N_\dom$ with  $\dom=r,p, y^k$&\\
	$\mathds{I}_n$& $n$-dimensional identity matrix\\
	\textbf{Abbreviations} &\\	
	\phx & Pipe heat exchanger\\
	MDP & Markov decision process\\
	IES, GES  & Internal, Geothermal  storage\\
	ODE, PDE, SDE & Ordinary, partial, stochastic  differential equation\\
\end{longtable}

\section{Details on the Transition Operator}
\label{app:trans_op_details}
The  functions  $\Sigma_{R}(n),\Sigma_{F}(n),\Sigma_{P}(n,\afix),\rho(n,\afix)$ and   $\mathcal{H}$  appearing in the transition operator in \eqref{state_recursion2} of Proposition \ref{prop:state_recursion} are given by
\begin{align}
	\Sigma_{\dom}^2 (n)& =\frac{\sigma^2_{\dom,n}}{2 \beta_\dom}(1-\mathrm{e}^{-2\beta_\dom\Delta_N}),~\dom=R,F,\\
	\Sigma_{P}^2  (n,\afix)& =\frac{{ \zeta_P^2 }\sigma_{R,n}^2}{2 \beta_R (\beta_R-\gamma)^2(\beta_R+\gamma)}\biggl\{\gamma+ 4 \beta_R \mathrm{e}^{-(\beta_R+\gamma)\Delta_N}-(\beta_R+\gamma) \mathrm{e}^{-2\beta_R \Delta_N}\nonumber\\
	&~~-\beta_R\left(2+ \mathrm{e}^{-2\gamma \Delta_N}\right)+\frac{\beta_R^2}{\gamma}\left( 1-\mathrm{e}^{-2\gamma \Delta_N}\right)\biggr\},\\
	\Sigma_{RP}^2   (n,\afix)& =-\frac{{ \zeta_P}\sigma_{R,n}^2}{2 \beta_R(\beta_R^2-\gamma^2)}\left(\beta_R-\gamma-2\beta_R\mathrm{e}^{-(\beta_R+\gamma)\Delta_N}+(\beta_R+\gamma)\mathrm{e}^{-2\beta_R\Delta_N}\right),\\
	\rho(n,\afix)& =\frac{\Sigma_{RP}^2  (n,\afix)}{\Sigma_{R}{ (n)}\Sigma_{P}  (n,\afix)}.
\end{align}
The function  $\mathcal{H}:\{0,\ldots,N-1\}\times\R^{\ell+1}\times \feasible\to \R$ is  linear in the state components $r,y$ and  given  by 
\begin{align}
	\mathcal{H}( n,r,y,\afix)
	&=\begin{cases} { h(n,r)}+
		\big(\Pamb+\frac{{ \zeta_F } }{\gamma}\big)(1-\mathrm{e}^{-\gamma\Delta_N}), &\afix=\Fuel\\[1ex]
		{ h(n,r)}+\big(\Pamb+\frac{{ \zeta_C} (p_{in}-\Pout)}{\gamma}\big)(1-\mathrm{e}^{-\gamma \Delta_N}), &\afix=\Charg\\[1ex]	
		{ h(n,r)}+\Pamb(1-\mathrm{e}^{-\gamma \Delta_N}), &\afix=\phantom{+}\Wait\\[1ex]
		{ h(n,r)}+\big(\Pamb-\frac{{ \zeta_D } \QinC}{\gamma}\big)(1-\mathrm{e}^{-\gamma\Delta_N})+\mathrm{e}^{-\gamma \Delta_N}\psi(n,y),&\afix=\Discharg,\\
		\mathrm{e}^{-\gamma \Delta_N}(p+P_{amb}(\mathrm{e}^{\gamma\Delta_N}-1)) & \afix=\Spill,
	\end{cases}
	\label{var-Ups}\\
	\text{with  	}~~  h(n,r)&= 
	\frac{{ \zeta_P }  r}{\beta_R-\gamma}\Big( \mathrm{e}^{-\beta_R \Delta_N}-\mathrm{e}^{-\gamma\Delta_N}\Big)-\frac{{ \zeta_P } \mu_{R,n}}{\gamma}(1-\mathrm{e}^{-\gamma\Delta_N}) \\
	\text{and 	}~~ \psi(n,y)
	&={ \zeta_D } \OutputOut \Big\{(\mathrm{e}^{(\gamma \mathds{I}_\ell+\overline{A})\Delta_N}-\mathds{I}_\ell)(\gamma \mathds{I}_\ell+\overline{A})^{-1}\Rstate+\nonumber
	\\	&\hspace*{4em}
	\big[(\mathrm{e}^{(\gamma \mathds{I}_\ell+\overline{A})\Delta_N}-\mathds{I}_\ell)(\gamma \mathds{I}_\ell+\overline{A})^{-1}-\frac{1}{\gamma}(\mathrm{e}^{\gamma \Delta_N}-1)\mathds{I}_\ell\big]\overline{A}^{-1}{B}g(n,-1)\Big\}.
\end{align}

\section{Proof of Lemma \ref{lem:RunningCost-closed}}
\label{Cost-closed-f1}
\begin{proof}
	Let  $k=0,\ldots,N-1, x=(r,f,p,y^\top)^\top\in\Statespace, \afix\in \actionspace$.  Then
	\begin{align}
		\overline{\runCDis}(k,x,\afix) =  \mathrm{e}^{\delta k\Delta_N}\runCDis(k,x,\afix) &= \mathrm{e}^{\delta t_k}\mathbb{E}_{k,x}\bigg[\int_{t_k}^{t_{k+1}}  \mathrm{e}^{-\delta s}\runCCont(X^{\overline \uprocess}(s),\afix) \mathrm{d} s \bigg]
		=\int_{t_k}^{t_{k+1}}  \mathrm{e}^{-\delta   (s-t_k)} \mathbb{E}_{k,x}\bigg[\runCCont(X^{\overline \uprocess}(s),\afix)  \bigg]\mathrm{d} s. 
	\end{align}
	For $\afix=+2,~\mu_{F,k}=\mu_F(t_k)$, $\Delta_N=t_{k+1}-t_k$,  and using the fact that $M_t^F=\displaystyle\int_{t_k}^{{  t}} \mathrm{e}^{-\beta_F (s-\tau)}dW_F(\tau)$ is a martingale, we have
	\begin{align}
		{  \overline{\runCDis}(k,x,\afix)}&=\int_{t_k}^{t_{k+1}}  \mathrm{e}^{-\delta   (s-t_k)} \mathbb{E}_{k,x}\big[\runCCont(X^{\overline \uprocess}(s),\afix)  \big]\mathrm{d} s
		=\price_F\int_{t_k}^{t_{k+1}}\mathrm{e}^{-\delta   (s-t_k)} \mathbb{E}\big[\mu_F(s)+F(s)~\mid F(t_k)=f\big]ds\\
		&=\price_F\int_{t_k}^{t_{k+1}}\mathrm{e}^{-\delta   (s-t_k)} \mathbb{E}\bigg[\mu_F(s)+f \mathrm{e}^{-\beta_F(s-t_k)}+\int_{t_k}^{s} \sigma_F(\tau) \mathrm{e}^{-\beta_F (s-\tau)}dW_F(\tau)\bigg]ds\\
		&=\price_F\int_{t_k}^{t_{k+1}}\mathrm{e}^{-\delta   (s-t_k)}\bigg(\mu_F(s)+f \mathrm{e}^{-\beta_F(s-t_k)})\bigg)ds=\price_F \bigg[\frac{\mu_{F,k}}{\delta}(1-\mathrm{e}^{-\delta \Delta_N})+\frac{f}{\delta+\beta_F}(1-\mathrm{e}^{-(\delta+\beta_F) \Delta_N}) \bigg]
	\end{align}
	For $\afix=+1$, we have 
	\begin{align}
		\overline{\runCDis}(k,x,\afix)&=\int_{t_k}^{t_{k+1}}  \mathrm{e}^{-\delta   (s-t_k)} \mathbb{E}_{k,x}\big[\runCCont(X^{\overline \uprocess}(s),\afix)  \big]\mathrm{d} s
		=\int_{t_k}^{t_{k+1}}\mathrm{e}^{-\delta   (s-t_k)}\mathbb{E}\big[{ \price_{HP}}(\Pin-\OutputOut Y(s))+{ \price_{P}}~\mid Y(t_k)=y\big]ds
		\\ &
		=\int_{t_k}^{t_{k+1}}({ \price_{HP}}\Pin+{ \price_{P}})\mathrm{e}^{-\delta   (s-t_k)}ds
		-{ \price_{HP}}\OutputOut \bigg\{\int_{t_k}^{t_{k+1}}\mathrm{e}^{-\delta   (s-t_k)}\Big(\mathrm{e}^{A(s-t_k)}\RState(t_k)+(\mathrm{e}^{A(s-t_k)}-\mathds{I}_\ell)A^{-1}Bg(t_k,+1)\Big)ds\bigg\}\\
		&=\frac{{ (\price_{HP}}\Pin+{ \price_{P}})}{\delta}(1-\mathrm{e}^{-\delta \Delta_N})-
		{ \price_{HP}} \OutputOut  \Big\{ \big(\mathrm{e}^{(A-\delta \mathds{I}_\ell)\Delta_N}-\mathds{I}_\ell\big)(A-\delta \mathds{I}_\ell)^{-1}y\nonumber\\&
		\hspace*{0.32\textwidth}+\Big[\big(\mathrm{e}^{(A-\delta \mathds{I}_\ell)\Delta_N}-\mathds{I}_\ell\big)(A-\delta \mathds{I}_\ell)^{-1}-\frac{1}{\delta}(1-\mathrm{e}^{-\delta \Delta_N})\mathds{I}_\ell\Big]A^{-1}Bg(t_n,+1)\Big\}.
	\end{align}
	Finally, for $\afix=-1$, we obtain
	\begin{align}
		\overline{\runCDis}(k,x,\afix)&=\int_{t_k}^{t_{k+1}}  \mathrm{e}^{-\delta   (s-t_k)} \mathbb{E}_{k,x}\Big[\runCCont(X^{\overline \uprocess}(s),\afix)  \Big]\mathrm{d} s
		={ \price_{P}}\int_{t_k}^{t_{k+1}}\mathrm{e}^{-\delta   (s-t_k)}ds =\frac{{ \price_{P}}}{\delta}(1-\mathrm{e}^{-\delta \Delta_N}).
	\end{align}
	For $\afix=0,-2$ the claim follows from $\runCCont(x,\afix)=0$.
\end{proof}

\end{document}